\documentclass[11pt]{amsart}

\usepackage{amsmath}
\usepackage{amsthm}
\usepackage{amsfonts}
\usepackage{amssymb}
\usepackage{eucal}

\usepackage[all]{xy}
\CompileMatrices

\DeclareFontFamily{OT1}{rsfs}{}
\DeclareFontShape{OT1}{rsfs}{n}{it}{<-> rsfs10}{}
\DeclareMathAlphabet{\mathscr}{OT1}{rsfs}{n}{it}

\DeclareMathOperator{\Spec}{Spec}

\theoremstyle{plain}
  \newtheorem{theorem}[subsubsection]{Theorem}
  \newtheorem{proposition}[subsubsection]{Proposition}
  \newtheorem{lemma}[subsubsection]{Lemma}
  \newtheorem{corollary}[subsubsection]{Corollary}

\theoremstyle{definition}
  \newtheorem{definition}[subsubsection]{Definition}

\theoremstyle{remark}
  \newtheorem{example}[subsubsection]{Example}
  \newtheorem{remark}[subsubsection]{Remark}

\usepackage{hyperref}
\usepackage{graphicx}
\usepackage[dvipsnames]{xcolor}
\usepackage{geometry}
\numberwithin{equation}{subsection}

\author{A. M. Masullo}
\title{Crystalline cohomology over general bases}

\begin{document}

\begin{abstract}
Building on ideas of Berthelot,
we develop a crystalline cohomology formalism over divided power rings $(A, I_0, \eta)$ for any ring $A$, 
allowing $\mathbf{Z}$-flat $A$. 
For a smooth $A$-scheme $Y$ and a closed
subscheme $X$ of $Y$ for which $\eta$ extends to $I_0 \mathscr{O}_X$, a (quasi-coherent) crystal $\mathscr{F}$
on $(X/A)_{\rm{cris}}$ is equivalent to a specific type of module with integrable $A$-linear
connection over a certain completion $D_{Y,\eta}(X)^{\wedge}$ (called ``pd-adic'') of the divided power envelope 
$D_{Y,\eta}(X)$ of $Y$ along $X$ (with divided power structure $\delta$)

Our main result, building on ideas of Bhatt and de Jong for $\mathbf{Z}/(p^e)$-schemes (where
pd-adic completion has no effect), 
is a natural isomorphism between ${\rm{R}}\Gamma((X/A)_{\rm{cris}}, \mathscr{F})$
and the Zariski hypercohomology of the pd-adically completed de Rham complex $\mathscr{F} \widehat{\otimes}
\widehat{\Omega}^*_{D_{Y,\eta}(X)^{\wedge}/A,\delta}$  arising
from the module with integrable connection over $D_{Y,\eta}(X)^{\wedge}$ associated to $\mathscr{F}$.
By a variant of the same methods, we obtain a representative of the complex
$\mathscr{F} \widehat{\otimes} \widehat{\Omega}^*_{D_{Y,\eta}(X)^{\wedge}/A,\delta}$
in the derived category of sheaves
of $A$-modules on $X$ in terms of a  \v{C}ech-Alexander construction.  

When $\mathscr{F}=\mathscr{O}_{X/A}$,  
our comparison theorem implies that in the derived category of sheaves of 
$A$-modules on $X$, the pd-adic completion of 
$\Omega^*_{D_{Y,\eta}(X)/A,\delta}$ functorially depends only on $X$. 
Over $\mathbf{Q}$-algebras $A$, so pd-adic completion becomes ideal-adic completion, this recovers a result of Hartshorne.

\end{abstract}

\maketitle

\tableofcontents

\section{Introduction}\label{intro}

\subsection{Motivation} 

In \cite[Appendice]{Ber}, Berthelot briefly sketched 
how to compare algebraic de Rham cohomology to crystalline cohomology when the base scheme is allowed to be arbitrary (not necessarily torsion).
An essential feature, as proposed in \cite[Appendice]{Ber}, is that the divided power thickenings $(B, J, \delta)$
corresponding to affine objects in the crystalline site are required to satisfy $m J^{[n]} = 0$ for some $m, n \ge 1$
(which may vary with $(B, J, \delta)$), or equivalently $N! J^{[N]} = 0$ for some $N \ge 1$.
Apart from some brief comments in \cite{faltings} where the site of such triples is called the ``extended site'', 
those ideas were not developed further (to the best of the author's knowledge), 
perhaps because for a long time the most striking applications of crystalline methods have been 
concentrated in $p^n$-torsion or $p$-adically complete settings.  

Recent activity in integral $p$-adic Hodge theory
(\cite{bms1}, \cite{bms2}, \cite{BS}), and the desire to understand the extent to which such methods can be adapted
to global situations without singling out a choice of preferred prime $p$, led us to find value in developing a crystalline theory over
general bases, with the expectation that such a theory will serve a useful role towards that goal.

We develop the formalism of crystalline cohomology over general base rings
(equipped with a divided power structure), avoiding the hypothesis that 
some power of a prime $p$ vanishes in the base ring. 
Our approach, inspired by \cite{BdJ}, 
enables us to obtain a comparison theorem with de Rham theory for the cohomology of a crystal over any
divided power ring (e.g., a ring equipped with the ideal $(0)$). When the crystal is the
structure sheaf, this recovers the comparison announced in \cite[Appendice]{Ber} with a different method of proof.

\subsection{Two completions}

Although $p$-adic completion
 at the sheaf-theoretic level as in \cite{BdJ} does not arise in our work, 
 the use of formal completion along closed subschemes in Hartshorne's work in \cite{hartshorne}
 on de Rham cohomology for singular varieties in characteristic 0 
 inspired us to try to adapt the $p$-adic completion techniques
 in \cite{BdJ} in a geometric manner.  The relevant tool for this purpose is a subtle geometric analogue of formal completion 
  introduced in \cite[Appendice]{Ber}: 
 completion along the descending system of ideals
 $n! J^{[n]}$, a process we call {\em pd-adic completion} (see Definition \ref{pdadicdef} and Remark \ref{pdadicrem}).

 It must be stressed that in the setting of $\mathbf{Z}/m\mathbf{Z}$-schemes
 (such as with $m = p^e$ for a prime $p$, as in \cite{BdJ}), pd-adic completion does nothing
 since $n! J^{[n]} = 0$ for all $n \ge m$.  So pd-adic completion is {\em irrelevant}
 in the crystalline theory over torsion rings; it is only perceived when allowing non-torsion base rings, as we do.
 Over $\mathbf{Q}$-algebras, pd-adic completion recovers the 
 familiar process of completion along a closed subscheme.  
For general divided power rings $B$, if $J = f B$ is principal then pd-adic completion along $J$ also 
 coincides with $J$-adic completion since $n! (fB)^{[n]} = f^n B$. However, (locally) principal $J$ are too special for
 our needs:  for a closed subsceme $X$ of a smooth scheme $Y$ over a divided power ring $(A, I_0, \eta)$ we'll 
 need to use pd-adic completion along the canonical divided power ideal
for the divided power envelope $D_{Y^{\nu+1},\eta}(X)$ of the fiber power $Y^{\nu+1}$ along the diagonally embedded $X$
 with all $\nu$, and those ideals are usually not locally principal when $\nu > 0$
 (even if $I_0$ is principal and $X = Y \bmod I_0$). 
 
 So although the way we use pd-adic completion is inspired by
 the role of $p$-adic completions in the theory over torsion rings as in \cite{BdJ}, it is logically
 unrelated.  Moreover, when working over a $\mathbf{Z}$-flat ring 
that isn't a $\mathbf{Q}$-algebra, the system of ideals $n!J^{[n]}$ is
 algebraically more complicated than completion along powers of an ideal when $J$ isn't locally principal.  
  
 \subsection{Results}

Our results in the affine setting, given in \S\ref{infpd}--\S\ref{affzarsite},
include an equivalence (in Proposition \ref{micequiv}) over any divided power ring
$(A, I_0, \eta)$, between:
\begin{itemize}
\item[(i)] the category of crystals of quasi-coherent modules
(abbreviated to ``quasi-coherent crystals'') on a closed subscheme $\Spec(R)$ 
of a smooth affine $A$-scheme $\Spec(P)$ admitting \'etale coordinates over $A$,
\item[(ii)] a specific category of topological 
modules (over the pd-adic completion of the divided power envelope of $\Spec(P)$ along
$\Spec(R)$ respecting $\eta$) equipped with an integrable connection satisfying certain
divisibility conditions over $\mathbf{Z}$. 
\end{itemize}
The divisibility conditions in (ii) are a globalization of a ``topologically quasi-nilpotent'' condition on connections 
arising from crystals in the $p$-adic theory.  In the special case 
$P=R$ (as we can use when $R$ is $A$-smooth and admits \'etale coordinates), such crystals 
thereby amount to $R$-modules equipped 
with an integrable connection satisfying specific divisibility conditions (automatic when $A$ is a $\mathbf{Q}$-algebra).

\begin{remark}
By work of Katz (see \S\ref{MICsec}, and in particular Example \ref{katzex}), a more refined and explicit version of 
the divisibility conditions on connections arising from crystals
 is satisfied over a dense open in the base for the Gauss--Manin connection on each relative de Rham cohomology sheaf
${\rm{R}}^i f_*(\Omega^*_{X/S})$ for a smooth proper morphism $f:X \to S$ with $S$ flat over $\mathbf{Z}$.
\end{remark}

Consider a closed subscheme $X$ of a smooth scheme $Y$
(e.g, a quasi-projective scheme $X$) over a general base ring $A$ equipped with a divided power structure
$(I_0, \eta)$ such that $\eta$ extends to $I_0 \mathscr{O}_X$, as is automatic when $I_0$ is locally
principal or $X$ is $A$-flat or $I_0 \mathscr{O}_X = 0$ (i.e., $X$ is an $A/I_0$-scheme).  We prove 
a global comparison isomorphism (in Theorem \ref{hartshornethm}) computing the crystalline cohomology of a 
quasi-coherent crystal $\mathscr{F}$ on $X$ relative to $(A, I_0, \eta)$ as Zariski hypercohomology
of a pd-adically completed de Rham complex 
$\mathscr{F} \widehat{\otimes} \widehat{\Omega}^*_{D_{Y,\eta}(X)^{\wedge}/A,\delta}$ (where $\delta$ is the divided power
structure).
By replacing $X$ with varying opens $U \subset X$, this ``sheafifies'' to an isomorphism
\begin{equation}\label{Ruisom}
{\rm{R}}u_{X/A,\ast}(\mathscr{F}) \simeq 
\mathscr{F} \widehat{\otimes} \widehat{\Omega}^*_{D_{Y,\eta}(X)^{\wedge}/A,\delta}
\end{equation} 
in the derived $\infty$-category of sheaves of $A$-modules on $X$ (see Proposition \ref{localmain}).  

The isomorphism (\ref{Ruisom}) shows the right side depends functorially only on $X$;
for $\mathbf{Q}$-algebras $A$ and $\mathscr{F}=\mathscr{O}_{X/A}$ this recovers
Hartshorne's result in (the proof of) \cite[Ch.\,II, Thm.\,1.4]{hartshorne}
that in the derived category of sheaves of $A$-modules on $X$, the formal completion
of $\Omega^*_{Y/A}$ along $X$ is independent of $Y$. 
We now have the same conclusion over any ring $A$ (the choice $I_0 = 0$ is permitted) by using pd-adic completion
as the correct generalization of usual completions over $\mathbf{Q}$-algebras.
A sheafified version of the method of proof of our global comparison isomorphism
yields yet another description of the right side of (\ref{Ruisom}) avoiding the mention of
differential forms, using a \v{C}ech-Alexander construction (which however involves $Y$); see Remark \ref{Fbullet}.  

In the case of $A$-smooth $X$, using $Y=X$ gives a site-theoretic interpretation of de Rham hypercohomology
for smooth $A$-schemes (announced in \cite[Appendice]{Ber}, recovering 
what Grothendieck did over $\mathbf{Q}$-algebras \cite[Thm.\,4.1]{dix})
as an instance of a comparison result with coefficients in a 
quasi-coherent crystal.  We emphasize that even though this comparison
theorem for $A$-smooth $X$ using the choice $Y=X$ doesn't involve pd-adic completions
in its statement, such completions are an essential ingredient in the proof
because (after reducing to the affine setting) pd-adic completions are used 
along the canonical ideal for $D_{Y^{\nu+1},\eta}(X)$ with all $\nu$, 
and those ideals are generally not locally principal for $\nu > 0$ even when $Y=X$. 

In the special case that $A = W$ is a complete discrete valuation ring with mixed characteristic $(0,p)$
and uniformizer $p$, for a smooth proper $W$-scheme  $Y$ with special fiber $Y_0$
we have $D_Y(Y_0) = Y$ and the pd-adic topology on
this divided power envelope is the $p$-adic topology.  
Our global comparison theorem in this context yields 
${\rm{H}}^i((Y_0/W)_{\rm{cris}}, \mathscr{O}_{Y_0/W}) \simeq {\rm{H}}^i_{\rm{dR}}(\widehat{Y}/W)$
without passing through crystalline cohomology 
over each $W_n = W/(p^n)$; see Remark \ref{dRcomp}. Here too, the statement does not mention pd-adic
completion along non-principal ideals but they are present in the proof via the ideals for $Y_0$
diagonally embedded in the higher fiber powers
of $Y$ over $W$. 

\subsection*{Acknowledgements}  I am grateful to the Stanford University Math department, the American Institute of Mathematics, and Ashok Vaish for financial support during the work on this paper, and to Brian Conrad for helpful discussions and expository advice.

\section{Infinitesimal divided power envelopes and associated de Rham complexes}\label{infpd}

\subsection{Initial algebraic constructions} 

Let $(A_0, I_0, \eta)$ be a divided power ring.  
Let $A$ be an $A_0$-algebra and $I$ an ideal in $A$; we don't assume
$I_0 A \subset I$.  Recall from the theory of divided power structures on ideals in rings
(see \cite[Ch.\,I]{Ber}, \cite[Ch.\,3]{bogus}, \cite[Tags 09PD, 07H7]{SP}, 
for example) that there is an initial map of pairs $(A,I) \to (D_{A,\eta}(I), I_D)$ for which $I_D$ is equipped with
divided powers compatible with $\eta$.   
We call $D_{A,\eta}(I)$ the {\em divided power envelope} of $A$ with respect to $I$ compatible with $\eta$, 
and $I_D$ is generated
by the divided powers of the images of the elements of $I$. 

For $n \ge 1$, let  $J^{[n]}$ denote the {\em $n$th divided power} of an ideal $J$ equipped with a divided power structure $\delta$
(i.e.,  $J^{[n]}$ is generated by products $\delta_{i_1}(x_1) \cdots \delta_{i_k}(x_k)$ for $x_1, \dots, x_k \in J$
and integers $i_1, \dots, i_k \ge 0$ satisfying $\sum_j i_j \ge n$).  
If $m J^{[n]} = 0$ with $m, n \ge 1$ then $N! J^{[N]} = 0$ for $N = \max(m,n)$, so 
every $x \in J$ satisfies $x^N=0$. 

For $m, n \ge 1$, define
$$D^{m,n}_{A,\eta}(I) := D_{A,\eta}(I)/m I_D^{[n]};$$
this is initial among divided power algebras $(B, J, \delta)$
over $(A, I)$ compatible with $\eta$ such that $m J^{[n]} = 0$.  (The notation $D^{m,n}_A(I)$ in \cite[Ch.\,I, (3.3.1)]{Ber} is defined
by killing $m I_D^{[n+1]}$ and assuming $I_0 = 0$.)  Note that $D^{1,1}_{A,\eta}(I) = D_{A,\eta}(I)/I_D = A/I$ when 
$\eta$ extends to $I_0 (A/I)$ (such as when $I_0$ is principal or $A/I$ is $A_0$-flat or $I_0(A/I)=0$); 
see \cite[3.15, 3.20(4)]{bogus}. 

For any $(A_0, I_0, \eta)$, 
$A_0 \langle t_1, \dots, t_k \rangle$ denotes the divided power envelope of $A_0[t_1,\dots, t_k]$ along $(t_1, \dots, t_k)$
compatible with $\eta$.  By \cite[3.20(6)]{bogus}, this agrees with the divided power envelope
for the case $I_0 = 0$.  For later use, we record the following key calculation. 

\begin{proposition}\label{dqt}
Let $(R, J, \delta)$ be a divided power ring, 
$\Spec(P) \to \Spec(R)$ a smooth map of affine schemes, and $\Spec(R) \hookrightarrow \Spec(P)$ a closed immersion over $R$ 
with corresponding ideal denoted $I$.  Let $I_D$ be the canonical ideal in $D_{P,\eta}(I)$ and 
$\langle t_1, \dots, t_d \rangle_+ \subset R\langle t_1, \dots, t_d \rangle$ be the divided power ideal pd-generated by the $t_j$'s. 

Suppose the locally free module $I/I^2$ over $P/I=R$ is free, 
and let $\{f_1, \dots, f_d\}$ be a collection of elements of $I$ that represent an $R$-basis of $I/I^2$.
Then for all $m, n \ge 1$ there is a unique isomorphism
of divided power algebras
$$R \langle t_1, \dots, t_d \rangle/m \langle t_1, \dots, t_d \rangle_+^{[n]} \simeq D_{P,\eta}^{m,n}(I)$$
over $(R, J, \delta)$ satisfying $t_i \mapsto f_i$ for all $i$.
\end{proposition}

\begin{remark}\label{D0}
By \cite[3.20(6)]{bogus}, we have naturally $D_{P,0}(I) = D_{P,\eta}(I)$ since $P/I = R$.
\end{remark}

\begin{remark}
If $m \cdot 1_R = 0$ then this result says 
$R \langle t_1, \dots, t_d \rangle \simeq D_P(I)$. The crystalline site defined later over a general divided power ring, 
which imposes the condition on divided power thickenings of affine objects that the divided power ideal $K$ satisfies $mK^{[n]}=0$ for some $m, n \ge 1$,
coincides with the existing notion over $\mathbf{Z}/(p^e)$-algebras respecting the usual divided power
structure on $(p)$.  Another crucial place
where we will use ``$m K^{[n]} =0$'' on affine objects in the crystalline site will
be to define the (integrable) connection associated to a quasi-coherent crystal (see Remark \ref{B(1)pd}).
\end{remark}

\begin{proof}
By Remark \ref{D0}, we may and do assume $I_0 = 0$. 
We first treat the case when the $f_j$'s generate $I$, and then we reduce the general case to that
case via a localization argument.
The smoothness of $P$ over $R$ implies two properties:  (i) each quotient map $P/I^N \to R$ admits an $R$-algebra section (by the infinitesimal criterion), 
(ii) ${\rm{Sym}}_R(I/I^2) \to {\rm{Gr}}_I(P) := \oplus_{r \ge 0} I^r/I^{r+1}$ is an isomorphism (by the existence Zariski-locally near $V(I)$ of
a regular sequence generating $I$; see \cite[Ch.\,IV, Cor.\,2.4]{fl} and bootstrap from the case of noetherian $R$
via \cite[Cor.\,to Thm.\,22.5]{crt}).  These two properties are the hypotheses in
\cite[Ch.\,I, Prop.\,3.5.1]{Ber}, and since we are first assuming $I$ is generated by the elements $f_j$ that represent an $R$-basis of $I/I^2$,
the method of proof of \cite[Ch.\,I, Prop.\,3.5.1(i)]{Ber} provides the desired isomorphism. 

In general, there is an open neighborhood $U$ of $V(I)$ in $\Spec(P)$ such that
the $f_j$'s generate $I$ over $U$; we can shrink $U$ to be quasi-compact. 
To find an affine open inside $U$ containing $V(I)$, we first consider the 
``Zariskification'' \cite[Ch.\,0, \S7.3(b)]{fk} (i.e., localize at $1+I$). 
Namely, for $B := A_{1+I}$ it is easy to see that $1 + IB \subset B^{\times}$, so $IB$ belongs to the Jacobson
radical of $B$ and hence the only open subset of $\Spec(B)$ that contains $V(IB)$ is the entire space.
Thus, the preimage of $U$ under $\Spec(B) \to \Spec(A)$ is the entirety of $\Spec(B)$, so by quasi-compactness of $U$
and writing $B$ as a direct limit of $A$-algebras of the form $A_g$ for $g \in 1+I$ we get such a $g$ so that 
$U$ has full preimage under $\Spec(A_g) \to \Spec(A)$.  This says $\Spec(A_g) \subset U$ inside $\Spec(A)$,
and $V(I) \subset \Spec(A_g)$ since $g \in 1 + I$. 
Thus, due to how we chose $U$, clearly $IP_g$ is generated by the $f_j$'s.  Since the composite map $R \to P \to P_g$ is smooth, 
the settled special case when the $f_j$'s generate $I$ applies to $P_g \twoheadrightarrow P_g/IP_g = P/I = R$
(first equality since $g \equiv 1 \bmod I$).

It remains to show that the natural map $D_P^{m,n}(I) \to D_{P_g}^{m,n}(IP_g)$ is an isomorphism for all $m, n \ge 1$. 
Since $I$ is finitely generated (as $P$ is finitely presented over $R$), if it has $e$ generators 
then $I^{mne}$ has vanishing image in $D_P^{m,n}(I)$ (see the proof of \cite[Ch.\,I, Prop.\,3.1.4]{Ber}).  Hence,
for $\overline{P} = P/I^{mnd}$ and $\overline{I} = I/I^{mne}$, 
by universality considerations the natural map $D_P^{m,n}(I) \rightarrow D_{\overline{P}}^{m,n}(\overline{I})$ is an isomorphism.
This also applies $(P_g, IP_g)$ in place of $(P, I)$, so it remains to note that the natural map
$P/I^N \to  P_g/I^N P_g$ is an isomorphism for all $N \ge 1$ (since $g \equiv 1 \bmod I$).
\end{proof}

\begin{definition}\label{pdadicdef} Let $(A_0, I_0, \eta)$ be a divided
power ring, and $I$ an ideal in an $A_0$-algebra $A$. Using the decreasing sequence of ideals
$N! I_D^{[N]}$ in $D_{A,\eta}(I)$  as $N$ grows, for any $D_{A,\eta}(I)$-module $M$ the inverse limit
$$\widehat{M} := \varprojlim M/(N! I_D^{[N]})M$$
makes sense; we call it the {\em pd-adic completion} of $M$. 
\end{definition}

Note that when $A$ is a $\mathbf{Z}/(m)$-algebra for some $m \ge 1$
then $\widehat{M} = M$.  In particular, pd-adic completions have nothing to do with
$p$-adic completions that one generally encounters for $p$-adic crystalline cohomology.
However, the role of pd-adic completions later on will be somewhat analogous
to that of $p$-adic completions in the usual development of $p$-adic crystalline methods.

In our main work, we will be using a fixed base ring or scheme that is arbitrary. 
By taking the divided power base ring to be
$\Spec(W(k)/(p^e))$ for a perfect field $k$ of characteristic $p$ and any $e \ge 1$, 
equipped with the usual divided powers on $(p)/(p^e)$, we will recover 
constructions from traditional crystalline cohomology.  But to avoid any misunderstanding, we stress that the traditional process of forming
$p$-adic completions is a step of rather different nature, with no logical counterpart in what we will be doing.

\begin{remark}\label{pdadicrem}
Clearly $\widehat{M}$ is a module over the pd-adic completion $\widehat{D}_{A,\eta}(I) := \varprojlim D_{A,\eta}(I)/N! I_D^{[N]}$ of $D_{A,\eta}(I)$, 
and $\widehat{D}_{A,\eta}(I)$ has a natural divided power structure on $\widehat{I}_D := \varprojlim I_D/N! I_D^{[N]}$, 
but we do {\em not} claim that $\widehat{D}_{A,\eta}(I)$ is complete for the topology defined by the decreasing collection of ideals
$N! \widehat{I}_D^{[N]}$; that is, 
$$\widehat{D}_{A,\eta}(I) \to \varprojlim \widehat{D}_{A,\eta}(I)/N! \widehat{I}_D^{[N]}$$
may not be an isomorphism since it is unlikely that the natural surjective map $\widehat{D}_{A,\eta}(I)/N! \widehat{I}_D^{[N]} \to D_{A,\eta}(I)/N! I_D^{[N]}$
is injective (and it is unlikely that the natural surjective map $\widehat{M}/N! \widehat{I}_D^{[N]} \widehat{M} \to M/N! _D^{[N]}M$
is injective).  So for us, ``pd-adic completion'' is just terminology referring to $\widehat{D}_{A,\eta}(I)$ (or $\widehat{M}$ above more generally).

One difficulty is that (even if $I$ is finitely generated) $I_D$ is not finitely generated, and probably $\widehat{I}_D$ is larger than
$I_D \widehat{D}_{A,\eta}(I)$.  Fortunately this nuisance will not be a problem in what follows.  We will need
to use $\widehat{D}_{A,\eta}(I)$ when relating crystals to modules with an integrable (divided power) connection in \S\ref{connsec}, but that
will not require pd-adic completeness of $(\widehat{D}_{A,\eta}(I), \widehat{I}_D)$. 
\end{remark}

\subsection{The category ${\rm{Cris}}(R/A)$ and divided power differentials}

Let $(A, I_0, \eta)$ be a divided power ring, 
and $R$ an $A$-algebra such that $\eta$ extends to $I_0 R$.
This final hypothesis holds when $R$ is $A$-flat or $I_0$ is principal \cite[3.15]{bogus}
or $I_0 R = 0$ (i.e., $R$ is an $A/I_0$-algebra).  

\begin{definition} The category ${\rm{Cris}}(R/(A, I_0, \eta))$ 
(usually abbreviated to ${\rm{Cris}}(R/A)$) consists of divided power rings
$(B, I, \delta)$ over $A$ compatible with $\eta$ such that:
\begin{itemize}
\item[(i)] there is given a closed immersion $\Spec(R) \hookrightarrow \Spec(B)$ over $A$ cut out by the ideal $I$,
\item[(ii)] $m I^{[n]} = 0$ for some $m, n \ge 1$.
\end{itemize}
Morphisms of such objects are defined in the evident manner (over $A$).
\end{definition}

The key point in the preceding definition is that we require $m I^{[n]} = 0$ for some $m, n \ge 1$. 
Note that if $M \ge m$ and $N \ge n$ then $M! I^{[N]} \subset m I^{[n]} = 0$.  Hence, the requirement on $(B, I, \delta)$ that there exists 
$m, n \ge 1$ such that $m I^{[n]} = 0$ can be checked  Zariski-locally on ${\rm{Spec}}(B)$. 
For $N = \max(m, n)$ we
have $N! I^{[N]} = 0$, so $x^N = 0$ for all $x \in I$ and hence $\Spec(B)$ has the same underlying topological space as $\Spec(R)$.

Now assume $R$ is finitely generated over $A$. We'll make many objects in
${\rm{Cris}}(R/A)$.  Choose 
a surjection $\pi:P \to R$ from a smooth $A$-algebra $P$ (e.g., a polynomial ring over $A$).
For $J = \ker(\pi)$, if $m, n \ge 1$ then $(B, I) := (D_{P,\eta}(J)/m J_D^{[n]}, J_D/m J_D^{[n]})$ is a divided power $A$-algebra 
such that 
$m I^{[n]} = 0$.  Since $\eta$ extends to $I_0 R$ by hypothesis, we
have $D_{P,\eta}(J)/J_D \simeq P/J = R$ by \cite[3.20(4)]{bogus}. 

These objects in ${\rm{Cris}}(R/A)$ arising from such smooth $P$ will be useful in what follows. 
(We could consider polynomial rings in a possibly infinite set of variables to avoid assuming $R$ is of finite type,
but prefer to focus on $R$ of finite type over $A$. Such large polynomial rings play a role in
developing a derived version of the theory presented here; such a theory is not discussed in this
document.)

\begin{remark}\label{crismap} 
Consider any $(B, I, \delta) \in {\rm{Cris}}(R/A)$, so
$m I^{[n]} = 0$ for some $m, n \ge 1$.  Thus, for $N = \max(m, n)$
we have $N!I^{[N]} = 0$, so
$x^N = 0$ for all $x \in I$.  Let $\pi: P \twoheadrightarrow R$ be an $A$-algebra
surjection from a smooth $A$-algebra $P$ (such as a polynomial ring), with $J = \ker \pi$.  
We claim that there exists a map of $A$-algebras $f:P \to B$ respecting their surjections to $R$,
so this map carries $J$ into $I$ and hence uniquely factors through a map of divided power
$A$-algebras $(D_{P,\eta}(J), J_D) \to (B, I)$.  Since $N! I^{[N]} = 0$, it follows that
we get a map of divided power $A$-algebras
$$(D_{P,\eta}(J)/N! J_D^{[N]}, J_D/N! J_D^{[N]}) \to (B,I);$$
note that the source is an object in ${\rm{Cris}}(R/A)$. 

To build $f$, the idea is that since $P$ is $A$-smooth and the ideal $I$ of
$\Spec(R) \hookrightarrow \Spec(B)$ has all elements with vanishing $N$th power,
a noetherian approximation argument should yield the existence of $f$. 
However, usually $I$ is not finitely generated (e.g., for examples arising from quotients of divided power envelopes), so 
there may be no power $I^e$ which vanishes, and hence we can't directly apply the usual infinitesimal smoothness criterion.  We
circumvent that difficulty via noetherian approximation,
as follows. Let $A_0 \subset A$ be a large enough finitely generated $\mathbf{Z}$-subalgebra
so that there exists a smooth $A_0$-algebra $P_0$ for which $A \otimes_{A_0} P_0 = P$. Let $R_0 \subset R$ be the image of
$P_0 \to P \twoheadrightarrow R$ and let $B_0 \subset B$ be a finitely generated $\mathbf{Z}$-subalgebra mapping
onto $R_0$, so $B_0 \to R_0$ has kernel $I_0$ of finite type. Since $I_0 \subset I$, we have $I_0^e = 0$ for some large $e$.
Then $P_0 \to R_0$ over $A_0$ lifts to an $A_0$-algebra map $P_0 \to B_0$, so the composite map $P_0 \to B_0 \to B$
yields an $A$-algebra map $f:P = A \otimes_{A_0} P_0 \to B$ respecting surjections to $R$ as desired.
\end{remark}

For the moment, we do not endow ${\rm{Cris}}(R/A)$ with a Grothendieck topology; we consider it just as a category for now (not yet as a site). 
The functor $\mathscr{O}_{R/A}^{\rm{cris}}$ on ${\rm{Cris}}(R/A)$ is defined by
$$\mathscr{O}_{R/A}^{\rm{cris}}(B,I, \delta) = B.$$
For the ideal $J := \ker(B \otimes_A B \twoheadrightarrow B \twoheadrightarrow R)$,
the quotient map $B \otimes_A B \to B$ carries $J$ onto $I$, so we get a canonical surjective map of divided power $A$-algebras 
$$B(1) := D_{B \otimes_A B}(J) \to B.$$  Hence, its kernel $I(1)$ is equipped with divided powers arising from those on $J_D$, so
we can define the $\mathscr{O}_{R/S}^{\rm{cris}}$-module $\Omega^1_{R/A, {\rm{cris}}}$
via
$$\Omega^1_{R/A, {\rm{cris}}}(B,I,\delta) = I(1)/I(1)^{[2]} =: \Omega^1_{B/A, \delta}$$
(keep in mind that we require $\delta$ to be compatible with $\eta$ on $I_0 \subset A$). 
This is an $\mathscr{O}_{R/A}^{\rm{cris}}$-module since $I(1)^2 \subset I(1)^{[2]}$ and $I(1)/I(1)^2$ is a module over $B(1)/I(1) = B$.

\begin{remark}
The notation $\Omega^*_{B/A,\delta}$ is justified because
the natural map ${\rm{d}}:B \to \Omega^1_{B/A, \delta}$ defined by $b \mapsto b \otimes 1 - 1 \otimes b \bmod I(1)^{[2]}$ 
is initial among $A$-linear divided power derivations from $B$ into a $B$-module: see \cite[Tag 07IW]{SP}
(and \cite[Tag 07J1]{SP}).  
\end{remark}

\begin{remark}\label{B(1)pd}
By definition of ${\rm{Cris}}(R/A)$, for every object $(B, I, \delta)$ there exists some $m, n \ge 1$ such that $m I^{[n]} = 0$.  But 
there is no reason for the divided power $A$-algebra $(B(1), J_D)$ to belong to ${\rm{Cris}}(R/A)$
since there are no evident integers $M, N \ge 1$ for which $M J_D^{[N]} = 0$ in $B(1)$.  
However, for the divided power algebra $(B(1)/I(1)^{[2]}, J_D/I(1)^{[2]})$ compatible with
$\eta$ on $I_0 \subset A$, we'll now find suitable such integers
(and the argument will show that we cannot generally take $M=1$ if $m=1$, so allowing general $m \ge 1$ in the definition
of ${\rm{Cris}}(R/A)$ is crucial for what follows). 

Since $J_D/I(1) = J$ via $B(1)/I(1) \simeq B$, we have $m J_D^{[n]} \subset I(1)$
since $m I^{[n]} = 0$.  Hence, in the quotient $B(1)/I(1)^{[2]}$ where the image of $I(1)^2$ vanishes, 
for all $x \in J$ the square of $m x^{[n]}$ vanishes, so $(2 m^2) (x^{[n]})^{[2]} = (m x^{[n]})^2$ has vanishing 
image.  

Since $(x^{[n]})^{[2]} = a_n x^{[2n]}$ for an integer $a_n$ depending only on $n$, 
we are reduced to showing that if $(\overline{B}, \overline{I})$ is a divided power ring
and there exist integers $m, n \ge 1$ such that $m x^{[n]} = 0$ for all $x \in \overline{I}$
then $M \overline{I}^{[N]} = 0$ for some $M, N \ge 1$.  
Consider an element
$$\xi := x_1^{[i_1]} \cdots x_r^{[i_r]} \in \overline{I}^{[n]}$$
with $\sum i_j \ge n$ and all $i_j \ge 1$.
If some $i_j \ge n$ then $m \xi = 0$ since $m x_j^{[i_j]} \in m \overline{I}^{[n]} = 0$.
If there are distinct $j, j'$ such that $n/2 \le i_j, i_{j'} < n$ then 
$m \xi = 0$ since $m x_j^{[i_j]} x_{j'}^{[i_{j'}]} \in m \overline{I}^{[n]} = 0$.
In general, if $1 \le k < \log_2(n)$ and there are distinct $j_1, \dots, j_{2^k}$ such that $n/2^k \le i_{j_1}, \dots, i_{j_{2^k}} < n/2^{k-1}$
then $m \xi = 0$.  

Hence, we can assume $i_j < n$ for all $j$,
and that for each $1 \le  k < \log_2(n)$ there are at most $2^k - 1$ values of $j$
such that $n/2^k \le i_j < n/2^{k-1}$.  But every positive integer less than $n$
lies in such an interval $[n/2^k, n/2^{k-1})$, so we are left in the situation where the number $r$ of such $j$'s is at most
$$\sum_{1 \le k < \log_2(n)} (2^k - 1) < 2(n-1).$$
In other words, either $m \xi = 0$ or else $r$ is bounded in terms of $n$ and every $i_j < n$.
But in the latter case there is an integer $C_n$ depending only on $n$ such that $C_n \xi$ is an integer multiple of
$x_1^{i_1} \cdots x_r^{i_r}$ with $r \le 2(n-1)$ and $\sum i_j <  2n^2 \log_2(n)$.  Thus, 
for $N \ge 2n^2 \log_2(n)$ such latter cases cannot arise when considering $\overline{I}^{[N]}$, so
$m \overline{I}^{[N]} = 0$.
\end{remark}

Note that both composite $A$-algebra maps $B \rightrightarrows B \otimes_A B \to B(1) := D_{B \otimes_A B,\eta}(J)$  are sections
to $B(1) \twoheadrightarrow B(1)/I(1) = B$, so this makes the $A$-module exact sequence
\begin{equation}\label{omega1B}
0 \to \Omega^1_{B/A, \delta} \to B(1)/I(1)^{[2]} \to B \to 0
\end{equation} 
an exact sequence of $B$-modules.

\medskip\medskip

\section{Crystals and connections}\label{connsec}

\subsection{Integrable connections and complexes}

In this section, we review some basic formalism relating crystals
to integrable connections.  It proceeds similarly to the standard $p$-adic crystalline theory.
When we discuss the equivalence of categories relating crystals to modules with a suitable integrable connection in \S\ref{MICsec}, 
the use of pd-adic completions  will involve some issues that do not arise with $p$-adic completions (see Remark \ref{pdadicrem}).

Let $(A, I_0, \eta)$ be a divided power ring and $R$ an $A$-algebra of finite type
such that $\eta$ extends to $I_0 R$ (e.g., $R$ is $A$-flat or $I_0$ is principal or $I_0 R = 0$). 

\begin{definition}\label{crystal} An $\mathscr{O}_{R/A}^{\rm{cris}}$-module $\mathscr{F}$
is a {\em crystal} if for every morphism $(B, I, \delta) \to (B', I', \delta')$ in ${\rm{Cris}}(R/A)$
the natural map of $B'$-modules $$B' \otimes_B \mathscr{F}(B, I, \delta) \to \mathscr{F}(B', I', \delta')$$
is an isomorphism. (We may write $\mathscr{F}_B$ rather than $\mathscr{F}(B, I, \delta)$ to simplify notation.) 
\end{definition}

The most basic example of a crystal is $\mathscr{F} = \mathscr{O}_{R/A}^{\rm{cris}}$, though for any crystal $\mathscr{F}$ whatsoever
we have via $${\rm{pr}}_1, {\rm{pr}}_2: \Spec(B(1)/I(1)^{[2]}) \rightrightarrows \Spec(B)$$ a pair of $B(1)/I(1)^{[2]}$-linear isomorphisms
$${\rm{pr}}_1^{\ast}(\mathscr{F}_B) \simeq \mathscr{F}_{B(1)/I(1)^{[2]}} \simeq {\rm{pr}}_2^{\ast}(\mathscr{F}_B).$$
Thus, for any $s \in \mathscr{F}_B$ we can form the difference 
$$\nabla(s) := {\rm{pr}}_1^{\ast}(s) - {\rm{pr}}_2^{\ast}(s) \in \mathscr{F}_{B(1)/I(1)^{[2]}}.$$
Using the quotient map $B(1)/I(1)^{[2]} \twoheadrightarrow B(1)/I(1) = B$, the resulting map
$$\mathscr{F}_{B(1)/I(1)^{[2]}} \to \mathscr{F}_B$$
carries both ${\rm{pr}}_1^{\ast}(s)$ and ${\rm{pr}}_2^{\ast}(s)$ to $s$, so
$$\nabla(s) \in \ker(\mathscr{F}_{B(1)/I(1)^{[2]}} \to \mathscr{F}_B).$$

The crystal property provides an identification
$$(B(1)/I(1)^{[2]}) \otimes_B \mathscr{F}_B \simeq \mathscr{F}_{B(1)/I(1)^{[2]}}$$
(using $B \to B(1)/I(1)^{[2]}$ defined by $B \to B(1)$ via $b \mapsto 1 \otimes b$), and
the split-exact $B$-linear sequence (\ref{omega1B}) provides an isomorphism
$$\ker(\mathscr{F}_{B(1)/I(1)^{[2]}} \to \mathscr{F}_B) \simeq \mathscr{F}_B \otimes_B \Omega^1_{B/A, \delta},$$ 
so we obtain a natural $A$-linear map
$$\nabla: \mathscr{F} \to \mathscr{F} \otimes_{\mathscr{O}_{R/A}^{\rm{cris}}} \Omega^1_{R/A, {\rm{cris}}}.$$

The purely algebraic computations in \cite[Tag 07J5, 07J6]{SP} (which never use the running hypothesis from \cite[Tag 07MF]{SP}
that some prime is locally nilpotent) show that $\nabla$ is a ``divided power'' connection
(i.e., $\nabla(bs) = s \otimes {\rm{d}}b + b \nabla(s)$ for $s \in \mathscr{F}_B$ and $b \in B$ for $(B, J, \delta) \in {\rm{Cris}}(R/A)$).
Letting $\Omega^i_{R/A, {\rm{cris}}}$ denote the $i$th exterior power of $\Omega^1_{R/A, {\rm{cris}}}$, 
the purely algebraic computations with usual K\"{a}hler differentials as in \cite[Tag 07I0]{SP}
work verbatim using divided power differentials and so $\nabla$ uniquely extends to a collection of $A$-linear maps
$$\nabla: \mathscr{F} \otimes_{\mathscr{O}_{R/A}^{\rm{cris}}} \Omega^i_{R/A, {\rm{cris}}} \to
\mathscr{F} \otimes_{\mathscr{O}_{R/A}^{\rm{cris}}} \Omega^{i+1}_{R/A, {\rm{cris}}}$$
satisfying by usual Leibnitz rule. The resulting diagram
\begin{equation}\label{Fcris}
 \mathscr{F} \stackrel{\nabla}{\to} \mathscr{F} \otimes \Omega^1_{R/A, {\rm{cris}}} \stackrel{\nabla}{\to}
\mathscr{F} \otimes \Omega^2_{R/A, {\rm{cris}}} \stackrel{\nabla}{\to} \dots
\end{equation}
is a complex due to \cite[Tag 07J6]{SP}.

In \S\ref{MICsec} we will provide a healthy supply of crystals via module-theoretic data, by adapting the procedures
in \cite[Tag 07J7]{SP} (which is for pd-thickenings on which a power of a fixed prime $p$ vanishes),
replacing $p$-adic completions with pd-adic completions.  
The role of $p$-adic completions in \cite[Tag 07J7]{SP} serves
a quite different purpose from our use of pd-adic completions, but nonetheless
the algebraic formalism will work similarly.  
Some care will be needed for our version since for a smooth
$A$-algebra $P$ and ideal $J$ in $P$, we will need to work with the 
$$\widehat{D} := \varprojlim_{n \ge 1} D_{P,\eta}(J)/n! J_D^{[n]}$$
but (as noted in Remark \ref{pdadicrem}) for the divided power ideal $\widehat{J} := \varprojlim_{n \ge 1} J_D/n! J_D^{[n]}$
generally $\widehat{D} \to \varprojlim \widehat{D}/n! \widehat{J}^{[n]}$ is likely not an isomorphism.

\subsection{Equivalence between crystals and modules with integrable connection}\label{MICsec}

Consider a divided power ring $(A, I_0, \eta)$, a finite type $A$-algebra $R$ such that $\eta$
extends to $I_0 R$ (as holds if $R$ is $A$-flat or $I_0$ is principal),
and a smooth $A$-algebra $P$ equipped with a surjection $P \twoheadrightarrow R$ having kernel ideal $J$, 
and an \'etale map $\Spec(P) \to \mathbf{A}^d_A$ over $A$ (e.g., given $R$, we can choose such $P$ that is a polynomial algebra over $A$).
Define $D_n = D_{P,\eta}(J)/n! J_D^{[n]}$ and let $\widehat{D} = \varprojlim D_n$ be the pd-adic completion of $D_{P,\eta}(J)$, 
equipped with its natural topology.  Let $\gamma$ be the divided power structure
on $J_D \subset D_{P,\eta}(J)$.  

Suppose we are given an inverse system
$(M_n)$ of $D_n$-modules
such that 
\begin{equation}\label{mdn}
M_{n+1} \otimes_{D_{n+1}} D_n \simeq M_n
\end{equation}
for all $n$, and so $M := \varprojlim M_n$ is a $\widehat{D}$-module.
As an example, if we are given a finite projective $\widehat{D}$-module $M$ (which will be the main case of interest) then for
$M_n := M \otimes_{\widehat{D}} D_n$ the natural map $M \to \varprojlim M_n$
is an isomorphism (since expressing $M$ as a direct summand of a free
$D$-module of finite rank reduces this to the trivial case $M=D$).  

We define $\widehat{\Omega}^1_{\widehat{D}/A,\gamma}$ to be the pd-adic completion
of $\Omega^1_{D_{P,\eta}(J)/A,\gamma}$. Since  $\Omega^1_{D_{P,\eta}(J)/A,\gamma} = 
D_{P,\eta}(J) \otimes_P \Omega^1_{P/A}$ (see \cite[3.20(1)]{bogus} and \cite[Tag 07HW]{SP}) and 
$\Omega^1_{P/A}$ is a finite projective $P$-module, clearly
$$\widehat{\Omega}^1_{\widehat{D}/A,\gamma} = \widehat{D} \otimes_P \Omega^1_{P/A},$$
so this is a finite projective $\widehat{D}$-module and as such has an evident complete topology.  Two cofinal systems of quotients
of $\widehat{\Omega}^1_{\widehat{D}/A,\gamma}$ by open $\widehat{D}$-submodules are
$$\{\Omega^1_{D_n/A, \gamma}\}_{n \ge 1},\,\,\,
\{\Omega^1_{D_{P,\eta}(J)/A,\gamma}/n! J_D^{[n]} \Omega^1_{D_{P,\eta}(J)/A,\gamma}\}_{n \ge 1}$$
since we have the factorization
$$\Omega^1_{D_{P,\eta}(J)/A, \gamma}/n! J_D^{[n]} \Omega^1_{D_{P,\eta}(J)/A,\gamma} \twoheadrightarrow 
\Omega^1_{D_n/A,\gamma} \twoheadrightarrow \Omega^1_{D_{P,\eta}(J)/A,\gamma}/(n-1)! J_D^{[n-1]} \Omega^1_{D_{P,\eta}(J)/A,\gamma}$$
as quotients of $\widehat{\Omega}^1_{\widehat{D}/A,\gamma}$.
Thus, the natural maps
\begin{equation}\label{Momega}
M \otimes_{\widehat{D}} \widehat{\Omega}^1_{\widehat{D}/A,\gamma} \to
\varprojlim_{n \ge 1} M_n \otimes_{D_n} \Omega^1_{D_n/A,\gamma},\,\,\,
M \otimes_{\widehat{D}} \widehat{\Omega}^1_{\widehat{D}/A,\gamma} \to
\varprojlim_{n\ge 1} M_n \otimes_{D_n} (\Omega^1_{D_{P,\eta}(J)/A,\gamma}/n! J_D^{[n]} \Omega^1_{D_{P,\eta}(J)/A,\gamma})
\end{equation}
are isomorphisms (so $M \otimes_{\widehat{D}} \widehat{\Omega}^1_{\widehat{D}/A,\gamma}$
coincides with the corresponding completed tensor product)
and the $A$-linear derivation ${\rm{d}}:D_{P,\eta}(J) \to \Omega^1_{D_{P,\eta}(J)/A,\gamma} \to \widehat{\Omega}^1_{\widehat{D}/A,\gamma}$
uniquely extends to a continuous $A$-linear derivation ${\rm{d}}: \widehat{D} \to \widehat{\Omega}^1_{\widehat{D}/A,\gamma}$.
We assumed $P$ is equipped with an \'etale $A$-algebra map
$A[t_1,\dots,t_d] \to P$, so $\Omega^1_{P/A}$ is $P$-free with the ${\rm{d}}t_j$'s as a basis.
Hence, $\widehat{\Omega}^1_{\widehat{D}/A}$ is $\widehat{D}$-free with the ${\rm{d}}t_j$'s as a basis. 

\begin{definition}\label{conndef}
Consider an \'etale map 
$\Spec(P) \to \mathbf{A}^d_A$, an ideal $J \subset P$, $D_n := D_{P,\eta}(J)/n! J_D^{[n]}$, and an inverse system $(M_n)$ of modules over
$(D_n)$ satisfying (\ref{mdn}) as above, and let $\gamma$ denote the divided power structure on
$J_D \subset D_{P,\eta}(J)$.  An $A$-linear {\em connection} on the topological module 
$M = \varprojlim M_n$ over $\widehat{D} = \varprojlim D_n$ is 
 a continuous $A$-linear map 
$$\nabla: M \to M \otimes_{\widehat{D}} \widehat{\Omega}^1_{\widehat{D}/A,\gamma} = \bigoplus_{j=1}^d M  {\rm{d}}t_j$$
satisying $\nabla(fm) = m \otimes {\rm{d}}f + f \nabla(m)$ for $f \in \widehat{D}$.
\end{definition}

Note that the definition of a connection doesn't require $P$ to admit \'etale coordinates over $A$, and doesn't rely on that choice.
But if we want to describe $\nabla$ in concrete terms using the $\widehat{D}$-basis of ${\rm{d}}t_j$'s then we need such a choice.

\begin{example}\label{cartex} One way to make an $A$-linear connection on $M$ is from a compatible system of $A$-linear connections
$\nabla_n: M_n \to M_n \otimes_{D_n} \Omega^1_{D_n/A}$ upon passage to the inverse limit.  The $\nabla$'s arising in this way 
are precisely those for which each composite map
$$M \stackrel{\nabla}{\to} M \otimes_{\widehat{D}} \widehat{\Omega}^1_{\widehat{D}/A,\gamma} \to M_n \otimes_{D_n} \Omega^1_{D_n/A,\gamma}$$
factors through $M \twoheadrightarrow M_n$ for each $n$, in which case
the resulting maps $\nabla_n: M_n \to M_n \otimes_{D_n} \Omega^1_{D_n/A,\gamma}$ are visibly connections compatible with change in $n$.
\end{example}

\begin{definition} A connection $\nabla:M \to M \otimes_{\widehat{D}} \widehat{\Omega}^1_{\widehat{D}/A, \gamma}$ is 
is {\em cartesian} if it arises as in Example \ref{cartex}. 
\end{definition}

\begin{remark} The concept of $\nabla$ being cartesian doesn't involve \'etale coordinates on $P$
but in general it depends on $(M_n)$ rather than being intrinsic to $M$ as a topological
$\widehat{D}$-module even if each $M_n$ is projective of finite rank over $D_n$.  The reason is that it doesn't seem possible to reconstruct
the $D_n$'s from ideal-theoretic data in $\widehat{D}$; see Remark \ref{pdadicrem}.

 In the special case that $R$ is smooth (with \'etale coordinates) and $P = R$ (i.e., $J=0$),
we have $\widehat{D} = R$ and $M$ is just an $R$-module by another name, so
the continuity conditions in Definition \ref{conndef} are automatic and the notion of connection considered here is the usual
algebraic notion $\nabla: M \to M \otimes_R \Omega^1_{R/A}$. In particular, the cartesian condition is automatic in such cases.
\end{remark}

The de Rham complex $\Omega^*_{D_{P,\eta}(J)/A,\gamma} = D_{P,\eta}(J) \otimes_P \Omega^*_{P/A}$
has differentials that are continuous for the pd-adic topology on each term, so we obtain 
a de Rham complex $\widehat{\Omega}^*_{\widehat{D}/A,\gamma} = \widehat{D} \otimes_P \Omega^*_{P/A}$ via termwise pd-adic completion
(with continuous differentials).  By the usual algebraic procedures analogous to \cite[Ch.\,I, \S2]{deligne}, 
a connection $\nabla$ on $M$ defines a collection of continuous $A$-linear maps
$$\nabla^i: M \otimes_{\widehat{D}} \widehat{\Omega}^i_{\widehat{D}/A,\gamma} \to M \otimes_{\widehat{D}} \widehat{\Omega}^{i+1}_{\widehat{D}/A,\gamma}$$
subject to the Leibnitz rule $\nabla(m \otimes \omega) = \nabla(m) \wedge \omega + m \otimes {\rm{d}}\omega$
and 
$\nabla^{i+1} \circ \nabla^i$ is linear given by wedge product against $\nabla^1 \circ \nabla^0: M \to M \otimes_{\widehat{D}} 
\widehat{\Omega}^2_{\widehat{D}/A,\gamma}$.  

\begin{definition} The connection $\nabla$ on $M$ is {\em integrable}
when $\nabla^1 \circ \nabla^0: M \to M \otimes_{\widehat{D}} \widehat{\Omega}^1_{\widehat{D}/A, \gamma}$
vanishes. The resulting complex formed by the $\nabla^i$'s is the {\em de Rham complex} associated to $(M, \nabla)$.
\end{definition}

If $\nabla$ is cartesian then each $\nabla^i$ arises via passage to the inverse limit for the maps
$$\nabla^i_n: M_n \otimes_{D_n} \Omega^i_{D_n/A,\gamma} \to M_n \otimes_{D_n} \Omega^{i+1}_{D_n/A,\gamma}$$
arising from $\nabla_n$, so in such cases $\nabla$ is integrable if and only if every $\nabla_n$ is integrable. 
Using the \'etale coordinates on $P$ we can write 
$$\nabla(m) = \sum_{j=1}^d \theta_j(m) \otimes {\rm{d}}t_j$$
for unique continuous $A$-linear $\theta_j:M \to M$, and $\nabla$ is integrable if and only if the $\theta_j$'s pairwise commute.
Clearly $\nabla$ is cartesian if and only if each $\theta_j: M \to M \twoheadrightarrow M_n$ factors as
$$M \twoheadrightarrow M_n \stackrel{\theta_{j,n}}{\to} M_n$$
for some (unique) $A$-linear $\theta_{j,n}$.

\begin{definition} A {\em morphism} $(M, \nabla) \to (M', \nabla')$ between two such pairs with cartesian connections 
is a $\widehat{D}$-linear map $f:M \to M'$ arising from a (necessarily unique)
a compatible system of $D_n$-linear maps
$f_n:M_n \to M'_n$ compatible with $\nabla_n$ and $\nabla'_n$ for every $n$ (so $f$ is compatible with $\nabla$ and $\nabla'$).

A cartesian connection $\nabla = \varprojlim \nabla_n$ on $M = \varprojlim M_n$
is {\em pd quasi-nilpotent} if for each $m \in M$ and $n \ge 1$, the image of $\theta_j^k(m)$ in $M_n$
is divisible by $n!$ for all sufficiently large $k$ $($perhaps depending on $m \in M$ and $n\ge 1$$)$.
Equivalently, for each $n \ge 1$ and $m \in M_n$, the element $\theta_{j,n}^k(m) \in M_n$ is divisible by $n!$ for all sufficiently large $k$
(perhaps depending on $m$).
\end{definition}

\begin{remark}\label{pdqindep}
Although the notion of a connection $\nabla:M \to M \otimes_{\widehat{D}} \widehat{\Omega}^1_{\widehat{D}/A}$
and the property of it being cartesian or integrable 
have nothing to do with a choice of \'etale coordinates on $P$, such coordinates are needed to definition the $\theta_j$'s and so 
are crucial in the definition of pd quasi-nilpotence. Thus, it is not evident if the property of $\nabla$
being pd quasi-nilpotent is independent of the choice of \'etale coordinates on $P$.  That it is indeed independent
of that choice is immediate from the equivalence of categories in Proposition \ref{micequiv} below.
\end{remark}

\begin{proposition}\label{micequiv} Fix a choice of closed immersion $j:\Spec(R) \hookrightarrow \Spec(P)$
into a smooth affine $A$-scheme admitting an \'etale $A$-map
$q:\Spec(P) \to \mathbf{A}^d_A$.  There is an equivalence of categories $($depending on $j$$)$ 
from the category of crystals on ${\rm{Cris}}(R/A)$
to the category of pairs $(M, \nabla)$ as above with $M = \varprojlim M_n$ and $\nabla$ a pd quasi-nilpotent integrable cartesian connection
on $M$.
\end{proposition}

Keep in mind that $A$ is equipped with a divided power structure $(I_0, \eta)$ and we are assuming
$\eta$ extends to $I_0 R$ (as holds when $R$ is $A$-flat or $I_0$ is principal or $I_0 R = 0$). 

\begin{remark}\label{indepq}
The definition of the functor from crystals to such pairs $(M, \nabla)$ will {\em not} use the existence of $q$, but
the meaning of pd quasi-nilpotence of $\nabla$ and the proof of the equivalence rely upon the choice of $q$.
Thus, the choice of $q$ is essential for the meaning of the result and for the proof, but
once the proof is over we'll know that the choice of $q$ doesn't matter for the meaning and validity of the equivalence,
granting that {\em some} $q$ exists.  Hence, we'll be able to globalize the equivalence
(since some $q$ always exists Zariski-locally); see Proposition \ref{globalmic} for the global version.
\end{remark}

Proposition \ref{micequiv}  is an analogue of \cite[Tag 07JH]{SP}.  If 
$R$ is $A$-smooth with \'etale coordinates then 
we can choose $P=R$ (so $J=0$), in which case the equivalence is with the category of $R$-modules
equipped with an integrable $A$-linear connection satisfying the additional divisibility condition on the $\theta_j^k$'s (which is automatic
when $A$ is a $\mathbf{Q}$-algebra).  In the special case that 
$A$ is a $\mathbf{Z}/N\mathbf{Z}$-algebra for some $N \ge 1$, the condition on the $\theta_j$'s
is that they are pointwise nilpotent on each $M_n$ for sufficiently large $n$.

Before we prove Proposition \ref{micequiv}, we record a natural supply of examples of vector bundles equipped with a pd-quasi-nilpotent connection
in the case $J=0$ and we prove a preliminary algebraic result that will be useful both here and later on to reduce some tasks with general smooth
algebras to the setting of polynomial algebras.

\begin{example}\label{katzex} Let  $X \to \Spec(R)$ be smooth and proper over a $\mathbf{Z}$-flat ring
$R$, and assume its de Rham cohomologies
are vector bundles (so their formation commutes with any base change on $R$).
Then the integrable Gauss-Manin connection $\mathbf{H}^i_{\rm{dR}}(X/R)$ is pd quasi-nilpotent: a more refined version of 
this divisibility condition on the $\theta_j$'s is given in \cite[Thm.\,5.10]{katz}.
\end{example}

As a preliminary step for the proof of Proposition \ref{micequiv}, we need to
discuss a result on rewriting certain divided power envelopes as
a divided power envelope of another type.  This is the mechanism by which we will be able to work
with general smooth $A$-algebras admitting \'etale coordinates (rather than just polynomial rings) in what follows.
First, we record a useful lemma that will be applied to smooth algebras:

\begin{lemma}\label{smoothpd} Let $(A,J,\gamma)$ be a divided power ring and $B$ be an $A$-algebra
equipped with a section $\sigma: \Spec(A) \hookrightarrow \Spec(B)$ corresponding to an ideal $I \subset B$.
  In the divided power envelope $D_{B,0}(I)$,
there exists a unique divided power structure on $J D_{B,0}(I) + I_D$ extending the natural divided power structure
on $I_D$ and the given $\gamma$ on $J$.  More specifically, the natural map
$D_{B,0}(I) \to D_{B,\gamma}(I)$ is an isomorphism. 
\end{lemma}

\begin{proof}
The isomorphism property for $D_{B,0}(I) \to D_{B,\gamma}(I)$ is \cite[3.20(6)]{bogus}.
\end{proof}

\begin{proposition}\label{newpd} Let $(A, I_0, \eta)$ be a divided power ring, 
$P$ a smooth $A$-algebra, and $R = P/J$.
Let $P'$ be a smooth $P$-algebra equipped with
a section $\sigma:\Spec(P) \hookrightarrow \Spec(P')$ corresponding to an ideal $K$,
and let $J'$ be the kernel of $P' \twoheadrightarrow P \twoheadrightarrow R$ $($i.e., $J'/K = J$$)$.

For the smooth $D_{P,\eta}(J)$-algebra $Q := D_{P,\eta}(J) \otimes_P P'$
and ideal $I := D_{P,\eta}(J) \otimes_P K$ of the associated section
$\Spec(D_{P,\eta}(J)) \hookrightarrow \Spec(Q)$, naturally 
\begin{equation}\label{divisom}
D_{P',\eta}(J') \simeq D_{Q,\eta}(I) = D_{Q,0}(I)
\end{equation} as divided power $D_{P,\eta}(J)$-algebras respecting 
the natural divided powers on $J'_D$ and $I_D + J_D D_{Q,\eta}(I)$ compatible with $\eta$ on $I_0 \subset A$.
\end{proposition}

The equality $D_{Q,0}(J) = D_{Q,\eta}(I)$ is an instance of Lemma \ref{smoothpd}
applied to the algebra $Q$ over the divided power base ring $D_{P,\eta}(J)$ equipped
with its divided powers on $I_0 D_{P,\eta}(J)$ extending $\eta$. 

\begin{remark} The merit of (\ref{divisom}) is that whereas $J' \subset P'$ is not the ideal of a section to
the smooth structure map $\Spec(P') \to \Spec(A)$, $I \subset Q$
is the ideal of a section to the smooth map $\Spec(Q) \to \Spec(D_{P,\eta}(J))$. When combined with Proposition \ref{dqt},
we'll be able to analyze $D_{P',\eta}(J')$ by carrying out computations with $D_{Q,0}(I)$ ``as if'' $\Spec(Q)$ were an affine space over $D_{P,\eta}(J)$
with $I$ the ideal of the origin (the price to pay is that the original divided power base ring $A$ is replaced with
$D_{P,\eta}(J)$ equipped with its divided powers on $J_D$ compatible with $\eta$); 
see (\ref{D1isom}) for an illustration.  The main cases of interest will be $P' = P^{\otimes_A (\nu+1)}$ (with $\nu \ge 1$) viewed as a $P$-algebra
via the first tensor factor and $\sigma$ the diagonal section.
\end{remark}

\begin{proof}
Note that $J' = K + JP'$ (since $JP' = J \otimes_P P'$ by $P$-flatness of $P'$). 
We will use the universal properties of divided power envelopes to build $D_{P,\eta}(J)$-algebra maps
$$g:D_{Q,\eta}(I) \to D_{P',\eta}(J'),\,\,\,
h:D_{P',\eta}(J') \to D_{Q,\eta}(I)$$ 
that will be  inverse to each other. 

Since $D_{P',\eta}(J')$ is a $D_{P,\eta}(J)$-algebra and a $P'$-algebra compatibly with the $P$-algebra structure on each,
we get a natural ring map $f:Q \to D_{P',\eta}(J')$ over $D_{P,\eta}(J)$.  This carries $I= D_{P,\eta}(J) \otimes_P K$
into the ideal generated by $J'$ and hence into the ideal $J'_D$ that has divided powers, so $f$ uniquely factors through a
map of divided power rings $g: D_{Q,\eta}(I) \to D_{P',\eta}(J')$  (over $D_{P,\eta}(J)$ since $f$ is a $D_{P,\eta}(J)$-algebra map).
To build $h$ in the reverse direction, we look more closely at $D_{Q,\eta}(I)$.  In this ring we have the ideal
$I_D$ equipped with divided powers, and as a $D_{P,\eta}(J)$-algebra we claim that the ideal generated by 
$J_D$ and $I_D$ has unique divided powers that extend the ones on $I_D$ and
$J_D \subset D_{P,\eta}(J)$.  This is an instance of Lemma \ref{smoothpd}. 

Since $I = D_{P,\eta}(J) \otimes_P K$, the natural $P$-algebra map
$$P' \to Q \to D_{Q}(I)$$
carries $J' = K + JP'$ into the ideal $J_D D_{Q}(I) + I_D$ equipped
with divided powers via Lemma \ref{smoothpd}.  Thus, we obtained a natural map of divided power rings
$g:D_{P',\eta}(J') \to D_{Q,\eta}(I)$ over $P \to D_{P,\eta}(J)$, so $g$ is clearly a map of
$D_{P,\eta}(J)$-algebras.  Since $g$ and $h$ are each built as maps of divided power $D_{P,\eta}(J)$-algebras, it is straightforward
to check via universality considerations (or bare-hand computations with generators) that they are inverse to each other.
This finishes the proof of Proposition \ref{newpd}.
\end{proof}

\begin{proof}(of Proposition \ref{micequiv}) 
Let $\mathscr{F}$ be a crystal on ${\rm{Cris}}(R/A)$.
Each $D_n = D_{P,\eta}(J)/[n]! J_D^{[n]}$ is an object in ${\rm{Cris}}(R/A)$, 
so we can define a $D_n$-module $M_n = \mathscr{F}_{D_n}$.
These modules 
satisfy (\ref{mdn}) by the crystal property
and  are equipped with compatible integrable connections $\nabla_n:M_n \to M_n \otimes_{D_n} \Omega^1_{D_n/A,\gamma}$
(see the discussion preceding (\ref{Fcris})).  Then $M := \varprojlim M_n$
is a $\widehat{D}$-module equipped with a cartesian integrable connection
$$\nabla = \varprojlim \nabla_n: M \to M \otimes_{\widehat{D}} \widehat{\Omega}^1_{\widehat{D}/A, \gamma}.$$

Upon writing
$\nabla_n(m) = \sum_{j=1}^d \theta_{j,n}(m) \otimes {\rm{d}}t_j$ for $m \in M_n$,
we have to show $\theta_{j,n}^k(m) \in n! M_n$ for all large $k$ (depending on $m$ and $n$).
For $P(1) := P \otimes_A P$ and $J(1) = \ker(P(1) \twoheadrightarrow P \twoheadrightarrow P/J = R)$,
consider $D(1) := D_{P \otimes_A P,\eta}(J(1))$.  By combining Propositions \ref{dqt} and \ref{newpd}, we have isomorphisms 
\begin{equation}\label{D1isom}
D(1)/n!J(1)_D^{[n]} \simeq D_n \langle \xi_1, \dots, \xi_d\rangle/n!(J_D + \langle \xi_1, \dots, \xi_d\rangle_+)^{[n]}
\end{equation} 
compatibly with change in $n$, where $\xi_j = t_j \otimes 1 - 1 \otimes t_j$.
For any $0 \le m \le n$ we have either $n-m \ge n/2$ or $m \ge n/2$, so 
the two systems of quotients
$$D(1)/n! J(1)_D^{[n]},\,\,\,
D_n(1) := D_n \langle \xi_1, \dots, \xi_d \rangle/n! \langle \xi_1, \dots, \xi_d \rangle_+^{[n]}$$
of $D(1)$ interweave each other, so the pd-adic completion $\widehat{D}(1)$ of $D(1)$ is computed as
the inverse limit of each system of quotients. In particular, naturally $\widehat{D}(1)$ is contained in the formal divided power series
$\widehat{D} \langle\!\langle \xi_1, \dots, \xi_d \rangle\!\rangle$ over the pd-adic completion $\widehat{D}$ of $D$.

The crystal property of $\mathscr{F}$ provides $D_n(1)$-linear isomorphisms
$$D_n(1) \otimes_{D_n} M_n \simeq \mathscr{F}_{D_n(1)} \simeq M_n \otimes_{D_n} D_n(1),$$
so upon passage to the inverse limit we get a composite isomorphism of $\widehat{D}(1)$-modules
$$c:\widehat{D}(1) \widehat{\otimes}_{\widehat{D}} M  \simeq M \widehat{\otimes}_{\widehat{D}} \widehat{D}(1) \subset
\prod_{K} M \xi^{[K]},$$
where $K$ varies through multi-indices $(k_1,\dots,k_d)$ in non-negative integers and $\xi^{[K]} := \xi_1^{[k_1]} \cdots \xi_d^{[k_d]}$. 
Writing $c(1 \widehat{\otimes} m) = (\theta_K(m) \xi^{[K]})$, we want to relate the $A$-linear $\theta_K: M \to M$
to the $\theta_j$'s.  By applying Proposition \ref{newpd} to $P(2) = P \otimes_A P \otimes_A P$
and $J(2) = \ker(P(2) \twoheadrightarrow R)$, we have
$$D_{P(2),\eta}(J(2))/n! J(2)_D^{[n]} \simeq D_n \langle \xi_1, \dots, \xi_d, \xi''_1, \dots, \xi''_d \rangle/
n!\langle \xi_1, \dots, \xi_d, \xi'_1, \dots, \xi'_d \rangle_+^{[n]}$$
where $\xi_j = t_j \otimes 1 \otimes 1 - 1 \otimes t_j \otimes 1$
and $\xi''_j = t_j \otimes 1 \otimes 1 - 1 \otimes 1 \otimes t_j$.
Computations as in the second half of the proof of \cite[Tag 07JG]{SP} then yield that $\theta_K = \prod \theta_j^{k_j}$,
so for each $j$ and each $k \ge n$ we see that the image of $\theta_j^k(m)$ in $M_n/n! M_n$
is the $\xi_j^{[k]}$-coefficient of the image of $c(1 \widehat{\otimes} m)$ in 
$$M_n \otimes_{D_n} D_n(1) = M_n \otimes_A A \langle \xi_1, \dots, \xi_d \rangle/n!\langle \xi_1, \dots, \xi_d \rangle_+.$$
For sufficienty large $k \ge n$ this coefficient vanishes since a divided power polynomial involves only finitely many monomials. 

Now we go in reverse: given an $(M, \nabla)$ of the desired type, 
we will naturally build a crystal $\mathscr{F}$ and then see that this naturally inverts the preceding functor
in the other direction.  We may write $\nabla(m) = \sum_j \theta_j(m) {\rm{d}}t_j$
for $A$-linear $\theta_j:M \to M$ arising from a compatible system of $A$-linear maps $\theta_{j,n}: M_n \to M_n$
since $\nabla$ is cartesian, and the $\theta_j$'s pairwise commute since $\nabla$ is integrable. 
Define $\theta^K = \prod \theta_j^{k_j}$ for $K = (k_1, \dots, k_d)$ (so $\theta^{(0,\dots,0)} = {\rm{id}}$). 
For each $n \ge 1$ and $m \in M$, if $K = (k_1, \dots, k_d)$
is a multi-index with $\sum_j k_j$ sufficiently large (equivalent to some $k_j$ to be sufficiently large)
then the image of $\theta^K(m)$ in $M_n$ belongs to $n! M_n$. 

By Remark \ref{crismap}, there exists a map $f:D_N = D_{P,\eta}(J)/N! J_D^{[N]} \to B$ in ${\rm{Cris}}(R/A)$
(i.e., this carries $J_D/N! J_D^{[N]}$ into $I$ respecting the divided powers).  We want to define
$$\mathscr{F}_B := M_N \otimes_{D_N,f} B;$$
this is unaffected by passing to larger $N$ since we are assuming (\ref{mdn}), 
so it only depends on the chosen map $D_{P,\eta}(J) \to B$ of divided power $A$-algebras, 
but we have to check that for a second map $g:D_N \to B$ in ${\rm{Cris}}(R/A)$ there is a canonical $B$-linear isomorphsm
$$c_{f,g}: M_N \otimes_{D_N,f} B \simeq M_N \otimes_{D_N,g} B$$
satisfying 
\begin{equation}\label{ctrans}
c_{g,h} \circ c_{f,g} = c_{f,h}
\end{equation} for a third such map $h:D_N \to B$.

Arguing as in the proof of \cite[Tag 07JH]{SP}, since $b_j := f(t_j) - g(t_j) \in \ker(B \to R) = J$,
we have $N! b_1^{[k_1]} \cdots b_d^{[k_d]} = 0$ whenever $\sum k_j \ge N$ since $N! J^{[N]} = 0$.
Thus, it is well-posed to define a $B$-linear map
$$c_{f,g}: M_N \otimes_{D_N,f} B \to M_N \otimes_{D_N,g} B$$
by $m \otimes 1 \mapsto \sum_K \theta^{K}(m) \otimes b^K$
since the sum is finite (due to the vanishing of $N! b^K$ when $\sum k_j \ne N$
and the fact that $\theta^K(m)$ has image in $M_N$ belonging to $N! M_N$ when $\sum k_j$ is 
sufficiently large depending on $m$ and $N$).
Note that $c_{f,f} = {\rm{id}}$, and the desired identity (\ref{ctrans}) says
$$\sum_{K''} \theta^{K''}(m) b_{f,h}^{[K'']} \stackrel{?}{=} \sum_{K,K'} \theta^{K''}(\theta^{K}(m)) b_{f,g}^{[K]} b_{g,h}^{[K']}$$
with $b_{f,g} = (f(t_1) - g(t_1), \dots, f(t_d) - g(t_d))$ and similarly for $b_{g,h}, b_{f,h}$.  Since
$b_{f,h} = b_{f,g} + b_{g,h}$ and $\theta^{K''} \circ \theta^{K} = \theta^{K'' + K'}$, the desired
identity holds by expanding $(b_{f,g} + b_{g,h})^{[K'']}$ in terms of all expressions of $K''$ as a sum of two multi-indices.
From the identity (\ref{ctrans}) we see that the maps $c_{f,g}$ are isomorphisms
of the required type and that $\mathscr{F}_{(M, \nabla)}: B \mapsto M_N \otimes_{D_N, f} B$ is a crystal.

To show that these two constructions are naturally inverse to each others, the only non-trivial point is
to check that for $\mathscr{F} = \mathscr{F}_{(M, \nabla)}$ the resulting connection on each
$\mathscr{F}_{D_n} = M_n$ is exactly $\nabla_n$.  This in turn follows from
the fact that in the definition of the $c_{f,g}$'s above (which underlies the crystal property of $\mathscr{F}_{(M,\nabla)}$)
we used $b_j := f(t_j) - g(t_j)$ with $\theta_j$ the ${\rm{d}}t_j$-coefficient of $\nabla$.
\end{proof}

\medskip\medskip

\section{An affine comparison theorem using pd-adic completions}\label{compsec}

\subsection{\v{C}ech-Alexander complexes and de Rham theory}

Let $(A, I_0, \eta)$ be a divided power ring, $R$ a finite type $A$-algebra
for which $\eta$ extends to $I_0 R$ (as holds when $R$ is $A$-flat or $I_0$ is principal or $I_0 R = 0$), 
and $\pi:P \twoheadrightarrow R$ a surjection over $A$
from a smooth $A$-algebra $P$ (e.g., a polynomial ring). Let $J = \ker(\pi)$. 
Adapting notation from \cite[2.1]{BdJ}, 
for every $n \ge 0$ let $$J(n) = \ker(P(n) := P^{\otimes_A (n+1)} \twoheadrightarrow R)$$
(so $J(0) = J$) and in the divided power envelope $D_{P(n)}(J(n))$ over $P(n)$ denote by
$J(n)_D$ the canonical divided power ideal and denote by $\gamma(n)$
the divided power structure on $J(n)_D$.  For $N \ge 1$, the quotient ring
$$D_{P(n),\eta}(J(n))/N! J(n)_D^{[N]}$$
is a divided power thickening of $R$ belonging to ${\rm{Cris}}(R/A)$, 
so for any {\em crystal} $\mathscr{F}$ in $\mathscr{O}^{\rm{cris}}_{R/A}$-modules 
 we may form the inverse limit $D_{P(n),\eta}(J(n))$-algebra
$$D(n) = \varprojlim_{N \ge 1} D_{P(n),\eta}(J(n))/N! J(n)_D^{[N]}$$
and $D(n)$-module
$$\mathscr{F}(n) = \varprojlim_{N \ge n} \mathscr{F}_{D_{P(n),\eta}(J(n))/N! J(n)_D^{[N]}}.$$
Since the ideal $J(n)_D^{[N]}$ is generally not finitely generated, we do not expect that the natural surjections
$$D(n)/N! J(n)_D^{[N]} D(n) \to D_{P(n),\eta}(J(n))/N! J(n)_D^{[N]},$$
$$\mathscr{F}(n)/N! J(n)_D^{[N]} \mathscr{F}(n) \to 
\mathscr{F}_{D_{P(n),\eta}(J(n))/N! J(n)_D^{[N]}}$$
to be isomorphisms. 

As $n \ge 0$ varies, the $\mathscr{F}(n)$'s constitute a cosimplicial $A$-module which in turn yields a complex of $A$-modules that we denote as
$${\rm{\mbox{\v{C}}A}}(P \twoheadrightarrow R, \mathscr{F}).$$
In analogy with \cite[2.1]{BdJ}, this is called the {\em \v{C}ech-Alexander complex} for $P \twoheadrightarrow R$ and $\mathscr{F}$;
for $\mathscr{F} = \mathscr{O}_{R/A}^{\rm{cris}}$ we denote it ${\rm{\mbox{\v{C}}A}}(P \twoheadrightarrow R)$.

\begin{theorem}\label{dPoincare}
In the above setting with a crystal $\mathscr{F}$, if $P$ admits \'etale coordinates over $A$ $($i.e., an \'etale map of $A$-algebras
$g:A[t_1,\dots,t_d] \to P$$)$ then there is a natural isomorphism
{\rm{
\begin{equation}\label{camap}
{\rm{\mbox{\v{C}}A}}(P \twoheadrightarrow R, \mathscr{F})  \simeq
\mathscr{F} \widehat{\otimes} 
 \widehat{\Omega}^*_{D_{P,\eta}(J)/A,\gamma} := \varprojlim_{N \ge 1} \mathscr{F}_{D_{P,\eta}(J)/N! J_D^{[N]}} 
 \otimes_{D_{P,\eta}(J)/N! J_D^{[N]}} 
 \Omega^*_{(D_{P,\eta}(J)/N! J_D^{[N]})/A,\gamma}
 \end{equation}
}}in the derived category of $A$-modules, where the right side is an inverse limit of complexes as in $(\ref{Fcris})$.
\end{theorem}

The proof of this result is unfortunately somewhat long, and will occupy the rest of \S\ref{compsec}.  First, we make two preliminary remarks.

\begin{remark}
The construction of the map (\ref{camap}) in the derived category is as a zig-zag of maps of complexes which
won't use the \'etale coordinates. But the proof
that it is a zig-zag of quasi-isomorphisms will use such a choice.  
Note that in the special case $P=R$ (so $J=0$), the right side is $\mathscr{F}_R \otimes_R \Omega^*_{R/A}$. 
The method we use is a variant on the technique introduced in \cite[\S2]{BdJ} for the case when the rings are
$\mathbf{Z}/(p^r)$-algebras for some $r \ge 1$.  
\end{remark}

\begin{remark}
In the special case $\mathscr{F} = \mathscr{O}^{\rm{cris}}_{R/A}$, for $\mu \ge 1$ the 
target of (\ref{camap}) has $\mu$th term
\begin{eqnarray*}
\widehat{\Omega}^{\mu}_{D_{P,\eta}(J)/A,\gamma} &=& \varprojlim_{N \ge 1} \Omega^{\mu}_{(D_{P,\eta}(J)/N! J_D^{[N]})/A,\gamma} \\
&=& \varprojlim_{N \ge 1} \Omega^{\mu}_{D_{P,\eta}(J)/A,\gamma}/(N! J_D^{[N]} \Omega^{\mu}_{D_{P,\eta}(J)/A,\gamma} + 
N! J_D^{[N-1]} {\rm{d}}J \wedge \Omega^{\mu-1}_{D_{P,\eta}(J)/A,\gamma})
\end{eqnarray*}
(see \cite[Tag 07HS(3)]{SP} for the behavior of divided power $\Omega^1$, and hence divided power 
$\Omega^{\mu} = \wedge^{\mu} \Omega^1$, under passage to divided
power quotient algebras).  This is the same as the pd-adic completion
$$\varprojlim_{N \ge 1} \Omega^{\mu}_{D_{P,\eta}(J)/A,\gamma}/N! J_D^{[N]} \Omega^{\mu}_{D_{P,\eta}(J)/A,\gamma}$$
of $\Omega^{\mu}_{D_{P,\eta}(J)/A,\gamma}$ since the two  inverse systems of quotients 
of $\Omega^{\mu}_{D_{P,\eta}(J)/A,\gamma}$ are cofinal in each other.
\end{remark}

We define the first-quadrant double complex
$$E^{\mu,\nu} := 
 \varprojlim_{N \ge 0} \mathscr{F}_{D_{P(\nu),\eta}(J(\nu))/N! J(\nu)_D^{[N]}} \otimes_{D_{P(\nu),\eta}(J(\nu))/N! J(\nu)_D^{[N]}} 
 \Omega^{\mu}_{(D_{P(\nu),\eta}(J(\nu))/N! J(\nu)_D^{[N]})/A,\gamma(\nu)}.
$$
Its differentials in the $\mu$-directions
are $(-1)^{\mu}$ times evaluating (\ref{Fcris}) on 
$D_{P(\nu),\eta}(J(\nu))/N! J(\nu)_D^{[N]}$ for each $N \ge 1$, and its 
differentials in the $\nu$-directions are alternating sums of the maps induced via
functoriality from the $\nu+2$ inclusions $P(\nu) \to P(\nu+1)$ (which carry $J(\nu)$ into $J(\nu+1)$ and thereby give 
rise to $\nu+2$ maps $D_{P(\nu),\eta}(J(\nu)) \to D_{P(\nu+1),\eta}(J(\nu+1))$
of divided power $A$-algebras).  The reason for the sign $(-1)^{\mu}$ is to ensure we get anti-commuting squares as in the definition of
``double complex''. 

Along the edges we have 
$$E^{\ast,0} = \mathscr{F} \widehat{\otimes} \widehat{\Omega}^*_{D_{P,\eta}(J)/A,\gamma},\,\,\,
E^{0,\ast} = {\rm{\mbox{\v{C}}A}}(P \twoheadrightarrow R, \mathscr{F}),$$ 
so it suffices to show that the total complex ${\rm{Tot}}(E^{\ast,\ast})$ is quasi-isomorphic to its two edges along the edge maps.

\subsection{Analysis of column complexes}

For the quasi-isomorphism with $E^{0,\ast} = {\rm{\mbox{\v{C}}A}}(P \twoheadrightarrow R, \mathscr{F})$, 
it suffices to show that for each $\mu > 0$ the complex $E^{\mu,\ast}$ is exact (for which it is harmless to remove the multiplication by
$(-1)^{\mu}$ on its differentials).

Since $\Omega^{\mu}_{D_{P,\eta}(J)/A, \gamma} = \Omega^{\mu}_{P/A} \otimes_P D_{P,\eta}(J)$ \cite[Tag 07HW]{SP}, we have
$$\Omega^{\mu}_{(D_{P,\eta}(J)/N! J_D^{[N]})/A, \gamma} =
\Omega^{\mu}_{D_{P,\eta}(J)/A, \gamma}/(N! J_D^{[N]} \Omega^{\mu}_{D_{P,\eta}(J)/A,\gamma} + N! J_D^{[N-1]} {\rm{d}}J \wedge
\Omega^{\mu-1}_{D_{P,\eta}(J)/A,\gamma}).$$
To show for $\mu > 0$ that the complex of $A$-modules
$E^{\mu,\ast}$ is homotopic to 0,  we need to naturally identify
$E^{\mu,\ast}$ with an inverse limit of a slightly different inverse system, in which
we use the quotients of $\Omega^{\mu}_{D_{P(\ast),\eta}(\ast)(J(\ast))/A,\gamma(\ast)}$
by $N! J(\ast)_D^{[N]}$. 

More specifically, we claim that the natural map
\begin{equation}\label{Fomegamap}
\varprojlim_{N \ge 1} \mathscr{F}_{D_{P,\eta}(J)/N! J_D^{[N]}} \otimes
\Omega^{\mu}_{D_{P,\eta}(J)/A,\gamma}/N! J_D^{[N]} \Omega^{\mu}_{D_{P,\eta}(J)/A,\gamma}
\to \varprojlim_{N \ge 1} \mathscr{F}_{D_{P,\eta}(J)/N! J_D^{[N]}} \otimes
\Omega^{\mu}_{(D_{P,\eta}(J)/N! J_D^{[N]})/A,\gamma}
\end{equation}
is an isomorphism (and so also with $(P, J)$ replaced by $(P(\nu), J(\nu))$ for all $\nu$).
Letting $\mathscr{F}_N := \mathscr{F}_{D_{P,\eta}(J)/N! J_D^{[N]}}$ and 
$\Omega^{\mu}_N := \Omega^{\mu}_{D_{P,\eta}(J)/A,\gamma}/N! J_D^{[N]} \Omega^{\mu}_{D_{P,\eta}(J)/A,\gamma}$
to simplify notation, the map (\ref{Fomegamap}) is 
\begin{equation}\label{Fomega}
\varprojlim \mathscr{F}_N \otimes \Omega^{\mu}_N \to
 \varprojlim \mathscr{F}_N \otimes \Omega^{\mu}_N/(N! J_D^{[N-1]} {\rm{d}}J \wedge \Omega^{\mu-1}_N).
 \end{equation}
 Since $\mathscr{F}_N \to \mathscr{F}_{N-1}$ is surjective (due to the {\em crystal} property), each transition map
 $$\mathscr{F}_N \otimes \Omega^{\mu}_N \to \mathscr{F}_{N-1} \otimes \Omega^{\mu}_{N-1}$$
 is surjective and (since $(N-1)!$ divides $N!$) has a unique factorization
 \begin{equation}\label{Ffactor}
\mathscr{F}_N \otimes \Omega^{\mu}_N \twoheadrightarrow 
 \mathscr{F}_N \otimes \Omega^{\mu}_N/(N! J_D^{[N-1]} {\rm{d}}J \wedge \Omega^{\mu-1}_N)
  \twoheadrightarrow \mathscr{F}_{N-1} \otimes \Omega^{\mu}_{N-1}
  \end{equation}
  whose first map underlies (\ref{Fomega}).  Thus, the two inverse systems of surjections in (\ref{Ffactor}) 
  interlace each other and hence (\ref{Fomega}) is an isomorphism. 
  
  Applying the preceding for each $(P(\ast), J(\ast))$ in place of $(P, J)$, 
to show $E^{\mu, \ast}$ is homotopic to 0 for each $\mu > 0$ we consider the 
following alternative inverse limit expression for $E^{\mu,\ast}$:
\begin{eqnarray*}
\varprojlim_{N \ge 1} \mathscr{F}_{D_{P(\ast),\eta}(J(\ast))/N! J(\ast)_D^{[N]}} 
\otimes_{D_{P(\ast),\eta}(J(\ast))/N! J(\ast)_D^{[N]}} \Omega^{\mu}_{D_{P(\ast),\eta}(J(\ast))/A, \gamma(\ast)}/N! J(\ast)_D^{[N]} 
\Omega^{\mu}_{D_{P(\ast),\eta}(J(\ast))/A, \gamma(\ast)}\\
= \varprojlim_{N \ge 1} \mathscr{F}_{D_{P(\ast),\eta}(J(\ast))/N! J(\ast)_D^{[N]}} 
\otimes_A \Omega^{\mu}_{P(\ast)/A},
\end{eqnarray*}
the equality using the
natural isomorphisms $\Omega^{\mu}_{D_{P(\ast),\eta}(J(\ast))/A,\gamma(\ast)} = D_{P(\ast),\eta}(J(\ast)) \otimes_{P(\ast)} \Omega^{\mu}_{P(\ast)/A}$
(see \cite[3.20(1)]{bogus} and then \cite[Tag 07HW]{SP} with $\mu=1$, and then pass to $\mu$th exterior powers).

We shall first show that the cosimplicial module $\Omega^1_{P(\ast)/A}$ over the cosimplicial $A$-algebra $P(*)$ is homotopic to 0 (over $P(*)$), so then 
by \cite[Tag 07KQ, Part (2)]{SP} such homotopy to 0 is inherited by the cosimplicial $P(*)$-module $\Omega^{\mu}_{P(\ast)/A}$'s
for fixed $\mu > 0$, and 
then for each $N \ge 1$ by  (the proof of) \cite[Tag 07KQ, Part (1)]{SP} the complex
$$\mathscr{F}_{D_{P(\ast),\eta}(J(\ast))/N! J(\ast)_D^{[N]}} \otimes_A \Omega^{\mu}_{P(\ast)/A}$$
for fixed $\mu>0$ is homotopic to 0 over $A$ with homotopies that are compatble with change in $N$.
Hence, the inverse limit $E^{\mu, \ast}$ over $N$ would be homotopic to 0 for $\mu > 0$, as desired.
The merit of having reduced our task
to showing $\Omega^1_{P(\ast)/A}$ is homotopic to 0 over $P(*)$ is that this has nothing to do with
the crystal $\mathscr{F}$ nor with the quotient $P/J=R$ or the ideals $J(\ast) = \ker(P(\ast) \to R)$.

Using an \'etale $A$-algebra map $g:A[\underline{t}]  := A[t_1, \dots, t_d] \to P$, we have the natural isomorphisms
$$P(\ast) \otimes_{A[\underline{t}](\ast)} \Omega^1_{A[\underline{t}]/A} \simeq \Omega^1_{P(\ast)/A}$$
as cosimplicial $P(*)$-modules. 
By \cite[Tag 07KQ, Part (3)]{SP}, for any map $B_{*} \to C_{*}$ of cosimplicial $A$-algebras
and any cosimplicial $A$-module $M_{*}$ over $B_{*}$, the $A$-linear complex $M_{*} \otimes_{B_{*}} C_{*}$ is homotopic
to 0 if $M_{*}$ is.  Thus, we are reduced to the case $P = A[\underline{t}]$ is a polynomial ring.
This special case is treated in the proof of \cite[Tag 07L9]{SP}.  

\subsection{Analysis of row complexes}\label{rowsec}

Now we fix $\nu \ge 0$ and consider the rows 
$E^{\ast,\nu}$.  The map $E^{\ast,\nu} \to E^{\ast,\nu+1}$ is given in each degree $\ast$
as an alternating sum of maps arising from pullback along the $\nu+2$ inclusions $\iota_j^{\nu}:P(\nu) \to P(\nu+1)$
for $0 \le j \le \nu+1$. 
The effect of each induces the {\em same} map in homology 
$${\rm{H}}^i(\iota_j^{\nu}) = f^i_{\nu}: {\rm{H}}^i(E^{\ast,\nu}) \to {\rm{H}}^i(E^{\ast,\nu+1})$$
(using standard co-degeneracies $P(\nu+1) \to P(\nu)$ to make homotopies between the maps of complexes),
so the alternating sums are 0 or that common map, depending on the parity of $\nu$. 
We will show that the maps $f^i_{\nu}$ are isomorphisms for all $i$, so then the map
$$\sum_{j=0}^{\nu+1} (-1)^j {\rm{H}}^{\mu}(\iota_j^{\nu}) =
\sum_{j=0}^{\nu+1} (-1)^j f^{\mu}_{\nu}:
{\rm{H}}^{\mu}(E^{\ast,\nu}) \to {\rm{H}}^{\mu}(E^{\ast,\nu+1})$$
would vanish for all $\mu$ when $\nu$ is even (e.g., $\nu=0$) and be a quasi-isomorphism for all $\mu$ when $\nu$ is odd. 
Hence, setting $E^{\mu,-1} := 0$ for all $\mu$, the homologies 
$${\rm{E}}_2^{\mu,\nu} = \ker({\rm{H}}^{\mu}(E^{\ast,\nu}) \to {\rm{H}}^{\mu}(E^{\ast,\nu+1}))/{\rm{im}}({\rm{H}}^{\mu}(E^{\ast,\nu-1}) \to
{\rm{H}}^{\mu}(E^{\ast,\nu}))$$
would vanish for all $\nu > 0$ and ${\rm{E}}_2^{\mu,0} = {\rm{H}}^{\mu}(E^{\ast,0})$, so 
the edge map from $E^{\ast,0}$ to the total complex would be a quasi-isomorphism (as desired).

To prove $f^i_{\nu}$ is an isomorphism for all $i$ and all $\nu \ge 0$, it suffices to show for each $\nu \ge 0$ that
that composite map $E^{\ast,0} \to E^{\ast,\nu}$ arising from a choice of inclusion $P = P(0) \to P(\nu) = P^{\otimes_A (\nu +1)}$
is a quasi-isomorphism; we'll use inclusion along the 0th tensor factor
for specificity, and of course we can assume $\nu > 0$.  
The key is to rewrite the $D_{P,\eta}(J)$-algebra $D_{P(\nu),\eta}(J(\nu))$ as a divided power envelope
of a smooth $D_{P,\eta}(J)$-algebra with a section as in Proposition \ref{newpd}, so we can then apply Proposition \ref{dqt} 
(with $D_{P,\eta}(J)$ in the role of $R$).  More specifically,
for $\nu > 0$ let $Q(\nu)$ be the smooth $D_{P,\eta}(J)$-algebra $D_{P,\eta}(J) \otimes_P P(\nu)$
and let $I(\nu)$ be the kernel of the map $Q(\nu) \twoheadrightarrow D_{P,\eta}(J)$ induced by 
scalar extension of the natural map $P(\nu) \twoheadrightarrow P$.
Then Proposition \ref{newpd} provides an isomorphism
$$D_{Q(\nu),0}(I(\nu)) \simeq D_{P(\nu),\eta}(J(\nu))$$ as divided power $D_{P,\eta}(J)$-algebras, 
compatibly with divided powers on $I(\nu)_D + J_D D_{Q(\nu)}(I(\nu))$ and $J(\nu)_D$.

For $n \ge 1$ and $\nu \ge 0$, define
$$D(\nu)_n = D_{P(\nu),\eta}(J(\nu))/n! J(\nu)_D^{[n]}, \,\,\,
M(\nu)_n = \mathscr{F}(D(\nu_n)),\,\,\,D_n = D(0)_n,\,\,\,
M_n = M(0)_n,$$
so $E^{\ast, \nu} = \varprojlim_{n \ge 1} M(\nu)_n \otimes_{D(\nu)_n} \Omega^{\ast}_{D(\nu)_n/A, \gamma(\nu)}$
(made into a complex using the connections at each finite level arising from the {\em crystal} nature of $\mathscr{F}$).
We want to show that for a {\em fixed} $\nu > 0$ and the inclusion $P \to P(\nu)$ into the 0th tensor factor, the natural map
$$ \varprojlim_{n \ge 1} M_n \otimes_{D_n} \Omega^{\ast}_{D_n/A, \gamma} \to
\varprojlim_{n \ge 1} M(\nu)_n \otimes_{D(\nu)_n} \Omega^{\ast}_{D(\nu)_n/A, \gamma(\nu)}.$$
is a quasi-isomorphism.  

As we saw via Proposition \ref{newpd}, we can identify the divided power $A$-algebra $D_{P(\nu),\eta}(J(\nu))$
with $D_{Q,0}(I)$ for a smooth $D_{P,\eta}(J)$-algebra $Q$ equipped with a section $Q \twoheadrightarrow D_{P,\eta}(J)$
having kernel $I$ and admitting an \'etale $D_{P,\eta}(J)$-map $\Spec(Q) \to \mathbf{A}^m_{D_{P,\eta}(J)}$ (since each $P(\nu)$ admits \'etale coordinates
over $P$, due to the hypothesis that $P$ admits \'etale coordinates over $A$); the divided powers $\delta$ are on the ideal
$K := J_D D_{Q,0}(I) + I_D \subset D_{Q,0}(I)$ corresponding to $J(\nu)$ extending the divided powers $\gamma$ on $J_D \subset D_{P,\eta}(J)$. 
By Proposition \ref{dqt}, if $L \subset D_{P,\eta}(J)\langle t_1,\dots,t_m \rangle$ is the ideal pd-generated by $J_D$ and the $t_j$'s then we have 
\begin{equation}\label{magic}
D_{P,\eta}(J)\langle t_1,\dots, t_m \rangle/n! L^{[n]} \simeq D_{Q,0}(I)/n!K^{[n]}.
\end{equation} 
as divided power algebras over $D_{P,\eta}(J)$ compatibly with change in $n$.  

Thus, our task can be reformulated as showing that the natural map
\begin{equation}\label{omegaqism}
 \varprojlim_{n \ge 1} M_n \otimes_{D_n} \Omega^{\ast}_{D_n/A, \gamma} \to
\varprojlim_{n \ge 1} M(\nu)_n \otimes_{D_{P,\eta}(J) \langle t_1, \dots, t_m \rangle/n!L^{[n]}} \,
\Omega^{\ast}_{(D_{P,\eta}(J)\langle t_1,\dots,t_m\rangle/n! L^{[n]})/A, \delta}.
\end{equation}
is a quasi-isomorphism. To get a feeling for this, let's consider a special case (which will also be used in the treatment of the general case).

\begin{example}\label{basicdR} Suppose $\mathscr{F} = \mathscr{O}^{\rm{cris}}_{R/A}$,
so via the isomorphism (\ref{Fomega}) we can identify (\ref{omegaqism}) with the map of
pd-adic completions
$$\widehat{\Omega}^*_{D_{P,\eta}(J)/A,\gamma} \to \widehat{\Omega}^*_{D_{Q,0}(I)/A, \delta}.$$
Assume $J=0$ (so $P=R$) and that $R=A$, so this map is 
$$A[0] \to \widehat{\Omega}^*_{A\langle t_1, \dots, t_m\rangle/A, \gamma}.$$
We are going to show that this is a homotopy equivalence in a manner
that interacts well with pd-adic topologies on the terms $\Omega^{\mu}_{A\langle t_1, \dots, t_m \rangle/A, \gamma}$.

Avoiding pd-adic completion, define the natural maps of $A$-linear complexes
$$f:A[0] \to \Omega^*_{A\langle t_1,\dots,t_m \rangle/A,\gamma},\,\,\,
h:\Omega^*_{A \langle t_1,\dots, t_m \rangle/A,\gamma} \to A[0]$$
for which the first is given in degree 0 by the inclusion $A \to A \langle t_1,\dots,t_m \rangle$
and the second is given in degree 0 by ``projection to the constant term''; these are both maps of complexes.
Clearly $h \circ f$ is the identity on $A[0]$, and we claim that the endomorphism $f \circ h$ of
$\Omega^*_{A \langle t_1, \dots, t_m \rangle/A,\gamma}$ is
homotopic to the identity. 

Letting $\gamma_j$ denote the divided power structure
on $(t_j)_D \subset D_{A[t_j]}((t_j)) =: A \langle t_j \rangle$, 
the isomorphism $$A \langle t_1 \rangle \otimes_A  \dots \otimes_A A \langle t_m \rangle \simeq A \langle t_1, \dots, t_m \rangle$$
of divided power $A$-algebras yields a natural isomorphism of 
 $A$-linear complexes
 $$ \Omega^*_{A \langle t_1 \rangle/A, \gamma_1} \otimes_A \dots \otimes_A \Omega^*_{A \langle t_m \rangle/A, \gamma_m} \simeq
 \Omega^*_{A\langle t_1, \dots, t_m \rangle/A, \gamma}.$$
 This identifies $f$ and $h$ above with tensor products of the analogous maps for the single-variable cases $A\langle t_j \rangle$
 ($1 \le j \le m$), so it suffices to treat the single-variable case.

The complex $\Omega^*_{A \langle t \rangle/A, \gamma}$ is a 2-term complex (in degrees 0 and 1) given by
$$A \langle t \rangle \to A \langle t \rangle {\rm{d}}t$$
defined by $f = \sum a_j t^{[j]} \mapsto {\rm{d}}f = \sum_{j \ge 1} a_j t^{[j-1]} {\rm{d}}t$.
We define the $A$-linear $k: \Omega^1_{A \langle t \rangle/A, \gamma} \to A \langle t \rangle$
by $k(\sum a_j t^{[j]} {\rm{d}}t) = \sum a_j t^{[j+1]}$, so ${\rm{d}}\circ k: \Omega^1_{A \langle t \rangle/A, \gamma} \to \Omega^1_{A \langle  t \rangle/A, \gamma}$
is the identity and $k \circ {\rm{d}}: A \langle t \rangle \to A \langle t \rangle$ is ``subtract off the constant term''. 
But $f \circ h$ in degree 0 is projection to the constant term, so $k \circ {\rm{d}} = 1 - (f \circ h)$ in degree 0.
Thus, $k$ defines a homotopy between $f \circ h$ and the identity on $\Omega^*_{A \langle t \rangle/A, \gamma}$.

Returning to the general multivariable setting, we claim that the $A$-linear $h$ and the $A$-linear homotopy $k$ betwen $f \circ h$ and the identity is
continuous for the termwise pd-adic topologies. Granting this, if $\widehat{f}$ and $\widehat{h}$
denote the induced maps between termwise pd-adically completed complexes
then $\widehat{k}$ induces such a homotopy between $\widehat{f} \circ \widehat{h}$
and the identity, so $\widehat{f}$ has a homotopy-inverse $\widehat{h}$.
The pd-adic continuity of $h$ is clear (this only has content in degree 0),
so let's focus on $k$.  In the single-variable setting, for the ideal $I = (t)_D \subset A \langle t \rangle$
we note that $k(I^{[N]} \Omega^1_{A \langle t \rangle/A, \gamma}) \subset
I^{[N+1]} A \langle t \rangle \subset I^{[N]} A \langle t \rangle$ for all $N \ge 0$.  Also, 
if $(R, J)$ and $(R', J')$ are $A$-flat divided power algebras
over $A$ and we equip $J'' := J \otimes_A R' + R \otimes_A J' \subset R \otimes_A R'$
with its evident divided power structure $\gamma''$ then for any $R$-module $M$ and $R'$-module $M'$
we have
$${J''}^{[N]} (M \otimes_A M') = \sum_{0 \le m \le N} (J^{[m]} \otimes {J'}^{[N-m]}) (M \otimes_A M').$$
Thus, if $T: M_1 \to M_2$ and $T':M'_1 \to M'_2$ are linear maps that satisfy
$T(J^{[n]} M_1) \subset J^{[n]} M_2$ and $T({J'}^{[n']} M'_1) \subset {J'}^{[n']} M'_2$ for all $n, n'$
then the $R \otimes_A R'$-linear map
$$T \otimes T': M_1 \otimes_A M'_1 \to M_2 \otimes_A M'_2$$
carries ${J''}^{[N]} (M_1 \otimes_A M'_1)$ into ${J''}^{[N]} (M_2 \otimes_A M'_2)$ for all $N \ge 0$. 
Applying this repeatedly to higher-order tensor products, we see that for the ideal $I = (t_1, \dots, t_m)_D \subset
A \langle t_1, \dots, t_m \rangle$ and all $\mu > 0$, the $A$-linear map
$$k_{\mu}: \Omega^{\mu}_{A \langle t_1, \dots, t_m \rangle/A, \gamma} \to \Omega^{\mu-1}_{A \langle t_1, \dots, t_m \rangle/A, \gamma}$$
carries $I_D^{[N]} \Omega^{\mu}_{A \langle t_1, \dots, t_m \rangle/A, \gamma}$ into
$I_D^{[N]} \Omega^{\mu-1}_{A \langle t_1, \dots, t_m \rangle/A, \gamma}$.
\end{example}

\begin{remark}\label{caseproof} The method in Example \ref{basicdR} can be used to directly treat
the case of (\ref{omegaqism}) with the trivial crystal $\mathscr{F} = \mathscr{O}^{\rm{cris}}_{R/A}$
but allowing any $J$ and $P$ (e.g., $P$ not necessarily a polynomial ring over $A$), 
though this extra generality is not useful for treating more general crystals
because the intervention of the connection associated to $\mathscr{F}$ and tensor products
over $D_n$'s and $D(\nu)_n$'s (rather than over $A$) are 
genuine complications that don't arise for the trivial crystal (and are overcome using a delicate balancing act between 
filtrations and pd-adic topologies).  Nonetheless, we now provide
a proof for the case of the trivial crystal allowing any $J$ and general $P$  in case the reader may want to see it
worked out directly. 

Let $(A, I_0, \eta)$ be a divided power ring, and $R$ an $A$-algebra equipped with a divided power
structure $(J, \delta)$ compatible with $\eta$ (such as $R = D_{P,\eta}(J)$
with the ideal $J_D$), and $Q$ a smooth $R$-algebra admitting
a section $\sigma: \Spec(R) \hookrightarrow \Spec(Q)$ and an \'etale $R$-map $h:\Spec(Q) \to \mathbf{A}^m_R$.
If $I$ is the ideal of $\sigma$ in $Q$ then $h^{-1}(0)$ contains $\Spec(Q/I)$ as a clopen subscheme, 
so $I/I^2$ is $R$-free due to the \'etaleness of $h$ and the freeness of the cotangent space along the zero-section of affine space.
We want to show that the natural map of pd-adic completions
\begin{equation}\label{omegahat}
\widehat{\Omega}^*_{R/A,\gamma} \to \widehat{\Omega}^*_{D_{Q,0}(I)/A, \delta}
\end{equation} 
is a quasi-isomorphism (which would imply that (\ref{omegaqism}) is a quasi-isomorphism when $\mathscr{F} = \mathscr{O}^{\rm{cris}}_{R/A}$). 

As with the isomorphism (\ref{Fomega}), the termwise pd-adically completed complex $\widehat{\Omega}^*_{D_{Q,0}(I)/A}$ is identified with the complex
$$\varprojlim_{N \ge 1} \Omega^*_{(D_{Q,0}(I)/N! I_D^{[N]})/A,\gamma},$$
and by Proposition \ref{dqt} (which applies since $I/I^2$ is $R$-free) this is identified with 
$$\varprojlim_{N \ge 1} \Omega^*_{(R\langle t_1,\dots,t_m \rangle/N! (t_1,\dots,t_m)^{[N]})/A,\gamma},$$
which in turn is identified with the termwise pd-adically completed complex
$\widehat{\Omega}^*_{R \langle t_1,\dots,t_m \rangle/A, \gamma}$.
In this way, we are reduced to the case $Q = R[t_1,\dots,t_m]$ and $J = (t_1, \dots, t_m)$.
For this special case, the natural $A$-linear map analogous to (\ref{omegahat}) without pd-adic completions is
\begin{equation}\label{omegaRA}
f:\Omega^*_{R/A,\gamma} \to \Omega^*_{R\langle t_1, \dots, t_m \rangle/A,\delta}.
\end{equation}
We shall build an $A$-linear map $h$ in the other direction
so that $h \circ f$ is the identity and
$f \circ h$ is homotopic to the identity via a homotopy $k$ such that $k$ and $h$ are (like $f$) continuous for the termwise pd-adic topologies
and hence pass to the termwise pd-adic completions (so (\ref{omegahat}) is a homotopy equivalence, and hence a quasi-isomorphism). 

Since $R \langle t_1, \dots, t_m \rangle = R \otimes_A A \langle t_1, \dots, t_m \rangle$, so 
$$\Omega^1_{R \langle t_1,\dots,t_m\rangle/A, \delta} = (R \otimes_A \Omega^1_{A \langle t_1, \dots, t_m \rangle, \delta}) \oplus
(A \langle t_1, \dots, t_m \rangle \otimes_A \Omega^1_{R/A, \gamma}),$$
we have 
$$\Omega^*_{R \langle t_1,\dots,t_m \rangle/A, \delta} =
\Omega^*_{A \langle t_1, \dots, t_m \rangle/A, \delta} \otimes_A \Omega^*_{R/A,\gamma}$$
(a tensor product of $A$-linear complexes).  Thus, (\ref{omegaRA}) is obtained as the tensor product
of the $A$-linear complex $\Omega^*_{R/A, \gamma}$ against the map of $A$-linear complexes
\begin{equation}\label{augment}
f:A[0] \to \Omega^*_{A\langle t_1, \dots, t_m\rangle /A, \delta}.
\end{equation}
By Example \ref{basicdR}, the map (\ref{augment}) has a homotopy inverse $h$, where for $\mu > 0$
the homotopy between $f \circ h$ and the identity carries $J_D^{[N]} \Omega^{\mu}_{A\langle t_1, \dots, t_m\rangle /A, \delta}$
into $J_D^{[N]} \Omega^{\mu-1}_{A\langle t_1, \dots, t_m\rangle /A, \delta}$ for all $N \ge 1$. 
Thus, applying the tensor product against $\Omega^*_{R/A,\gamma}$ preserves
$(1 \otimes f) \circ (1 \otimes h)$
having a pd-adically continuous homotopy to the identity, so (\ref{omegahat}) has a homotopy inverse
and hence is a quasi-isomorphism. 
\end{remark}

\subsection{Filtration arguments with pd-adic completions}

Returning to the treatment of the general quasi-isomorphism problem for (\ref{omegaqism}) for an arbitrary but {\em fixed}
$\nu > 0$, we need to define some filtrations
so that we can eventually disentangle the crystal and its connection from suitable de Rham differentials over $A$. Inspired
by the treatment in \cite[Tag 07LD]{SP} for the case of $p$-adic completions for crystals in $p$-adic crystalline cohomology, 
we will define filtrations on the complexes on both sides of (\ref{omegaqism}) that in each homological degree terminate after finitely many steps 
and such that the induced maps between successive quotient complexes are quasi-isomorphisms; that will clearly do the job. 
The pd-adic topologies in our setting are harder to work with than $p$-adic topologies, but we'll  push through a
filtration method anyway. 

On the left side of (\ref{omegaqism}), for $k \ge 0$ we define the subcomplexes $F^k = \varprojlim_{n \ge 1} F^k_n$ where
$$F^k_n := \sigma_{\ge k}(M_n \otimes_{D_n} \Omega^*_{D_n/A,\gamma})$$
is the naive truncation in degrees $\ge k$ (i.e., replace all terms in degree $< k$ with 0). 
By consideration of the Leibnitz rule for connections, 
$$F^k_n/F^{k+1}_n \simeq M_n \otimes_{D_n} \Omega^k_{D_n/A,\gamma}[-k]$$
and the quotient complex $F^k/F^{k+1}$ is the one-term complex consisting of 
$$\varprojlim_{n \ge 1} M_n \otimes_{D_n} \Omega^k_{D_n/A,\gamma}$$
placed in degree $k$. 

To define the appropriate filtration $G^k$ on the right side of (\ref{omegaqism}), we use the behavior of divided power $\Omega^1$'s with respect
to quotient algebras (modulo a divided power ideal) and the isomorphism
$$\Omega^*_{D_{P,\eta}(J) \langle t_1, \dots, t_m \rangle/A, \delta} \simeq
\Omega^*_{D_{P,\eta}(J)/A, \gamma} \otimes_A \Omega^*_{A \langle t_1, \dots, t_m \rangle/A, \delta}$$
to identify the complex
$\Omega^*_{D_{P,\eta}(J)\langle t_1, \dots, t_m \rangle/A, \delta}$ with the tensor product complex
$$\Omega^*_{D_{P,\eta}(J)/A, \gamma} \otimes_A \Omega^*_{A \langle t_1, \dots, t_m \rangle/A, \delta}.$$
Let $\langle \underline{t} \rangle_+ \subset A\langle t_1, \dots, t_m \rangle$ be the ideal pd-generated by the $t_j$'s,
and let 
$$B_n := A \langle t_1, \dots, t_m \rangle/n! \langle \underline{t} \rangle_+^{[n]}.$$
Bringing in the connection arising from the $M(\nu)_n$'s
and using the {\em crystal property} $M(\nu)_n = M_n \otimes_{D_n} D(\nu)_n$
with $D(\nu)_n$ given by (\ref{magic}), the term at level $n$ on the right side of (\ref{omegaqism}) is 
identified with the quotient of the complex
$$(M_n \otimes_{D_n} \Omega^*_{D_n/A,\gamma}) \otimes_A \Omega^*_{B_n/A, \delta}$$
(using the connection on $M_n \otimes_{D_n} \Omega^*_{D_n/A,\gamma}$
and the de Rham differentials on $\Omega^*_{B_n/A, \delta}$) modulo the subcomplex given by the image of 
\begin{eqnarray*}
n!(\sum_{0 < m < n} J_D^{[m]}(M_n \otimes_{D_n} \Omega^*_{D_n/A, \gamma}) \otimes_A
\langle \underline{t} \rangle_+^{[n-m]}(\Omega^*_{B_n/A, \delta})) +\\
n!(J_D^{[m-1]}(M_n \otimes_{D_n} ({\rm{d}}J \wedge \Omega^{\ast-1}_{D_n/A,\gamma})) \otimes_A
\langle \underline{t} \rangle_+^{[n-m]}\Omega^{\ast}_{B_n/A,\delta}) + \\
n!(J_D^{[m]}(M_n \otimes_{D_n} \Omega^*_{D_n/A, \gamma}) \otimes_A
\langle \underline{t} \rangle_+^{[n-1-m]}({\rm{d}}\underline{t} \wedge \Omega^{*-1}_{B_n/A, \delta}))
\end{eqnarray*}

Fortunately the preceding unpleasant subcomplex is not something we'll need to work with.  The reason is that 
for $0 < m < n$ at least one of $m-1$ or $n-1-m$ is $\ge n' := \lfloor n/2 \rfloor  - 1$, so the natural surjective map
$$(M_n \otimes_{D_n} \Omega^*_{D_n/A,\gamma}) \otimes_A \Omega^*_{B_n/A, \delta} \twoheadrightarrow
(M_{n'} \otimes_{D_{n'}} \Omega^*_{D_{n'}/A,\gamma}) \otimes_A \Omega^*_{B_{n}/A, \delta}$$
factors through the quotient map onto the $n$th term on the right side of (\ref{omegaqism}).  Since inverse systems that interweave
through each other have the same inverse limit, we see that 
the inverse limit on the right side of (\ref{omegaqism}) coincides with 
\begin{equation}\label{betterlim}
\varprojlim_{n \ge 1} (M_n \otimes_{D_n} \Omega^*_{D_n/A,\gamma}) \otimes_A \Omega^*_{B_n/A, \delta}.
\end{equation}
Recalling that 

We use this description to define a filtration by subcomplexes $G^k = \varprojlim G^k_n$
where $$G^k_n \subset \sigma_{\ge k} (M_n \otimes_{D_n} \Omega^*_{D_n/A,\gamma}) \otimes_A \Omega^*_{B_n/A, \delta}.$$
(Note that for this to make sense, we are using that the complex at level $n$ in (\ref{betterlim}) is a tensor product of
complexes; i.e., we have decoupled the connection from the de Rham complex $\Omega^*_{B_n/A, \delta}$.) 

From the definitions it is clear that (\ref{omegaqism}) carries $F^k$ into $G^k$, so it suffices to show
that $F^k/F^{k+1} \to G^k/G^{k+1}$ is a quasi-isomorphism for all $k \ge 0$.
Since the transition maps $G^k_{n+1} \to G^k_n$ are surjective, the natural map of complexes 
$$G^k/G^{k+1} \to \varprojlim_{n \ge 1} (G^k_n/G^{k+1}_n)$$
is an isomorphism.  By right exactness of tensor products, 
$$G^k_n/G^{k+1}_n \simeq (M_n \otimes_{D_n} \Omega^k_{D_n/A,\gamma}) \otimes_A \Omega^*_{B_n/A,\delta}[-k];$$
i.e., as a complex this is the tensor product of the $A$-module $M_n \otimes_{D_n} \Omega^k_{D_n/A,\gamma}$
against the shifted $A$-linear de Rham complex $\Omega^{*}_{B_n/A,\delta}[-k]$.   The same argument that showed (\ref{Fomega}) is an isomorphism 
shows that the natural map
\begin{eqnarray*}
\varprojlim_{n \ge 1} (M_n \otimes_{D_n} \Omega^k_{D_n/A,\gamma}) \otimes_A \Omega^{*}_{A\langle t_1,\dots,t_m \rangle/A,\delta}[-k]/
n! \langle \underline{t} \rangle_+^{[n]} \Omega^{*}_{A\langle t_1,\dots,t_m \rangle/A, \delta}[-k] &\to& 
\varprojlim_{n \ge 1} (G^k_n/G^{k+1}_n) \\ &=& G^k/G^{k+1}
\end{eqnarray*} 
is an isomorphism.

Thus, the map $F^k/F^{k+1} \to G^k/G^{k+1}$ is identified with the inverse limit over $n \ge 1$ of the maps
$$M_n \otimes_{D_n} \Omega^k_{D_n/A,\gamma}[-k] \to
 (M_n \otimes_{D_n} \Omega^k_{D_n/A,\gamma}) \otimes_A \Omega^{*}_{A\langle t_1,\dots,t_m \rangle/A,\delta}[-k]/
n! \langle \underline{t} \rangle_+^{[n]} \Omega^*_{A\langle t_1,\dots,t_m \rangle/A, \delta}[-k]$$
is a quasi-isomorphism.   Undoing the $k$-fold shift, this is the tensor product $M_n \otimes_{D_n} \Omega^k_{D_n/A, \gamma}$ 
(over $A$) against the $n$th term in the inverse system
defining the termwise pd-adic completion of
$\Omega^*_{A\langle t_1, \dots, t_m \rangle/A, \delta}$.
Noting that the transition maps
$$M_{n+1} \otimes_{D_{n+1}} \Omega^k_{D_{n+1}/A,\gamma} \to
M_n \otimes_{D_n} \Omega^k_{D_n/A}$$
are surjective, our task is an instance of the following more general problem:
if $\{H_n\}$ is an inverse system of $A$-modules with surjective transition maps $H_{n+1} \to H_n$ then for $H := \varprojlim_{n \ge 1} H_n$
we claim that the natural map
\begin{equation}\label{Hmap}
H[0] = \varprojlim_{n \ge 1} H_n[0] \to \varprojlim_{n \ge 1} H_n \otimes_A \Omega^{*}_{A\langle t_1,\dots,t_m \rangle/A,\delta}/
n! \langle \underline{t} \rangle_+^{[n]} \Omega^*_{A\langle t_1,\dots,t_m \rangle/A, \delta}
\end{equation}
is a quasi-isomorphism.  We will show that this is even a homotopy equivalence.

Since the transition maps $H_{n+1} \to H_n$ is surjective, each map $H \to H_n$ is surjective.
Thus, $H$ is separated and complete for the $A$-linear topology defined by the submodules
$\ker(H \twoheadrightarrow H_n)$.  We can thereby identify the target of (\ref{Hmap})
with the complex
$$H \widehat{\otimes}_A \widehat{\Omega}^*_{A\langle t_1, \dots, t_m \rangle/A, \delta}$$
given termwise by completed tensor products over $A$.  
Such completed tensor products are {\em functorial} over $A$ in continuous maps of $A$-modules
in either tensor factor, so the continuous homotopy realizing
$$A[0] \to \widehat{\Omega}^*_{A \langle t_1, \dots, t_m \rangle/A, \delta}$$
as a continuous homotopy equivalence at the end of Example \ref{basicdR} can be used
to show that (\ref{Hmap}) is also a homotopy equivalence of complexes of (topological) $A$-modules.
This completes the proof of Theorem \ref{dPoincare}.

\medskip\medskip

\section{The affine crystalline site}\label{affzarsite}

Up to now, we have studied the the category ${\rm{Cris}}(R/A)$ for a
divided power ring $(A, I_0, \eta)$ and a finitely generated $A$-algebra $R$ for which $\eta$ extends to $I_0 R$, 
without any use of the Zariski topology.  In effect, we have been considering ${\rm{Cris}}(R/A)$
as a site with the indiscrete topology (the only covers of an object are isomorphisms).
As a warm-up to globalization beyond the affine case,
we now consider a Zariski topology on ${\rm{Cris}}(R/A)$. 

\subsection{Relation of Zariski and indiscrete topologies}\label{relsec}

For any object $(\Spec(B), J, \delta)$ in ${\rm{Cris}}(R/A)$ (i.e., $B$ is an $A$-algebra equipped with a surjection $B \twoheadrightarrow R$
whose kernel $J$ is equipped with divided powers $\delta$ compatible with $\eta$
such that $m J^{[n]}=0$ for some integers $m, n \ge 1$), a {\em Zariski cover} consists of objects
$(\Spec(B_i), J_i, \delta^i)$ where $\Spec(B_i) \subset \Spec(B)$ is an open subscheme and $(J_i, \delta^i)$ is the restriction of $(J, \delta)$. 
We denote by ${\rm{Cris}}(R/A)_{\rm{Zar}}$ the site given by the category ${\rm{Cris}}(R/A)$ equipped with the Zariski topology,
and ${\rm{Cris}}(R/A)_{\rm{ind}}$ denotes the site given by  the category ${\rm{Cris}}(R/A)$ equipped with the indiscrete topology. 
The functor $\mathscr{O}_{R/A}^{\rm{cris}}: (B, J, \delta) \mapsto B$ on ${\rm{Cris}}(R/A)$ is clearly a sheaf for the Zariski topology.  

\begin{remark}
For the consideration of crystalline
sites and associated sheaf categories 
in the rest of this paper, we employ the procedures in \cite[Tag 000H]{SP} to appropriately cut down
${\rm{Cris}}(R/A)$ to a small category in order to handle set-theoretic
issues arising for ``global sections'', derived functors, and so on. 
\end{remark}

Note that ${\rm{Cris}}(R/A)$ admits fiber products:  if 
$$(\Spec(B'), J', \delta') \to (\Spec(B), J, \delta) \leftarrow (\Spec(B''), J'', \delta'')$$
are two maps to an object (so $J$ is the preimage of $J'$ under $B \to B'$
and of $J''$ under $B \to B''$) then for $C = B' \otimes_B B''$ there is an evident map onto $R$ with kernel generated
by the image $I$ of $J' \otimes_B J''$, and $(\Spec(D_{C,\eta}(I)), I_D, \gamma)$ is a fiber product.  This latter
triple belongs to the category ${\rm{Cris}}(R/A)$ because $D_{C,\eta}(I)/I_D = C/I = R$ (final equality
by \cite[3.20(4)]{bogus}) 
and $I_D$ is pd-generated by products of images of $J'$ and $J''$ (so if we pick $m, n \ge 1$
so large that $m {J'}_D^{[n]} = 0$ and $m {J''}_D^{[n]} = 0$ then $m I_D^{[n]} =0$ since
$I_D^{[n]}$ is spanned over $D_{C,\eta}(I)$ by the image of ${J'}^{[n]} \otimes {J''}^{[n]}$). 

The functor 
$$q: {\rm{Cris}}(R/A)_{\rm{ind}} \to {\rm{Cris}}(R/A)_{\rm{Zar}}$$
 given by the identity on objects is clearly continuous, so $q_*$ carries sheaves to sheaves.
 Furthermore, since ${\rm{Cris}}(R/A)$ has both fiber products and equalizers (it is easy to verify the preservation of
 the condition ``$m J^{[n]} = 0$ for some $m, n \ge 1$'' under the natural equalizer construction), 
 by \cite[Tag 00X4]{SP} (and its proof) and \cite[Tag 00X5]{SP} the left adjoint $q^{-1}$ is exact;
 i.e., $q$ defines a morphism of sites ${\rm{Cris}}(R/A)_{\rm{Zar}} \to {\rm{Cris}}(R/A)_{\rm{ind}}$.
  Letting $\mathscr{O}_{R/A}^{{\rm{cris}},{\rm{ind}}}$ and $\mathscr{O}_{R/A}^{{\rm{cris}}, {\rm{Zar}}}$ denote
the sheaf $\mathscr{O}_{R/A}^{\rm{cris}}$ respectively on ${\rm{Cris}}(R/A)_{\rm{ind}}$ and ${\rm{Cris}}(R/A)_{\rm{Zar}}$,
we have $q_*(\mathscr{O}_{R/A}^{{\rm{cris}},{\rm{ind}}}) = \mathscr{O}_{R/A}^{{\rm{cris}},{\rm{Zar}}}$.
Thus, there will be no risk of confusion by writing $\mathscr{O}_{R/A}$ to denote both 
 $\mathscr{O}_{R/A}^{{\rm{cris}},{\rm{ind}}}$ and $\mathscr{O}_{R/A}^{{\rm{cris}},{\rm{Zar}}}$.
 
 \begin{definition}\label{qcohaffine} An $\mathscr{O}_{R/A}$-module $\mathscr{F}$ is {\em quasi-coherent}
 if for any $(\Spec(B), J, \delta)$ in ${\rm{Cris}}(R/A)$ the $\mathscr{O}_B$-module restriction of $\mathscr{F}$ to the affine Zariski site of
 $(\Spec(B), J, \delta)$ is quasi-coherent.
 \end{definition}
 
 \begin{lemma}\label{RpiO}
 For $i > 0$, the sheaves ${\rm{R}}^iq_*(\mathscr{O}_{R/A})$ vanish.
 More generally, if $\mathscr{F}$ is an $\mathscr{O}_{R/A}$-module that is quasi-coherent then ${\rm{R}}^iq_*(\mathscr{F})$ vanishes for all $i > 0$.
 \end{lemma}
 
 \begin{proof}
 By \cite[Exp.\,V, \S5, Prop.\,5.1(1)]{SGA4}, ${\rm{R}}^i q_*(\mathscr{F})$ is the sheafification of the presheaf 
 $$(\Spec(B), J, \delta) \mapsto {\rm{H}}^i_{\rm{Zar}}((B, J, \delta), \mathscr{F})$$
 on  ${\rm{Cris}}(R/A)_{\rm{Zar}}$.
 Such vanishing is proved in  \cite[Tag 07JJ]{SP} (by reducing to the vanishing of higher \v{C}ech cohomology of a quasi-coherent sheaf
 for a finite cover of an affine scheme by affine open subschemes) under a hypothesis related to a prime being nilpotent. But that
 hypothesis is not used in the argument (beyond its role in some foundational definitions that we have replaced
 with the condition $``mJ^{[n]}  = 0$'' for some $m, n \ge 1$), so the same argument applies to our setting. 
 \end{proof}
 
 We define the topoi
 $$(R/A)_{\rm{cris}} = {\rm{Shv}}({\rm{Cris}}(R/A)_{\rm{Zar}}),\,\,\,
 \mathscr{C}_{R/A} = {\rm{Shv}}({\rm{Cris}}(R/A)_{\rm{ind}}).$$
 We define a {\em crystal} of $\mathscr{O}_{R/A}$-modules in $(R/A)_{\rm{cris}}$ exactly as we 
did for $\mathscr{C}_{R/A}$ in Definition \ref{crystal}. 
Since $\Gamma((R/S)_{\rm{cris}}, \cdot) = \Gamma(\mathscr{C}_{R/A}, q_*(\cdot))$, Lemma \ref{RpiO}
gives
\begin{equation}\label{changetop}
{\rm{R}}\Gamma((R/S)_{\rm{cris}}, \mathscr{F}) \simeq {\rm{R}}\Gamma(\mathscr{C}_{R/A}, \mathscr{F})
\end{equation}
in $D^+(A)$ for any quasi-coherent $\mathscr{O}_{R/A}$-module $\mathscr{F}$.

\subsection{Affine Poincar\'e Lemma}

 Motivated by (\ref{changetop}) and the method in \cite[Tag 07JM]{SP}
 for computing cohomology for the indiscrete topology via a weakly final object, we prove:
 
 \begin{proposition}\label{chaoticH}
Let $\mathscr{F}$ be a quasi-coherent
$\mathscr{O}_{R/A}$-module on $(R/A)_{\rm{cris}}$, and let $\pi:P \twoheadrightarrow R$ be an $A$-algebra surjection from a smooth $A$-algebra,
with $J = \ker \pi$.  For $\nu \ge 0$, let $P(\nu) = P^{\otimes_A (\nu+1)}$ and $J(\nu) = \ker(P(\nu) \twoheadrightarrow R)$.
Define $D(\nu)_n = D_{P(\nu),\eta}(J(\nu))/n! J(\nu)_D^{[n]}$, with $J(\nu)_n = J(\nu)_D/n! J(\nu)_D^{[n]}$
and $\gamma(\nu)^n$ the divided power structure on $J(\nu)_n$, and let $M(\nu) = \varprojlim_{n \ge 1} \mathscr{F}(D(\nu)_n, J(\nu)_n, \gamma(\nu)^n)$.

The $A$-linear complex
{\rm{
$${\rm{\mbox{\v{C}}A}}(P \twoheadrightarrow R, \mathscr{F}) = (M(0) \to M(1) \to M(2) \dots)$$
}}
computes ${\rm{R}}\Gamma((R/A)_{{\rm{cris}}}, \mathscr{F})$.
\end{proposition}

When $\mathscr{F}$ is moreover a {\em crystal}, 
Theorem \ref{dPoincare} provides a natural isomorphism in $D^+(A)$
between the \v{C}ech-Alexander complex ${\rm{\mbox{\v{C}}A}}(P \twoheadrightarrow R, \mathscr{F})$ and 
$$\varprojlim_{N \ge 1} \mathscr{F}(D_{P,\eta}(J)/N! J_D^{[N]}, J_D/N! J_D^{[N]}, \gamma) \otimes_{D_{P,\eta}(J)/N! J_D^{[N]}} 
 \Omega^*_{(D_{P,\eta}(J)/N! J_D^{[N]})/A,\gamma} = M \widehat{\otimes}_{\widehat{D}} \widehat{\Omega}^*_{\widehat{D}/A, \gamma},$$
 where 
 $\widehat{D} = \varprojlim D_{P,\eta}(J)/N! J_D^{[N]}$ and 
 $M = \varprojlim M_N$ is the $\widehat{D}$-module with pd quasi-nilpotent $A$-linear connection 
 correspondng to the crystal $\mathscr{F}$ via Proposition \ref{micequiv}. 
 Hence, Proposition \ref{chaoticH} yields a Poincar\'e Lemma for computing the crystalline cohomology of (quasi-coherent) 
crystals in the affine case.  In Theorem \ref{hartshornethm} we will formulate
a global Poincar\'e lemma for crystals, building on the affine case.  (The use of the indiscrete topology
in the proof of Proposition \ref{chaoticH} does not globalize.) 

\begin{proof}
Using (\ref{changetop}), we may and do focus on ${\rm{R}}\Gamma(\mathscr{C}_{R/A}, \mathscr{F})$ 
rather than ${\rm{R}}\Gamma((R/A)_{{\rm{cris}}}, \mathscr{F})$.

Since $P(0) = P$ and $J(0) = J$, to simplify notation we write $D_n$ and $J_n$ to respectively denote $D(0)_n = D_{P,\eta}(J)/n! J_D^{[n]}$
and the ideal $J(0)_n \subset D(0)_n$, with $\gamma^n$ denoting the divided power structure on $J_n$.
For any object $(\Spec(B), J, \delta)$ in ${\rm{Cris}}(R/A)$, we have seen in Remark \ref{crismap}
that for sufficiently large $n$ there exists a morphism $(\Spec(D_n), J_n, \gamma^n) \to (\Spec(B), J, \delta)$. 
This says that for the pd-adic completion $\widehat{D} = \varprojlim_{n \ge 1} D_n$
equipped with its ideal $\widehat{J} = \varprojlim_{n \ge 1} J_n$ have divided power structure $\widehat{\gamma} = \varprojlim \gamma^n$, 
every object in ${\rm{Cris}}(R/A)$
admits a map to $(\Spec(\widehat{D}), \widehat{J}, \widehat{\gamma})$.  However,
the latter triple is not an object in ${\rm{Cris}}(R/A)$, but nonetheless we can define its ``functor of points''
$h_{\widehat{D}}$ on ${\rm{Cris}}(R/A)$ by
$$h_{\widehat{D}}((\Spec(B), I, \delta)) = \varinjlim_{n \ge 1} {\rm{Hom}}_{{{\rm{Cris}}(R/A)}}((\Spec(B), I, \delta), (\Spec(D_n), J_n, \gamma^n)).$$
The preceding shows that the value of $h_{\widehat{D}}$ on any object in ${\rm{Cris}}(R/A)$ is non-empty.   Hence,
the unique map $h_{\widehat{D}} \to \ast$ to the final object in $\mathscr{C}_{R/A}$ is an epimorphism.

Since $\mathscr{C}_{R/A}$ is the topos for the indiscrete topology on ${\rm{Cris}}(R/A)$, so every 
sheaf colimit in $\mathscr{C}_{R/A}$ coincides with the corresponding presheaf colimit, we can write
$$h_{\widehat{D}} = \varinjlim {\rm{Hom}}_{{\rm{Cris}}(R/A)}(\cdot, (\Spec(D_n), J_n, \gamma^n))$$
in $\mathscr{C}_{R/A}$.  In particular, for any $\mathscr{F}$ in $\mathscr{C}_{R/A}$, we have
$${\rm{Hom}}_{\mathscr{C}_{R/A}}(h_{\widehat{D}}, \mathscr{F}) = 
\varprojlim_{n \ge 1} \mathscr{F}((D_n, J_n, \gamma^n)) = M(0).$$

For any $\nu \ge 0$, consider the $(\nu+1)$-fold product $h_{\widehat{D}} \times \dots \times h_{\widehat{D}}$.  
By the universal property of divided power envelopes and the fact that
$(P(\nu), J(\nu))$ is the $(\nu+1)$-fold coproduct of copies of $(P, J)$
in the category of pairs $(C, K)$ consisting of an $A$-algebra $C$ and an ideal $K \subset C$, 
it is easy to check that $h_{\widehat{D}}^{\nu+1}$ coincides with the functor $h_{\widehat{D}(\nu)}$
defined similarly to $h_{\widehat{D}}$ by using $D_{P(\nu),\eta}(J(\nu))$ in place of $D_{P,\eta}(J)$. 
Hence,
$${\rm{Hom}}(h_{\widehat{D}}^{\nu+1}, \mathscr{F}) = {\rm{Hom}}(h_{\widehat{D}(\nu)}, \mathscr{F}) =
\varprojlim_{n \ge 1} \mathscr{F}((D(\nu)_n, J(\nu)_n, \gamma(\nu)^n)) = M(\nu).$$

Since $h_{\widehat{D}} \to \ast$ is an epimorphism, by (the proof of) \cite[Tag 07JM]{SP}
applied to the category $\mathscr{C}_{R/A}$ and the object $h_{\widehat{D}}$ 
it follows that ${\rm{R}}\Gamma(\mathscr{C}_{R/A}, \mathscr{F})$ is computed in the derived
category of $A$-modules by the natural complex 
$${\rm{Hom}}(h_{\widehat{D}}, \mathscr{F}) \to {\rm{Hom}}(h_{\widehat{D}(1)}, \mathscr{F}) \to
{\rm{Hom}}(h_{\widehat{D}(2)}, \mathscr{F}) \to \dots.$$
This is precisely the complex $M(0) \to M(1) \to M(2) \to \dots$.
\end{proof}

\begin{corollary}\label{HFOmega}
Let $\mathscr{F}$ be a quasi-coherent $\mathscr{O}_{R/A}$-module on $(R/A)_{\rm{cris}}$
such that for any morphism $(B', I', \delta') \to (B, I, \delta)$ in $(R/A)_{\rm{cris}}$
with $B' \to B$ surjective, the natural map $\mathscr{F}(B, I, \delta) \to \mathscr{F}(B', I',\delta')$ is surjective. 
Then for all $\mu > 0$ we have
$${\rm{H}}^{\nu}((R/A)_{\rm{cris}}, \mathscr{F} \otimes_{\mathscr{O}_{R/A}} \Omega^{\mu}_{R/A, {\rm{cris}}}) = 0$$
for all $\nu \ge 0$.
\end{corollary}

Note that the surjectivity hypothesis on $\mathscr{F}$ is satisfied by crystals, but is also satisfied
by some quasi-coherent non-crystals (such as $\mathscr{F}(U, T, \delta) = \Omega^1_{T/A, \delta}$).

\begin{proof}
Since $\mathscr{F} \otimes_{\mathscr{O}_{R/A}} \Omega^{\mu}_{R/A, {\rm{cris}}}$ is quasi-coherent, by Proposition \ref{chaoticH} these cohomologies
are the homologies of the evident complex with $\nu$th term
$$\varprojlim_{n \ge 1} \mathscr{F}(D(\nu)_n, J(\nu)_n, \gamma(\nu)^n) \otimes_{D(\nu)_n} \Omega^{\mu}_{D(\nu)_n/A, \gamma(\nu)^n}.$$
But this is precisely the complex $E^{\mu, \ast}$ considered in the proof of Theorem \ref{dPoincare},
where we saw that it is homotopic to 0 for $\mu > 0$: the only aspect of the crystal property used in that part of the proof
is the behavior with respect to ring surjections as is being assumed here.
\end{proof}

\medskip\medskip

\section{The global crystalline site}

\subsection{Crystals and connections beyond the affine case}

Now we globalize, beginning wth a definition similar to \cite[Tag 07IG]{SP}.

\begin{definition}\label{globalsite}
Let $(S, \mathscr{I}_0, \eta)$ be a scheme equipped with a quasi-coherent ideal sheaf and a divided power structure 
on it.
Let $S_0 \subset S$ be the zero-scheme of $\mathscr{I}_0$, and $X$ an $S$-scheme
such that $\eta$ extends to $\mathscr{I}_0 \mathscr{O}_X$ $($as happens if $X$ is $S$-flat
or $\mathscr{I}_0$ is locally monogenic or $X$ is an $S_0$-scheme$)$.  Define the category 
${\rm{Cris}}(X/S)$ to consist of triples $(U, T, \delta)$
where $U \subset X$ is an open subscheme, $j:U \hookrightarrow T$ is a closed immersion 
and $\delta$ is a divided power structure on the quasi-coherent ideal $\mathscr{J}$ of $U$ in $\mathscr{O}_T$
such that:
\begin{itemize}
\item[(i)] there is an open cover $\{T_{\alpha}\}$ of $T$ and integers $n_{\alpha} \ge 1$
such that $n_{\alpha}! (\mathscr{J}|_{T_{\alpha}})^{[n_{\alpha}]} = 0$ $($so $U \to T$ is a homeomorphism$)$, 
\item[(ii)] $\delta$ is compatible with $\eta$ (i.e., a divided power structure exists on $\mathscr{J} + \mathscr{I}_0 \mathscr{O}_T$
extending $\delta$ on $\mathscr{J}$ and $\eta$ on $\mathscr{I}_0$).
\end{itemize}

This category is equipped with the Zariski topology on objects (i.e., use Zariski covers of each $T$), and as such we
call it the {\em crystalline site} for $X$ over $S$.  The corresponding topos is denoted $(X/S)_{\rm{cris}}$,
and we define the sheaf of rings $\mathscr{O}_{X/S}$ by $(U, T, \delta) \mapsto \Gamma(T, \mathscr{O}_T)$. 
\end{definition}

The category ${\rm{Cris}}(X/S)$ admits fiber products via quasi-coherent divided power envelopes
(similarly to what we have seen for ${\rm{Cris}}(R/A)$). 
There is an evident forgetful functor $u: {\rm{Cris}}(X/S) \to {\rm{Zar}}(X)$ via $(U, T, \delta) \mapsto U$, and
exactly as in \cite[Tag 07IL]{SP} this yields a morphism of sites
$$u_{X/S}: (X/S)_{\rm{cris}} \to {\rm{Shv}}(X_{\rm{Zar}}).$$
Moreover, $(X/S)_{\rm{cris}}$ is functorial in $X \to S$ and the divided power structure $\eta$ on the ideal of $S_0$ in $S$, 
by the same procedure as in \cite[Tag 07IK]{SP}; this is compatible with $u_{X/S}$ (cf.\,\cite[Tag 07MH]{SP}).

\begin{definition}\label{qcohcrys} An $\mathscr{O}_{X/S}$-module $\mathscr{F}$ is {\em quasi-coherent} if
for any $(U, T, \delta) \in {\rm{Cris}}(X/S)$, the $\mathscr{O}_T$-module $\mathscr{F}_T$ on $T_{\rm{Zar}}$
defined by $\mathscr{F}$ is quasi-coherent.  (This is equivalent to the quasi-coherence for sheaves
of modules on a ringed site in \cite[Tag 03DL]{SP} applied to $\mathscr{O}_{X/S}$-modules.)

A {\em crystal} in $(X/S)_{\rm{cris}}$ is a quasi-coherent $\mathscr{O}_{X/S}$-module $\mathscr{F}$ such that for any 
$f: (U', T', \delta') \to (U, T, \delta)$
in ${\rm{Cris}}(X/S),$ the natural $\mathscr{O}_{T'}$-linear map
$f^{\ast}(\mathscr{F}_T) \to \mathscr{F}_{T'}$ is an isomorphism. 
\end{definition}

\begin{remark}\label{localqc}
There is a general notion of ``quasi-coherent sheaf'' on a ringed site
defined in \cite[Tag 03)H]{SP}, and this agrees with 
Definition \ref{qcohcrys}. 
\end{remark}

Suppose there exists a closed immersion $j: X \hookrightarrow Y$ of $X$ into a smooth $S$-scheme $Y$,
and let $D_{Y,\eta}(X) \to Y$ be the affine map corresponding to 
the quasi-coherent divided power  $\mathscr{O}_Y$-algebra $\mathscr{D}_{\mathscr{O}_Y,\eta}(\mathscr{O}_X)$ 
equipped with its
quasi-coherent divided power ideal $(\mathscr{J}, \delta)$ compatible with $\eta$.  The affine $Y$-scheme
$$D_{Y,\eta}(X)_n = \Spec_Y(\mathscr{D}_{\mathscr{O}_Y,\eta}(\mathscr{O}_X)/n! \mathscr{J}^{[n]})$$
contains $D_{Y,\eta}(X)_1 = X$ as a closed subscheme with the same underlying topological space
(the equality with $X$ is due to $\eta$ extending to $\mathscr{I}_0 \mathscr{O}_X$).
The triple $(X, D_{Y,\eta}(X)_n, \delta)$ is an object in ${\rm{Cris}}(X/S)$.

To globalize Proposition \ref{micequiv} beyond the affine case, we just need to globalize
a definition:

\begin{definition} Let $\mathscr{F}$ be a quasi-coherent $\mathscr{O}_{X/S}$-module, and $(\nabla_n)$ a compatible system 
of integrable (divided power) connections
on the Zariski sheaves $\mathscr{F}_n = \mathscr{F}_{D_{Y,\eta}(X)_n}$ over $S$.  
We say $(\nabla_n)$ is {\em pd quasi-nilpotent}
if for each affine open $V = \Spec(P) \subset Y$ that admits an \'etale $S$-map to some $\mathbf{A}^d_S$
and the resulting expression $\nabla_n(s) = \sum_j \theta_{j,n}(s)\otimes {\rm{d}}t_j$ for $s \in \mathscr{F}_n(V)$,
we have $\theta_{j,n}^k(s) \in n! \mathscr{F}_n(V)$ for all $1 \le j \le d$ and each $n \ge 1$ 
with sufficiently large $k$ $($depending on $n$ and $s$$)$. 
\end{definition}

By Remark \ref{pdqindep} the notion of pd quasi-nilpotence is coordinate-independent
and so it is sufficient to check it using one choice of affine open cover of $X$ by such $V$'s since we can work Zariski-locally over $S$
(and thereby invoke the equivalence in Proposition \ref{micequiv}).
Since $\mathscr{F} = \mathscr{O}_{X/S}$ is a crystal, and its associated connection on
each $\mathscr{D}_{Y,\eta}(X)_n$ is the (divided power) de Rham differential, that differential 
is always pd quasi-nilpotent.

\begin{remark}\label{qnconn}
In the special case that $X$ is $S$-smooth and we choose $Y=X$, then $\mathscr{F}_n = \mathscr{F}_X$
for all $n$, so the pd quasi-nilpotence condition then says that when writing $\nabla|_V = \sum \theta_j \otimes {\rm{d}}t_j$
for an affine open $V \subset X$ admitting \'etale coordinates over $S$, we have $\theta_j^k(s) \in n! \mathscr{F}(V)$
for all $s \in \mathscr{F}(V)$ and large $k$ (depending on $n$ and $s$).  This is called {\em quasi-nilpotence}
since the only divided power structures intervening are the ones uniquely extending $\eta$ on flat $S$-schemes. 

In the special case of the crystal $\mathscr{F} = \mathscr{O}_{X/S}$, so the associated
connection is the de Rham differential $\mathscr{O}_X \to \Omega^1_{X/S}$, the quasi-nilpotent condition 
expresses the concrete fact that $(\partial/\partial t_j)^p$ vanishes on any \'etale algebra over a polynomial ring
over an $\mathbf{F}_p$-algebra.
\end{remark}

\begin{proposition}\label{globalmic}
 The category of crystals in $(X/S)_{\rm{cris}}$ is equivalent to the category of inverse systems
$(\mathscr{M}_n, \nabla_n)$ where: 
$\mathscr{M}_n$ is quasi-coherent over $\mathscr{D}_n = \mathscr{D}_{\mathscr{O}_Y,\eta}(\mathscr{O}_X)/n! \mathscr{J}^{[n]}$
with $\mathscr{D}_n \otimes_{\mathscr{D}_{n+1}} \mathscr{M}_{n+1} \to \mathscr{M}_n$
an isomorphism for all $n \ge 1$, and $(\nabla_n)$ is a pd quasi-nilpotent compatible system
of integrable $($divided power$)$ connections on the $\mathscr{M}_n$'s over $S$.

In particular, if $X \to S$ is smooth $($so we can use $Y=X$; note that $\eta$ extends to $\mathscr{I}_0 \mathscr{O}_X$
for all $S$-smooth schemes due to $S$-flatness$)$, then the category of crystals in $(X/S)_{\rm{cris}}$
is equivalent to the category of pairs $(\mathscr{M}, \nabla)$ 
where $\mathscr{M}$ is quasi-coherent sheaf on $X$ and $\nabla: \mathscr{M} \to
\mathscr{M} \otimes_{\mathscr{O}_X} \Omega^1_{X/S}$ is a quasi-nilpotent integrable connection on $\mathscr{M}$ over $S$.
\end{proposition}

\begin{proof}
This is immediate from the equivalence in the affine case from Proposition \ref{micequiv} (see
Remark \ref{indepq}).
\end{proof}

\begin{remark}
The analogous equivalence result in the $p$-adic crystalline setting is \cite[Tag 07JH]{SP}.
The  pd-adic completion process used here is more involved than $p$-adic completion since it isn't an ideal-adic completion and the ideals involved aren't
finitely generated in general.
\end{remark}

\subsection{Global comparison: statement and construction}

The following result is a generalization of a torsion-level version of \cite[Thm.\,3.6]{BdJ}, 
though our method of proof is different in some respects
(e.g., to handle some issues with pd-adic completions that don't arise for $p$-adic completions). 

To streamline the notation, we write $D_Y(X)$ rather than $D_{Y,\eta}(X)$. 

\begin{theorem}\label{hartshornethm} 
For any quasi-coherent crystal $\mathscr{F}$ in $(X/S)_{\rm{cris}}$, 
we have naturally
$${\rm{R}}\Gamma((X/S)_{\rm{cris}}, \mathscr{F}) \simeq
{\rm{R}}\Gamma(X_{\rm{Zar}}, \mathscr{F} \widehat{\otimes}_{\mathscr{O}^{\wedge}_{D_Y(X)}} \widehat{\Omega}^*_{D_Y(X)^{\wedge}/S,\delta})$$
where we define
\begin{equation}\label{bigdef}
\mathscr{F} \widehat{\otimes}_{\mathscr{O}^{\wedge}_{D_Y(X)}} \widehat{\Omega}^*_{D_Y(X)^{\wedge}/S,\delta} :=
\varprojlim_{n \ge 1} \mathscr{F}_{D_Y(X)_n} \otimes_{\mathscr{O}_{D_Y(X)_n}} \Omega^*_{D_Y(X)_n/S,\delta}
\end{equation}
as a complex of abelian sheaves on $X_{\rm{Zar}}$ $($using that $D_Y(X)_n$ has the same underlying space as $X$$)$.
\end{theorem}

\begin{remark}\label{DXA}
If $S = \Spec(A)$ is affine then (\ref{bigdef}) is even an isomorphism in the derived category of sheaves of $A$-modules
on $X_{\rm{Zar}}$; the interested reader can check that the proof as given keeps track of the $A$-linear structure
throughout. 
\end{remark}

\begin{remark}\label{Hrem}
When applied with $S$ a $\mathbf{Q}$-scheme (or even more specifically the spectrum of a field of characteristic 0), this result 
recovers \cite[Ch.\,II, \S1, Thm.\,1.4]{hartshorne} 
on de Rham cohomology of varieties with general singularities
over a field of characteristic 0, 
since for $\mathbf{Q}$-schemes $S$ the site ${\rm{Cris}}(X/S)$
coincides with the infinitesimal site of $X$. 
\end{remark}

In the special case $Y=X$ (so $X$ is $S$-smooth), Theorem \ref{hartshornethm} says:

\begin{corollary}\label{globalaffbase}  If $X \to S$ is a smooth map of schemes 
then for any quasi-coherent $\mathscr{O}_X$-module $\mathscr{F}$
equipped with a quasi-nilpotent integrable connection $\nabla$ over $S$, and $\mathscr{F}_{\rm{cris}}$
the associated crystal on ${\rm{Cris}}(X/S)$, 
$${\rm{R}}\Gamma((X/S)_{\rm{cris}}, \mathscr{F}_{\rm{cris}}) 
\simeq {\rm{R}}\Gamma(X_{\rm{Zar}}, \mathscr{F} \otimes_{\mathscr{O}_X} \Omega^*_{X/S})$$
where $\mathscr{F} \otimes_{\mathscr{O}_X} \Omega^*_{X/S}$ is made into a complex via $\nabla$.

In particular, if $X$ is smooth and proper over $\Spec(A)$ for a noetherian $A$ and $(\mathscr{F}, \nabla)$ is 
a coherent sheaf on $X$ equipped with a quasi-nilpotent integrable connection over $A$ 
then ${\rm{H}}^i((X/S)_{\rm{cris}}, \mathscr{F}_{\rm{cris}})$ is a finitely generated $A$-module
that vanishes for $i > 2d$ with $d = \max_{s \in S} \dim(X_s)$. 
\end{corollary}

The proof of Theorem \ref{hartshornethm} is a bit long. In this section we focus on constructing a natural map 
\begin{equation}\label{alphadef}
\alpha: {\rm{R}}\Gamma((X/S)_{\rm{cris}}, \mathscr{F}) \to 
{\rm{R}}\Gamma(X_{\rm{Zar}}, \mathscr{F} \widehat{\otimes}_{\mathscr{O}^{\wedge}_{D_Y(X)}} \widehat{\Omega}^*_{D_Y(X)^{\wedge}/S,\delta})
\end{equation}
that is functorial in $X$ (over $(S, \mathscr{I}_0, \eta)$). In \S\ref{globalproof} we will use 
a \v{C}ech-to-derived spectral sequence argument to reduce the quasi-isomorphism property to the case of affine $S$, $X$, and $Y$.
This task in the affine case
will be settled using (the proof of) Theorem \ref{dPoincare}. 

Motivated by \cite[Thm.\,3.2]{BdJ}, we now use a similar argument to loc.\,cit.\,to see that applying ${\rm{R}}u_{X/S,\ast}$ to the natural map 
$$\mathscr{F} \otimes_{\mathscr{O}_{X/S}} \Omega^*_{X/S, {\rm{cris}}} \to \mathscr{F}[0]$$
yields a quasi-isomorphism.  It suffices to show that ${\rm{R}}u_{X/S,\ast}(\mathscr{F} \otimes_{\mathscr{O}_{X/S}}
\Omega^{\mu}_{X/S,{\rm{cris}}}) = 0$ for all $\mu > 0$.  Its homology in degree $\nu \ge 0$ is the sheafification of the presheaf
whose value on affine objects $(\Spec(R), J, \delta)$ over an open affine $\Spec(A) \subset S$ (with $\delta$
compatible with $\eta$) is
$${\rm{H}}^{\nu}((\Spec(R)/S)_{\rm{cris}}, \mathscr{F} \otimes \Omega^{\mu}_{X/S, {\rm{cris}}}) =
{\rm{H}}^{\nu}((R/A)_{\rm{cris}}, \mathscr{F}_R \otimes_{\mathscr{O}_{R/A}^{\rm{cris}}} \Omega^{\mu}_{R/A, {\rm{cris}}}),$$
where $\mathscr{F}_R$ is the quasi-coherent restriction of $\mathscr{F}$ to ${\rm{Cris}}(R/A)$ (this is the full subcategory of affine objects in
${\rm{Cris}}(\Spec(R)/\Spec(A))$) equipped with the Zariski topology.   The right side vanishes for all $\mu > 0$
(and all $\nu \ge 0$) by Corollary \ref{HFOmega}; that corollary imposes a ``surjectivity'' hypothesis on $\mathscr{F}$ that is satisfied because
$\mathscr{F}$ is a crystal. 

We conclude from the commutative diagram
$$\xymatrix{
{\rm{R}}\Gamma((X/S)_{\rm{cris}}, \mathscr{F}) \ar@{=}[r] & {\rm{R}}\Gamma(X_{\rm{Zar}}, {\rm{R}}u_{X/S,\ast}(\mathscr{F})) \\
{\rm{R}}\Gamma((X/S)_{\rm{cris}}, \mathscr{F} \otimes_{\mathscr{O}_{X/S}} \Omega^*_{X/S,{\rm{cris}}}) \ar[u] \ar@{=}[r] & 
{\rm{R}}\Gamma(X_{\rm{Zar}}, {\rm{R}}u_{X/S,\ast}(\mathscr{F} \otimes_{\mathscr{O}_{X/S}} \Omega^*_{X/S,{\rm{cris}}})) \ar[u]^-{\simeq}}
$$
that the left side is an isomorphism in $D(A)$; we write this as
\begin{equation}\label{dagger}
{\rm{R}}\Gamma((X/S)_{\rm{cris}}, \mathscr{F}) \simeq {\rm{R}}\Gamma((X/S)_{\rm{cris}}, \mathscr{F} \otimes_{\mathscr{O}_{X/S}} \Omega^*_{X/S,{\rm{cris}}}).
\end{equation}
Note that this map naturally goes from right to left.

To make contact with the ``termwise pd-adic completion'' (\ref{bigdef}) of $\mathscr{F} \otimes_{\mathscr{O}_{X/S}} \Omega^*_{X/S,{\rm{cris}}}$, 
we now consider the direct limit sheaf $h_{\widehat{D}} := \varinjlim h_{(X, D_Y(X)_n, \delta)}$ in $(X/S)_{\rm{cris}}$.
Using the map $h_{\widehat{D}} \to \ast$ to the final object of the topos, we have a natural transformation
${\rm{Hom}}(\ast, \cdot) \to {\rm{Hom}}(h_{\widehat{D}}, \cdot)$ of functors
${\rm{Ab}}((X/S)_{\rm{cris}}) \to {\rm{Ab}}$.  Passing to derived functors, this defines naturally
$${\rm{R}}\Gamma((X/S)_{\rm{cris}}, \mathscr{G}^{\bullet}) \to {\rm{R}}\Gamma(h_{\widehat{D}}, \mathscr{G}^{\bullet})$$
for $\mathscr{G}^{\bullet}$ in $D^+({\rm{Ab}}((X/S)_{\rm{cris}}))$.  
For $\mathscr{F}$ as in the statement of the theorem, we therefore have a natural map
\begin{equation}\label{2dagger}
b_{\mathscr{F}}: {\rm{R}}\Gamma((X/S)_{\rm{cris}}, \mathscr{F} \otimes_{\mathscr{O}_{X/S}} \Omega^*_{X/S,{\rm{cris}}}) \to 
{\rm{R}}\Gamma(h_{\widehat{D}}, \mathscr{F} \otimes_{\mathscr{O}_{X/S}} \Omega^*_{X/S,{\rm{cris}}}).
\end{equation}
For any $\mathscr{G}^{\bullet} \in D^+({\rm{Ab}}((X/S)_{\rm{cris}}))$
such that all maps $\mathscr{G}^{\bullet}_{D_Y(X)_{n+1}} \to \mathscr{G}^{\bullet}_{D_Y(X)_n}$ of complexes of abelian sheaves on $X_{\rm{Zar}}$
are surjective termwise (e.g., all $\mathscr{G}^i$ quasi-coherent,
such as $\mathscr{G}^{\bullet} = \mathscr{F} \otimes_{\mathscr{O}_{X/S}} \Omega^*_{X/S,{\rm{cris}}}$), 
we will build a natural isomorphism
\begin{equation}\label{3dagger}
{\rm{R}}\Gamma(h_{\widehat{D}}, \mathscr{G}^{\bullet}) \simeq {\rm{R}}\Gamma(X_{\rm{Zar}}, \varprojlim \mathscr{G}^{\bullet}_{D_Y(X)_n}).
\end{equation}
Composing (\ref{dagger}), (\ref{2dagger}), and (\ref{3dagger}) would then define (\ref{alphadef}), and its quasi-isomorphism property
amounts to the same for $b_{\mathscr{F}}$ (which we'll prove at the end).  Note that the proof of the quasi-isomorphism property for
(\ref{dagger}) used Corollary \ref{HFOmega}. 

To construct (\ref{3dagger}), we now analyze ${\rm{R}}\Gamma(h_{\widehat{D}},\cdot)$.
As a functor on ${\rm{Ab}}((X/S)_{\rm{cris}})$ we have 
$$\Gamma(h_{\widehat{D}}, \cdot) = \varprojlim \Gamma((X,D_Y(X)_n,\delta), \cdot).$$
A general construction with abelian sheaves in any topos
given in \cite[2.3]{BdJ} gives that 
$${\rm{R}}\Gamma(h_{\widehat{D}}, \cdot) = {\rm{R}}\varprojlim({\rm{R}}\Gamma(\widehat{D}, \cdot))$$
as functors $D^+({\rm{Ab}}((X/S)_{\rm{cris}})) \to D^+({\rm{Ab}})$, where
${\rm{R}}\Gamma(\widehat{D},\cdot)$ denotes the total derived functor of the functor 
$$\mathscr{G} \mapsto \Gamma(\widehat{D}, \mathscr{G}) := (\Gamma((X,D_Y(X)_n, \delta), \mathscr{G}))_{n \ge 1}$$
from the category ${\rm{Ab}}((X/S)_{\rm{cris}})$ to the category of inverse systems of abelian groups. 
The key point now is to rewrite ${\rm{R}}\Gamma(\widehat{D}, \mathscr{G}^{\bullet})$ on bounded-below complexes $\mathscr{G}^{\bullet}$ 
in more convenient terms.  

For any $\mathscr{G}$ in ${\rm{Ab}}((X/S)_{\rm{cris}})$ we have
$$\Gamma(\widehat{D}, \mathscr{G}) = (\Gamma(X_{\rm{Zar}}, \mathscr{G}_{D_Y(X)_n}))_{n \ge 1}$$
as an inverse system of abelian groups, where $\mathscr{G}_{D_Y(X)_n}$ is the abelian sheaf on $(D_Y(X)_n)_{\rm{Zar}} = X_{\rm{Zar}}$
defined by $\mathscr{G}$ (using the object $(X, D_Y(X)_n, \delta) \in {\rm{Cris}}(X/S)$).  
In other words, $\Gamma(\widehat{D}, \cdot) = F_1 \circ F_2$ where $F_2: {\rm{Ab}}((X/S)_{\rm{cris}}) \to {\rm{Ab}}(X_{\rm{Zar}})^{\mathbf{N}}$
associates to any $\mathscr{G}$ the inverse system $(\mathscr{G}_{D_Y(X)_n})_{n \ge 1}$ of abelian sheaves on $X_{\rm{Zar}}$
and $F_1:{\rm{Ab}}(X_{\rm{Zar}})^{\mathbf{N}} \to {\rm{Ab}}^{\mathbf{N}}$ is the functor
carrying an inverse system of abelian sheaves on $X_{\rm{Zar}}$ to the inverse system of global sections.

The functor $F_2$ is clearly exact, and we claim that it carries injectives to $F_1$-acyclics. 
For any injective $\mathscr{I}$ in ${\rm{Ab}}((X/S)_{\rm{cris}})$ and object $(U, T, \delta)$ in ${\rm{Cris}}(X/S)$,
the associated abelian sheaf $\mathscr{I}_T$ on $T_{\rm{Zar}}$ is injective because 
the restriction of $\mathscr{I}$ to the slice category over $(U, T, \delta)$ is injective
and the functor assigning to any abelian sheaf on the slice category over $(U, T, \delta)$ the associated
Zariski sheaf on $T$ has an exact left adjoint. 

We have shown that in the inverse system $(\mathscr{I}_{D_Y(X)_n})_{n \ge 1}$ all members
are injective as abelian sheaves on the common Zariski site $(D_Y(X)_n)_{\rm{Zar}} = X_{\rm{Zar}}$.
We want to show that this inverse system is acyclic for $F_1$. For that purpose, we shall 
compute the derived functors of $F_1$.  There is a natural guess: the collection of functors
$$(\mathscr{G}_n)_{n \ge 1} \mapsto ({\rm{H}}^j(X_{\rm{Zar}}, \mathscr{G}_n))_{n \ge 1}$$
for $j \ge 0$ form a $\delta$-functor in an evident manner and recover $F_1$ for $j=0$. This really is the derived
functor of $F_1$ because it is erasable: the functors for $j > 0$ vanish on inverse systems of flasque sheaves,
so the Godement {\em canonical} flasque construction does the job. 

Having confirmed that $F_2(\mathscr{I})$ is $F_1$-acyclic, since $F_2$ is exact we conclude that
$${\rm{R}}\Gamma(\widehat{D}, \mathscr{G}^{\bullet}) = {\rm{R}}F_1(F_2(\mathscr{G})) = {\rm{R}}F_1((\mathscr{G}^{\bullet}_{D_Y(X)_n})_{n \ge 1}).$$
Thus, 
\begin{equation}\label{star}
{\rm{R}}\Gamma(h_{\widehat{D}}, \mathscr{G}^{\bullet}) 
\simeq {\rm{R}}\varprojlim({\rm{R}}F_1((\mathscr{G}^{\bullet}_{D_Y(X)_n}))).
\end{equation}

\begin{remark}\label{SPrem} Beware that for our purposes, ${\rm{R}}\varprojlim$ denotes an actual derived functor
$D^+(\mathcal{A}^{\mathbf{N}}) \to D^+(\mathcal{A})$ for an abelian category $\mathcal{A}$ with enough injectives
(and with countable products, so it admits inverse limits), whereas
in \cite{SP} the same notation is used with a different meaning: a functor
$D(\mathcal{A})^{\mathbf{N}} \to D(\mathcal{A})$ defined in \cite[Tag 0BK7]{SP}.
To avoid confusion, we denote the latter functor as ${\rm{R}}\varprojlim_{\rm{SP}}$; we will not use it. 
\end{remark}

In (\ref{star}) we are working with ${\rm{R}}\varprojlim \circ {\rm{R}}F_1$, where $F_1 = G^{\mathbf{N}}: {\rm{Ab}}((X/S)_{\rm{cris}})^{\mathbf{N}} \to
{\rm{Ab}}^{\mathbf{N}}$ is the termwise application of $G=\Gamma(X_{\rm{Zar}},\cdot)$.
We claim that ${\rm{R}}\varprojlim$ passes through derived functors of such a form $G^{\mathbf{N}}$ rather generally, as follows:

\begin{lemma}\label{GN} Let $G: \mathcal{A} \to \mathcal{B}$ is a left-exact functor between Grothendieck abelian categories
admitting countable products so inverse limits exist in each caegory.
Assume also that $G$ commutes with inverse limits $($in the sense that
$G \circ \varprojlim = \varprojlim \circ G^{\mathbf{N}}$$)$
and that $G$ has an exact left adjoint.

The functors $D^+(\mathcal{A}^{\mathbf{N}}) \to D^+(\mathcal{B})$
defined by 
$${\rm{R}}\varprojlim \circ {\rm{R}}(G^{\mathbf{N}}),\,\,\,{\rm{R}}G \circ {\rm{R}}\varprojlim$$
are naturally isomorphic.  
\end{lemma}

The hypotheses on the abelian categories are satisfied by ${\rm{Ab}}(X_{\rm{Zar}})$ and ${\rm{Ab}}$,
and the hypotheses on $G$ are satisfied by $\Gamma(X_{\rm{Zar}},\cdot)$.
An analogous result is proved in \cite[Tag 08U1]{SP}
for ${\rm{R}}\varprojlim_{\rm{SP}}$ (see Remark \ref{SPrem}), using $({\rm{R}}G)^{\mathbf{N}}: D(\mathcal{A})^{\mathbf{N}} \to D(\mathcal{B})^{\mathbf{N}}$
rather than ${\rm{R}}(G^{\mathbf{N}})$. 

\begin{proof}
The functor $\varprojlim: \mathcal{A}^{\mathbf{N}} \to \mathcal{A}$ carries injectives to injectives because
it has as a left adjoint the exact functor assigning to any object $C \in \mathcal{A}$ the constant inverse system with object $C$ at every level.
Hence, $${\rm{R}}G \circ {\rm{R}} \varprojlim = {\rm{R}}(G \circ \varprojlim) = {\rm{R}}(\varprojlim \circ G^{\mathbf{N}}).$$

By hypothesis $G$ has an exact left adjoint, say $F: \mathcal{B} \to \mathcal{A}$.
It is easy to check that $G^{\mathbf{N}}: \mathcal{A}^{\mathbf{N}} \to \mathcal{B}^{\mathbf{N}}$ has a left adjoint given by
$K = (K_m)_{m \ge 1} \mapsto (F(K_m))_{m \ge 1}$, and this is exact because $F$ is exact. 
Thus, $G$ carries injectives to injectives and hence
${\rm{R}}(\varprojlim \circ G^{\mathbf{N}}) = {\rm{R}}\varprojlim \circ {\rm{R}}(G^{\mathbf{N}})$, so the desired isomorphism
of functors is established.  
\end{proof}

Combining (\ref{star}) and Lemma \ref{GN}, we have shown 
\begin{equation}\label{0}
{\rm{R}}\Gamma(h_{\widehat{D}}, \mathscr{G}^{\bullet}) \simeq 
{\rm{R}}\Gamma(X_{\rm{Zar}}, {\rm{R}}\varprojlim((\mathscr{G}^{\bullet}_{D_Y(X)_n})_{n \ge 1}))
\end{equation}
for $\mathscr{G}^{\bullet} \in D^+({\rm{Ab}}((X/S)_{\rm{cris}}))$.
We next aim to show that the ${\rm{R}}\varprojlim$ on the right side of (\ref{0}) is just the inverse limit complex
when 
$$\mathscr{G}^i_{D_Y(X)_{n+1}} \to \mathscr{G}^i_{D_Y(X)_n}$$
is surjective in ${\rm{Ab}}(X_{\rm{Zar}})$ for all $i$ and all $n$ (e.g., all $\mathscr{G}^i$ quasi-coherent).

\begin{lemma}
Let $(\mathscr{G}_n)$ be an inverse system of abelian sheaves on a topological space $X$.
Assume there exists a basis $\mathscr{B}$ for the topology such that for all $U \in \mathscr{B}$ the following are satisfied:
\begin{itemize}
\item[(i)] $\varprojlim^1 (\mathscr{G}_n(U))=0$ $($e.g., all transition maps $\mathscr{F}_{n+1}(U) \to \mathscr{F}_n(U)$ are surjective$)$, 
\item[(ii)] ${\rm{H}}^j(U, \mathscr{G}_n) = 0$ for all $j > 0$ and all $n$.
\end{itemize}
Then $\mathscr{G}_n$ is $\varprojlim$-acyclic (i.e., ${\rm{R}}^k \varprojlim (\mathscr{G}_n) = 0$ for all $k > 0$).

In particular, for any bounded-below complex $(\mathscr{G}^{\bullet}_n)$ of inverse systems 
of quasi-coherent sheaves on a scheme such that the transition maps $\mathscr{G}^i_{n+1} \to \mathscr{G}^i_n$ are surjective
for each $i$ and $n$, we have ${\rm{R}}\varprojlim((\mathscr{G}^{\bullet}_n)) = \varprojlim_n \mathscr{G}^{\bullet}_n$ 
$($the complex given by forming termwise inverse limits$)$. 
\end{lemma}

\begin{proof}
The key idea is to pass from sheaves to presheaves: 
for the category ${\rm{AbPsh}}(X)$ of abelian presheaves on $X$,
we factor the functor $\varprojlim: {\rm{Ab}}(X)^{\mathbf{N}} \to {\rm{Ab}}(X)$ as a composition of left-exact functors
$${\rm{Ab}}(X)^{\mathbf{N}} \stackrel{G=T^{\mathbf{N}}}{\to} {\rm{AbPsh}}(X)^{\mathbf{N}} \stackrel{\varprojlim_{\rm{P}}}{\longrightarrow}
{\rm{AbPsh}}(X) \stackrel{F}{\to} {\rm{Ab}}(X)$$
where $T$ is the forgetful functor from sheaves to presheaves, $\varprojlim_{\rm{P}}$ is the inverse limit functor on inverse systems of abelian presheaves
(not to be confused with the analogue $\varprojlim$ on inverse systems of abelian sheaves), 
and $F$ is the sheafification functor.   The functor $T$ has left adjoint given by the exact functor $F$, so likewise 
$G=T^{\mathbf{N}}$ has left adjoint $F^{\mathbf{N}}$ that is exact and hence $G$ carries injectives to injectives.
Since $F$ is exact, we conclude that
$${\rm{R}}\varprojlim = {\rm{R}}(F \circ {{\varprojlim}}_{\rm{P}} \circ G) = F \circ {\rm{R}}{{{\varprojlim}}_{\rm{P}}} \circ {\rm{R}}G.$$
Thus, we have a spectral sequence
\begin{equation}\label{spseq}
F \circ ({\rm{R}}^i{{{\varprojlim}}_{\rm{P}}}) \circ {\rm{R}}^jG \Rightarrow {\rm{R}}^{i+j}\varprojlim.
\end{equation}

Let's compute ${\rm{R}}^jG$ and ${\rm{R}}^j\varprojlim_{\rm{P}}$.  We claim
\begin{equation}\label{Rformula}
({\rm{R}}^jG)((\mathscr{F}_n)) = (U \mapsto ({\rm{H}}^j(U,\mathscr{F}_n))_{n \ge 1}),\,\,\,
({\rm{R}}^j\varprojlim_{\rm{P}})((\mathscr{H}_n)) = (U \mapsto {\varprojlim}^j((\mathscr{H}_n(U)))).
\end{equation}
Both proposed formulas are $\delta$-functors (respectively valued in inverse systems of presheaves
and in presheaves), and each is as expected for $j=0$, so we just need to check erasability.
The first is erasable by using canonical flasque resolutions, and the second is erasable
because $\mathscr{H}_n \mapsto \mathscr{H}_n(U)$ carries injective inverse systems of abelian presheaves
to injective abelian groups due to having an exact left adjoint. 

Let's feed the descriptions from (\ref{Rformula}) into the spectral sequence (\ref{spseq}) evaluated on $(\mathscr{G}_n)$.
First, $({\rm{R}}^jG)(\mathscr{G}_n)$ is the inverse system of presheaves whose value on an open $U \subset X$
is $({\rm{H}}^j(U, \mathscr{G}_n))$, so for $U \in \mathscr{B}$ its value is 0 when $j>0$ and is $(\mathscr{G}_n(U))$
when $j=0$.  Thus, 
$$({\rm{R}}^i{{\varprojlim}}_{\rm{P}})({\rm{R}}^jG((\mathscr{G}_n))): U \mapsto {{{\varprojlim}}^i}(({\rm{H}}^j(U, \mathscr{G}_n)))$$
and for $U \in \mathscr{B}$ this vanishes when $j > 0$ and is equal to $\varprojlim^i(\mathscr{G}_n(U))$ when $j=0$.
But $\varprojlim^i = 0$ for $i > 1$ and by hypothesis $\varprojlim^1(\mathscr{G}_n(U)) = 0$ for $U \in \mathscr{B}$, so upon applying
the sheafification functor $F$ we get
$$(F \circ ({\rm{R}}^i {{\varprojlim}}_{\rm{P}}) \circ {\rm{R}}^jG)((\mathscr{G}_n)) = 0$$
when $j > 0$ and when $j=0$ with $i > 0$.  Thus, the ${\rm{E}}_2^{i,j}$-term in the spectral
sequence (\ref{spseq}) evaluated on $(\mathscr{G}_n)$
vanishes whenever $(i,j) \ne (0,0)$, so $({\rm{R}}^k \varprojlim)((\mathscr{G}_n)) = 0$ for $k > 0$, as desired.

The final assertion in the lemma now follows by taking $\mathscr{B}$ to be the collection of affine open subschemes,
since a total derived 
functor between bounded-below derived categories is computed on bounded-below complexes of acyclics 
by applying the original functor termwise to the complex.
\end{proof}

We conclude that 
$${\rm{R}}\varprojlim(((\mathscr{F} \otimes \Omega^*_{X/S,{\rm{cris}}})_{D_Y(X)_n})_{n \ge 1}) =
(\varprojlim (\mathscr{F}_{D_Y(X)_n} \otimes_{\mathscr{O}_{D_Y(X)_n}} \Omega^*_{D_Y(X)_n/S, \delta})).$$
But the inverse limit complex 
on the right side is exactly the meaning of 
the complex of abelian sheaves $\mathscr{F} \widehat{\otimes}_{\mathscr{O}^{\wedge}_{D_Y(X)}} \widehat{\Omega}^*_{D_Y(X)^{\wedge}/S,\delta}$.
Putting it all together, we have constructed a zig-zag map
$$\alpha: {\rm{R}}\Gamma((X/S)_{\rm{cris}}, \mathscr{F}) \to {\rm{R}}\Gamma(X_{\rm{Zar}}, 
\mathscr{F} \widehat{\otimes}_{\mathscr{O}^{\wedge}_{D_Y(X)}} \widehat{\Omega}^*_{D_Y(X)^{\wedge}/S,\delta})$$
in $D^+({\rm{Ab}})$ that is 
natural in the quasi-coherent crystal 
$\mathscr{F}$, and we want to show this is a quasi-isomorphism.  (Up to this point, we have used quasi-coherence of $\mathscr{F}$
but only mildly used the crystal property.)  

\subsection{Global comparison: proof}\label{globalproof}

By the construction of $\alpha$, to prove it is a quasi-isomorphism amounts  to showing $b_{\mathscr{F}}$ in
(\ref{2dagger}) is a quasi-isomorphism.  We shall use a \v{C}ech-to-derived spectral sequence to reduce to the affine setting
with \'etale coordinates (as in Theorem \ref{dPoincare}).  To carry out the argument in a manner which makes
compatibility with spectral sequences abutting to the homologies of the sourca and target of
$b_{\mathscr{F}}$ most transparent, we now reformulate 
$b_{\mathscr{F}}$ in terms of a slice category.   

Recall that for any site $C$
(such as ${\rm{Cris}}(X/S)$, or the topos $(X/S)_{\rm{cris}}$ regarded as a site) and sheaf of sets $\mathscr{H}$ on $C$ (such as $h_{\widehat{D}}$
on ${\rm{Cris}}(X/S)$), 
the slice category ${\rm{Shv}}(C)/\mathscr{H}$ of sheaves of sets equipped with a map to $\mathscr{H}$ is naturally equivalent 
to the category ${\rm{Shv}}(C/\mathscr{H})$ of sheaves of sets on the site of pairs $(U, h)$ consisting of objects $U$ in $C$ and
$h \in \mathscr{H}(C)$. The evident ``restriction functor'' ${\rm{Shv}}(C) \to {\rm{Shv}}(C/\mathscr{H})$
(denoted $\mathscr{G} \mapsto \mathscr{G}|_{\mathscr{H}}$ is thereby identified with
${\rm{Shv}}(C) \to {\rm{Shv}}(C)/\mathscr{H}$ defined by $\mathscr{G} \mapsto \mathscr{G} \times \mathscr{H}$.  Thus, 
on abelian sheaves this restriction is exact, and it has an exact left adjoint, so it carries injective resolutions to injective resolutions for abelian sheaves
(so omitting the notation $(\cdot)|_{\mathscr{H}}$ in derived functors is harmless).

Since ${\rm{Hom}}(\mathscr{H}, \cdot)$ on ${\rm{Shv}}(C)$ is identified with the global sections of
the restriction over $\mathscr{H}$, we conclude that 
${\rm{R}}{\rm{Hom}}(\mathscr{H}, \cdot) \simeq {\rm{R}}\Gamma(C/\mathscr{H},  (\cdot)|_{\mathscr{H}})$
on ${\rm{Ab}}(C)$, so $b_{\mathscr{F}}$ is identified with the application to $\mathscr{F} \otimes_{\mathscr{O}_{X/S}}
\Omega^*_{X/S, {\rm{cris}}}$ of the  ``pullback map''
$${\rm{R}}\Gamma((X/S)_{\rm{cris}}, \mathscr{G}^{\bullet}) \to {\rm{R}}\Gamma((X/S)_{\rm{cris}}/h_{\widehat{D}}, \mathscr{G}^{\bullet})$$
along the forgetful functor of sites $(X/S)_{\rm{cris}}/h_{\widehat{D}} \to (X/S)_{\rm{cris}}$. 
In other words, the quasi-isomorphism property for $\alpha$ amounts to the quasi-isomorphism property for the pullback map 
\begin{equation}\label{bmap}
b'_{\mathscr{F}}: {\rm{R}}\Gamma((X/S)_{\rm{cris}}, \mathscr{F} \otimes_{\mathscr{O}_{X/S}} \Omega^*_{X/S, {\rm{cris}}}) \to
{\rm{R}}\Gamma((X/S)_{\rm{cris}}/h_{\widehat{D}}, \mathscr{F} \otimes_{\mathscr{O}_{X/S}} \Omega^*_{X/S, {\rm{cris}}}).
\end{equation}

The merit of the reformulation (\ref{bmap}) will be its clean interaction with \v{C}ech hypercovers.  
To carry this out, we need to introduce some notation. For any open $U \subset X$ and $\mathscr{H} \in (X/S)_{\rm{cris}}$ (such as $h_{\widehat{D}}$) with restriction $\mathscr{H}_U \in (U/S)_{\rm{cris}}$, 
we have the morphisms of topoi $j: (U/S)_{\rm{cris}} \to (X/S)_{\rm{cris}}$
and  $j_{\mathscr{H}}: (U/S)_{\rm{cris}}/\mathscr{H}_U \to (X/S)_{\rm{cris}}/\mathscr{H}$ which satisfy
$$(j_{\mathscr{H}})_!(\mathscr{G}|_{\mathscr{H}_U}) \simeq j_!(\mathscr{G})|_{\mathscr{H}}$$
naturally in $\mathscr{G} \in (U/S)_{\rm{cris}}$.  In particular, writing $\ast_U \in (U/S)_{\rm{cris}}$ and $\ast \in (X/S)_{\rm{cris}}$
to denote the final objects, clearly $\ast_U|_{\mathscr{H}_U} = \mathscr{H}_U \in (U/S)_{\rm{cris}}/\mathscr{H}_U$
is the final object and 
\begin{equation}\label{jcover}
(j_{\mathscr{H}})_!(\ast_U)|_{\mathscr{H}_U} \simeq j_!(\ast_U)|_{\mathscr{H}}.
\end{equation}

To harness (\ref{bmap}), consider an open cover $\{Y_i\}_{i \in I}$ of $Y$
over open $V_i \subset S$, and let $U_i = X \cap Y_i$, so
$U_i$ is the closed subscheme of $Y_i$ and $(U_i/S)_{\rm{cris}} = (U_i/V_i)_{\rm{cris}}$.
Note that $h_{\widehat{D}}|_{(U_i/S)_{\rm{cris}}} = h_{\widehat{D}_i}$
with $\widehat{D}_i$ corresponding to the inverse system of divided power thickenings $D_{Y_i}(U_i)_n$ of
$U_i$ over $V_i$.  For each multi-index $\underline{i} = (i_0, \dots, i_n) \in I^{n+1}$,
we define $U_{\underline{i}} = U_{i_0} \cap \dots \cap U_{i_n}$ and similarly $Y_{\underline{i}}$ and $V_{\underline{i}}$. 
For the natural morphisms of sites $j_i: {\rm{Cris}}(U_i/S) \to {\rm{Cris}}(X/S)$ we see that
$K^0 := \coprod (j_i)_!(\ast_i) \to \ast$ is a cover (since the $U_i$'s cover $X$ topologically) whose 
$(n+1)$th fiber power $K^n$
is $\coprod_{{\underline{i}} \in I^{n+1}} (j_{\underline{i}})_{!}(\ast_{\underline{i}})$.
By (\ref{jcover}), the restriction over any $\mathscr{H} \in (X/S)_{\rm{cris}}$  
of the \v{C}ech hypercover $K^{\bullet} \to \ast$
is the \v{C}ech hypercover of the final object in
$(X/S)_{\rm{cris}}/\mathscr{H}$ arising from the topoi $(U_i/S)_{\rm{cris}}/\mathscr{H}_{U_i}$. 
Thus, for any bounded below complex $\mathscr{G}^{\bullet}$ of abelian sheaves in 
$(X/S)_{\rm{cris}}$, ``restriction over $\mathscr{H}$'' defines a morphism of spectral sequences
$$\xymatrix{
{\prod_{{\underline{i}} \in I^{n+1}} {\mathbf{H}}^m((U_{\underline{i}}/S)_{\rm{cris}}, \mathscr{G}^{\bullet})} \ar@{=>}[r] \ar[d] &
{{\mathbf{H}}^{n+m}((X/S)_{\rm{cris}}, \mathscr{G}^{\bullet})}  \ar[d]\\
{\prod_{{\underline{i}} \in I^{n+1}} {\mathbf{H}}^m((U_{\underline{i}}/S)_{\rm{cris}}/\mathscr{H}_{U_{\underline{i}}}, \mathscr{G}^{\bullet})}  \ar@{=>}[r] &
{{\mathbf{H}}^{n+m}((X/S)_{\rm{cris}}/\mathscr{H}, \mathscr{G}^{\bullet})}} 
$$
in which the vertical maps are ``pullback over $\mathscr{H}$''.  
Taking $\mathscr{H}$ to be $h_{\widehat{D}}$, this commutative
diagram expresses the sense in which
the homology maps arising from (\ref{bmap}) are
compatible with the spectral sequence arising from the open cover $\{U_i\}$ of $X$.

We conclude that if the quasi-isomorphism problem is settled for all $U_{\underline{i}} \hookrightarrow Y_{\underline{i}}$ over $V_{\underline{i}}$  then 
(\ref{bmap}) is also a quasi-isomorphism.  Thus, if we choose each $Y_i = \Spec(P_i)$ to be affine
with $V_i = \Spec(A_i)$ affine and $P_i$ admitting \'etale coordinates over $A_i$
then the general case reduces to the case of such overlaps $U_{\underline{i}} \to Y_{\underline{i}} \to V_{\underline{i}}$
(which are all affine when $Y$ is $S$-separated).  These latter cases are all separated over $V_{\underline{i}}$,
so we have thereby reduced the general quasi-isomorphism property for
(\ref{bmap}) to the special case when $S = \Spec(A)$, $X = \Spec(R)$, and $Y = \Spec(P)$
are affine and $P$ admits \'etale coordinates over $A$.  
We let $(I_0, \eta)$ be the divided power structure on $A$. 

\begin{remark}
It may seem that we are now done, since in the special affine circumstances to which we have reduces ourselves, 
the isomorphism asserted in the statement of Theorem \ref{hartshornethm} is between objects that (in our present affine setting)
are known to be isomorphic by the discussion between the statement and proof of Theorem \ref{chaoticH}.

However, we don't know that in such affine settings this isomorphism
coincides (in $D(\widehat{A})$, or in $D({\rm{Ab}})$), at least up to signs, 
with the map $\alpha$ that we are presently aiming to prove is a quasi-isomorphism.
We will be using ideas and results from \S\ref{compsec}--\S\ref{affzarsite} in our proof of the quasi-isomorphsm property, 
but we are not done with the proof of Theorem \ref{hartshornethm} due to this issue with identifying maps in a derived category (at least up to signs).
\end{remark}

In \S\ref{relsec} we defined $\mathscr{C}_{R/A}$ to be the topos of sheaves on ${\rm{Cris}}(R/A)_{\rm{ind}}$, and
we saw in (\ref{changetop}) via Lemma \ref{RpiO} that naturally
${\rm{R}}\Gamma((R/S)_{\rm{cris}}, \mathscr{G}) \simeq {\rm{R}}\Gamma(\mathscr{C}_{R/A}, \mathscr{G})$
for any quasi-coherent $\mathscr{O}_{R/A}$-module $\mathscr{G}$. 
Using the identification
$$\mathscr{C}_{R/A}/h_{\widehat{D}} \simeq {\rm{Shv}}({\rm{Cris}}(R/A)_{\rm{ind}}/h_{\widehat{D}}) =
{\rm{Shv}}(({\rm{Cris}}(R/A)/h_{\widehat{D}})_{\rm{ind}}),$$
the same method of proof yields
\begin{equation}\label{changetopD}
{\rm{R}}\Gamma((R/S)_{\rm{cris}}/h_{\widehat{D}}, \mathscr{G}) \simeq {\rm{R}}\Gamma(\mathscr{C}_{R/A}/h_{\widehat{D}}, \mathscr{G})
\end{equation}
because an analogue of Lemma \ref{RpiO} holds (with essentially the same proof) 
for the slice categories over $h_{\widehat{D}}$:  for
$q_{\widehat{D}}: {\rm{Cris}}(R/A)_{\rm{ind}}/h_{\widehat{D}} \to {\rm{Cris}}(R/A)_{\rm{Zar}}/h_{\widehat{D}}$
and any quasi-coherent $\mathscr{O}_{R/A}$-module $\mathscr{G}$ we have
${\rm{R}}^i (q_{\widehat{D}})_*(\mathscr{G}|_{h_{\widehat{D}}})) = 0$ for all $i > 0$.
These comparisons are the main ingredient in the proof of:

\begin{lemma}\label{cechind}
Let $\mathscr{G}^{\ast}$ be a bounded below complex of quasi-coherent $\mathscr{O}_{R/A}$-modules in $(R/A)_{\rm{cris}}$, and 
$K^{\bullet}$ the hypercover $h_{\widehat{D}(\bullet)} = h_{\widehat{D}}^{\bullet+1}$ of the final object in $(R/A)_{\rm{cris}}$. 

\begin{itemize}
\item[(i)] The natural map
$${\rm{Tot}}^{\oplus}(\Gamma(K^{\bullet}, \mathscr{G}^{\ast})) \to {\rm{R}}\Gamma((R/A)_{\rm{cris}}, \mathscr{G}^*)$$
is a quasi-isomorphism.
\item[(ii)] The restriction $K^{\bullet}|_{h_{\widehat{D}}}$ as a simplicial object in $(R/A)_{\rm{cris}}/h_{\widehat{D}}$
is the hypercover of the final object $h_{\widehat{D}}$, 
given termwise by fiber powers of $h_{\widehat{D}} \times h_{\widehat{D}} = h_{\widehat{D}(1)}$ over $h_{\widehat{D}}$
$($using ${\rm{pr}}_2$$)$, and the natural map
$${\rm{Tot}}^{\oplus}(\Gamma(K^{\bullet}|_{h_{\widehat{D}}}, \mathscr{G}^{\ast})) \to {\rm{R}}\Gamma((R/A)_{\rm{cris}}/h_{\widehat{D}}, \mathscr{G}^*)$$
is a quasi-isomorphism.
\end{itemize}
\end{lemma}

For any first-quadrant double complex $C^{\ast,\ast}$, we define ${\rm{Tot}}^{\oplus}(C^{\ast,\ast})$ to be the complex
with $n$th term $\oplus_{i+j=n} C^{i,j}$ and the usual differentials.

\begin{proof}
We saw in the proof of the affine Poincar\'e lemma (Proposition \ref{chaoticH}) that $h_{\widehat{D}}$ covers the final object
in $\mathscr{C}_{R/A}$, and the description of $K^{\bullet}|_{h_{\widehat{D}}}$ in (ii) is clear.   Thus, 
via the comparison isomorphism in (\ref{changetop}) and its analogue in (\ref{changetopD}) for the slice categories over $h_{\widehat{D}}$
it suffices to prove the analogues of (i) and (ii) for the indiscrete topologies (working with the site ${\rm{Cris}}(R/A)$
of affine objects). 

Using induction on $n \in \mathbf{Z}$ such that $\mathscr{G}^i = 0$ for all $i < n$, by standard naive truncaton arguments
we reduce to the case that $\mathscr{G}^*$ is concentrated in a single degree, and then $\mathscr{G}^* = \mathscr{G}[0]$
for a quasi-coherent $\mathscr{O}_{R/A}$-module $\mathscr{G}$.  Thus, (i) is exactly Proposition \ref{chaoticH}
because the isomorphism invoked from (the proof of) \cite[Tag 07JM]{SP} near the end of its proof is precisely
the natural map from the complex of sections over the layers of a \v{C}ech hypercover to the total derived functor cohomology complex. 
The exact same formalism applies in the setting of (ii).
\end{proof}

We now apply Lemma \ref{cechind} to $\mathscr{G}^* = \mathscr{F} \otimes_{\mathscr{O}_{R/A}} \Omega^*_{R/A, {\rm{cris}}}$, 
each of whose terms are quasi-coherent $\mathscr{O}_{R/A}$-modules.   This provides quasi-isomorphisms
\begin{equation}\label{TotisomK}
{\rm{Tot}}^{\oplus}(\Gamma(K^{\bullet}, \mathscr{F} \otimes_{\mathscr{O}_{R/A}} \Omega^*_{R/A, {\rm{cris}}})) \simeq
{\rm{R}}\Gamma((R/A)_{\rm{cris}}, \mathscr{F} \otimes_{\mathscr{O}_{R/A}} \Omega^*_{R/A, {\rm{cris}}}),
\end{equation}
\begin{equation}\label{TotisomK'}
{\rm{Tot}}^{\oplus}(\Gamma(K^{\bullet}|_{h_{\widehat{D}}}, \mathscr{F} \otimes_{\mathscr{O}_{R/A}} \Omega^*_{R/A, {\rm{cris}}})) \simeq
{\rm{R}}\Gamma((R/A)_{\rm{cris}}/h_{\widehat{D}}, \mathscr{F} \otimes_{\mathscr{O}_{R/A}} \Omega^*_{R/A, {\rm{cris}}}).
\end{equation}
But from the definitions, the double complex of terms
$$\Gamma(K^{\nu}, \mathscr{F} \otimes_{\mathscr{O}_{R/A}} \Omega^{\mu}_{R/A, {\rm{cris}}})$$
in (\ref{TotisomK}) 
is exactly the double complex of terms $E^{\mu,\nu}$ (with the same differentials) as considered in the proof of Theorem \ref{dPoincare}.
In that proof, more specifically in \S\ref{rowsec}, we saw that the edge map 
$E^{*,0} \to {\rm{Tot}}^{\oplus}(E^{*,*})$ is a quasi-isomorphism.  Since $E^{*,0} = M \widehat{\otimes} \widehat{\Omega}^*_{\widehat{D}/A, \widehat{\gamma}}$
for $\widehat{D} = \varprojlim D_P(J)/n! J_D^{[n]}$ and the $\widehat{D}$-module
$M = \varprojlim M_n$ with pd quasi-nilpotent integrable
connection associated to the crystal $\mathscr{F}$ (i.e., $M_n = \mathscr{F}(D_{P,\eta}(J)/n! J_D^{[n]}, J_D/n! J_D^{[n]}, \gamma)$), 
from (\ref{TotisomK}) we get a quasi-isomorphism
$$\beta: {\rm{R}}\Gamma((R/A)_{\rm{cris}}, \mathscr{F} \otimes_{\mathscr{O}_{R/A}} \Omega^*_{R/A, {\rm{cris}}}) \simeq
M \widehat{\otimes} \widehat{\Omega}^*_{\widehat{D}/A, \widehat{\gamma}}.$$

Let's see that analogous considerations work with the quasi-isomorphism in (\ref{TotisomK'}).  
Using the description of $K^{\bullet}|_{h_{\widehat{D}}}$ in Lemma \ref{cechind}(ii), the double complex
of terms
$$\Gamma(K^{\nu}|_{h_{\widehat{D}}}, \mathscr{F} \otimes_{\mathscr{O}_{R/A}} \Omega^{\mu}_{R/A, {\rm{cris}}})$$
in (\ref{TotisomK'}) 
is the double complex of terms ${E'}^{\mu, \nu} := E^{\mu, \nu+1}$ for $\mu, \nu \ge 0$
in which the alternating sum $\delta^{\mu,\nu}$ of face maps 
${E'}^{\mu, \nu} \to {E'}^{\mu, \nu+1}$
defining the differentials of the double complex in the $\nu$-direction is the alternating sum of face maps
$E^{\mu, \nu+1} \to E^{\mu, \nu+2}$ {\em omitting} the final face map.
Thus, the induced map ${\rm{H}}^i(\delta^{\ast,\nu}): {\rm{H}}^i({E'}^{\ast, \nu}) \to {\rm{H}}^i({E'}^{\ast,\nu+1})$
is the alternating sum of the maps ${\rm{H}}^i(E^{\ast, \nu+1}) \to {\rm{H}}^i(E^{\ast,\nu+2})$
arising from the face maps $E^{\mu, \nu+1} \to E^{\mu, \nu+2}$ {\em omitting} the final face map.

We saw in \S\ref{rowsec} that all such face maps yield the same effect
${\rm{H}}^i(E^{\ast,\nu+1}) \to {\rm{H}}^i(E^{\ast,\nu+2})$ that is moreover an isomorphism, so since there are $(\nu+3)-1 = \nu+2$ such maps, 
the alternating sum vanishes for $\nu$ even (such as $\nu=0$) and is an isomorphism for $\nu$ odd. 
This is precisely the same parity situation as in the setting without passing to slice categories over $h_{\widehat{D}}$,
so the edge degeneration argument {\em still} works, now using the edge ${E'}^{\ast,0} = E^{\ast,1} =
M(1) \widehat{\otimes} \widehat{\Omega}^*_{\widehat{D}(1)/A, \widehat{\gamma}(1)}$, 
where $\widehat{D}(1)$ is the pd-adic completion of $D(1) = D_{P(1),\eta}(J(1))$ (with $P(1) = P \otimes_A P$ and $J(1) = \ker(P(1) \to R)$)
and $M(1) = \varprojlim M(1)_n$ is the $\widehat{D}(1)$-module with
pd quasi-nilpotent integrable connection associated to the crystal $\mathscr{F}$ in $(R/A)_{\rm{cris}}$. 
Hence, via (\ref{TotisomK'}) we get a quasi-isomorphism
$$\beta': {\rm{R}}\Gamma((R/A)_{\rm{cris}}/h_{\widehat{D}}, \mathscr{F} \otimes_{\mathscr{O}_{R/A}} \Omega^*_{R/A, {\rm{cris}}}) \simeq
M(1) \widehat{\otimes} \widehat{\Omega}^*_{\widehat{D}(1)A,  \widehat{\gamma}(1)}.$$

Using $D_{P,\eta}(J) \to D_{P(1),\eta}(J(1))$ defined by the inclusion $P \to P \otimes_A P = P(1)$ into the second tensor factor, there is a natural map
$$f^*: M \widehat{\otimes} \widehat{\Omega}^*_{\widehat{D}/A, \widehat{\gamma}} \to M(1) \widehat{\otimes} 
\widehat{\Omega}^*_{\widehat{D}(1)/A, \widehat{\gamma(1)}}$$
linear over $\widehat{D} \to \widehat{D}(1)$.  Putting it all together, we get a diagram in $D^+(A)$:
\begin{equation}\label{bdiagram}
\xymatrix{
{\rm{R}}\Gamma((R/A)_{\rm{cris}}, \mathscr{F} \otimes_{\mathscr{O}_{R/A}} \Omega^*_{R/A,{\rm{cris}}}) \ar[d]_-{b'_{\mathscr{F}}} \ar[r]_-{\simeq}^-{\beta} &
M\widehat{\otimes} \widehat{\Omega}^*_{\widehat{D}/A, \widehat{\gamma}} \ar[d]^-{f^*} \\
{\rm{R}}\Gamma((R/A)_{\rm{cris}}/h_{\widehat{D}}, \mathscr{F} \otimes_{\mathscr{O}_{R/A}} \Omega^*_{R/A,{\rm{cris}}}) \ar[r]^-{\simeq}_-{\beta'} & 
M(1) \widehat{\otimes} \widehat{\Omega}^*_{\widehat{D}(1)/A, \widehat{\gamma}(1)}}
\end{equation}
This diagram commutes (in the derived category) because it is the concatenation of the diagrams
$$
\xymatrix{
{\rm{R}}\Gamma((R/A)_{\rm{cris}}, \mathscr{F} \otimes_{\mathscr{O}_{R/A}} \Omega^*_{R/A,{\rm{cris}}}) \ar[d]_-{b'_{\mathscr{F}}} \ar[r]^-{\simeq} &
 {\rm{R}}\Gamma(\mathscr{C}_{R/A}, \mathscr{F} \otimes_{\mathscr{O}_{R/A}} \Omega^*_{R/A,{\rm{cris}}})   \ar[d]  \\
{\rm{R}}\Gamma((R/A)_{\rm{cris}}/h_{\widehat{D}}, \mathscr{F} \otimes_{\mathscr{O}_{R/A}} \Omega^*_{R/A,{\rm{cris}}}) \ar[r]_-{\simeq} & 
{\rm{R}}\Gamma(\mathscr{C}_{R/A}/h_{\widehat{D}}, \mathscr{F} \otimes_{\mathscr{O}_{R/A}} \Omega^*_{R/A,{\rm{cris}}})}
$$
(which commutes in the derived category) and 
$$\xymatrix{
 {\rm{R}}\Gamma(\mathscr{C}_{R/A}, \mathscr{F} \otimes_{\mathscr{O}_{R/A}} \Omega^*_{R/A,{\rm{cris}}})   \ar[d] \ar[r]^-{\simeq} &  
 {\rm{Tot}}^{\oplus}(\Gamma(K^{\bullet}, \mathscr{F} \otimes_{\mathscr{O}_{R/A}} \Omega^*_{R/A, {\rm{cris}}}))  \ar[d] &
  E^{\ast,0} \ar[d] \ar[l]_-{{\rm{qism}}}^-{\rm{edge}}
  \\
 {\rm{R}}\Gamma(\mathscr{C}_{R/A}/h_{\widehat{D}}, \mathscr{F} \otimes_{\mathscr{O}_{R/A}} \Omega^*_{R/A,{\rm{cris}}}) \ar[r]_-{\simeq} & 
 {\rm{Tot}}^{\oplus}(\Gamma(K^{\bullet}|_{h_{\widehat{D}}}, \mathscr{F} \otimes_{\mathscr{O}_{R/A}} \Omega^*_{R/A, {\rm{cris}}}))  & 
 E^{\ast,1} \ar[l]^-{{\rm{qism}}}_-{\rm{edge}}
}
$$
(whose left part commutes in the derived category and right part commutes as a diagram of complexes); the right side of the latter diagram 
is the map along the right side of (\ref{bdiagram}) by the definitions.

The commutativity of (\ref{bdiagram}) reduces the quasi-isomorphism property for $b'_{\mathscr{F}}$ to the same for $f^*$.
But ${\rm{H}}^i(f^*)$ is exactly the map $g^i:{\rm{H}}^i(E^{*,0})\to {\rm{H}}^i(E^{*,1})$ induced by the inclusion 
$P \to P(1)$ along the second tensor factor, and that map (as well as the one resting on inclusion along the other tensor factor)
was shown to be an isomorphism in \S\ref{rowsec}. 
This finishes the proof of Theorem \ref{hartshornethm}. 

\medskip

By using higher homotopical methods, Theorem \ref{hartshornethm} can be upgraded to a more
localized formulation (which recovers \cite[(7.1.2)]{bogus} when $S$ is a $\mathbf{Z}/(p^m)$-scheme
and $\mathscr{I}_0 = p\mathscr{O}_S$ equipped with its usual divided powers):

\begin{proposition}\label{localmain} In the setting of Theorem $\ref{hartshornethm}$ with $S = \Spec(A)$, 
$${\rm{R}}u_{X/S,\ast}(\mathscr{F}) \simeq \mathscr{F} \widehat{\otimes}_{\mathscr{O}^{\wedge}_{D_Y(X)}} \widehat{\Omega}^*_{D_Y(X)^{\wedge}/S,\delta}$$
in the derived $\infty$-category $D(X,A)$ of sheaves of $A$-modules on $X_{\rm{Zar}}$.
In particular, the right side depends functorially only on $X$ $($not on $Y$$)$. 
\end{proposition}

\begin{proof}
For each open $U \subset X$ and open $V \subset Y$ satisfying $V \cap X = U$, we have
$D_Y(X)^{\wedge}|_U = D_V(U)^{\wedge}$.  In this way, 
Theorem \ref{hartshornethm} applied to $U \to S$ and $U \hookrightarrow V$ yields an isomorphism in $D(A)$ between
$${\rm{R}}\Gamma((U/S)_{\rm{cris}}, \mathscr{F}) = {\rm{R}}\Gamma(U_{\rm{Zar}}, {\rm{R}}u_{X/S,\ast}(\mathscr{F}))$$
and 
$${\rm{R}}\Gamma(U_{\rm{Zar}},  \mathscr{F} \widehat{\otimes}_{\mathscr{O}^{\wedge}_{D_Y(X)}} \widehat{\Omega}^*_{D_Y(X)^{\wedge}/S,\delta})$$
compatibly with restriction along open immersions $U' \subset U$ (see Remark \ref{DXA}).

By \cite[Cor.\,2.1.2.3]{lurie}, there is an equivalence 
from the derived $\infty$-category $D(X,A)$ to the derived $\infty$-category of $D(A)$-valued hypersheaves on $X_{\rm{Zar}}$
by assigning to each $K^{\bullet} \in D(X,A)$ the hypersheaf $U \mapsto {\rm{R}}\Gamma(U, K^{\bullet})$.
In particular, for any $K^{\bullet}$ and ${K'}^{\bullet}$, a collection of compatible isomorphisms
${\rm{R}}\Gamma(U, K^{\bullet}) \simeq {\rm{R}}\Gamma(U, {K'}^{\bullet})$ in $D(A)$
for open $U \subset X$ arises from a unique isomorphism
$K^{\bullet} \simeq {K'}^{\bullet}$ in $D(X,A)$.
Thus, we obtain the desired isomorphism in $D(X,A)$.
\end{proof}

\begin{remark}\label{genS}
Proposition \ref{localmain} can be formulated and established
for non-affine $S$ by working with the derived $\infty$-category of sheaves of $h^{-1}\mathscr{O}_S$-modules
on $X_{\rm{Zar}}$, where $h:X \to S$ is the structure map. 
\end{remark}

\begin{remark}\label{Fbullet}
There is another way to describe 
$\mathscr{F} \widehat{\otimes}_{\mathscr{O}^{\wedge}_{D_Y(X)}} \widehat{\Omega}^*_{D_Y(X)^{\wedge}/S,\delta}$
in $D(X,A)$ without differential forms, via globalizing constructions in the proof of the affine result in Theorem \ref{dPoincare}. 
Let's explain how this goes, assuming for simplicity that $Y$ is separated over $\Spec(A)$.

For $\nu \ge 0$, let $D(\nu)_n = D_{Y^{\nu+1}}(X)_n \subset D_{Y^{\nu+1}}(X)$ be the closed subscheme
defined by the vanishing of $n! \mathscr{J}(\nu)^{[n]}$,
where $\mathscr{J}(\nu)$ is the ideal defining the closed immersion $X \hookrightarrow Y \hookrightarrow Y^{\nu+1}$
(fiber product over $S$), and let $\gamma(\nu)$ be the divided powers on $\mathscr{J}(\nu)_D \subset \mathscr{O}_{D_{Y^{\nu+1}}(X)}$ compatible with $\eta$. 
The scheme $D(\nu)_n$ is a globalization of $\Spec(D_{P(\nu),\eta}(J(\nu))/n!J(\nu)^{[n]})$ 
from our earlier work in the affine setting.
Since $X \hookrightarrow D(\nu)_n$ is a closed immersion with the same underlying topological space,
we may and do view each quasi-coherent sheaf
$\mathscr{F}(\nu)_n := \mathscr{F}_{(X, D(\nu)_n, \gamma(\nu))}$ on $(D(\nu)_n)_{\rm{Zar}}$ as a sheaf of $A$-modules
on $X_{\rm{Zar}}$. 
Hence, $$\widehat{\mathscr{F}}(\nu) := \varprojlim_{n \ge 1} \mathscr{F}(\nu)_n$$
makes sense as a sheaf of $A$-modules on $X_{\rm{Zar}}$, as does
$$\mathscr{E}^{\mu,\nu} := \varprojlim_{n \ge 1} \mathscr{F}(\nu)_n \otimes_{\mathscr{O}_{D(\nu)_n}} \Omega^{\mu}_{D(\nu)_n/S, \gamma(\nu)}.$$

The sheaves $\mathscr{E}^{\mu,\nu}$ are a globalization of $E^{\mu,\nu}$ from the proof of Theorem \ref{dPoincare}.
In particular, we organize the $\mathscr{E}^{\mu,\nu}$'s into a double complex 
of sheaves of $A$-modules on $X_{\rm{Zar}}$ in exactly the same
way that we did for the $E^{\mu,\nu}$'s (using that $\mathscr{F}$ is a quasi-coherent crystal for defining the differentials in
the $\mu$-directions). 
Clearly $\mathscr{E}^{0,\ast}$ is the evident complex 
\begin{equation}\label{Fhat}
\widehat{\mathscr{F}}(\ast) := (\widehat{\mathscr{F}}(0) \to \widehat{\mathscr{F}}(1) \to \widehat{\mathscr{F}}(2) \to \dots)
\end{equation}
(globalizing the \v{C}ech-Alexander complex ${\rm{\mbox{\v{C}}A}}(P \twoheadrightarrow R, \mathscr{F})$ in the affine setting).
Also, $\mathscr{E}^{\ast,0}$ is exactly the definition of
$\mathscr{F} \widehat{\otimes}_{\mathscr{O}^{\wedge}_{D_Y(X)}} \widehat{\Omega}^*_{D_Y(X)^{\wedge}/S,\delta}$.

Consider the augmentation maps from those edges:
$$\mathscr{E}^{\ast,0} \to {\rm{Tot}}(\mathscr{E}^{\ast,\ast}) \leftarrow \mathscr{E}^{0,\ast}.$$
We claim these are quasi-isomorphisms, so in $D(X,A)$ this gives an isomorphism
$$\mathscr{F} \widehat{\otimes}_{\mathscr{O}^{\wedge}_{D_Y(X)}} \widehat{\Omega}^*_{D_Y(X)^{\wedge}/S,\delta} \simeq
\widehat{\mathscr{F}}(\ast),$$
where the right side doesn't involve differential forms (though clearly it involves $X \hookrightarrow Y$ over $A$).

Our quasi-isomorphism problem is a Zariski-local problem on $X$, so to check it we may reduce
to the case when $X = \Spec(R)$ is affine and $Y = \Spec(P)$ is affine 
such that $P$ admits \'etale coordinates over $A$. It suffices to prove
the quasi-isomorphism property at the level of sections over {\em all} open affines in $X$ having the form $V \cap X$ for affine
open $V \subset \Spec(P)$, 
 and these are all instances of the quasi-isomorphism results given by Theorem \ref{dPoincare}. (In the affine setting we had homotopy equivalences and not just quasi-isomorphisms, but that doesn't matter here.) 
\end{remark}

\begin{remark}\label{dRcomp}
Let $W$ be a complete discrete valuation ring with mixed characteristic $(0,p)$
and uniformizer $p$, so $I_0 := pW$ has unique divided powers. 
For a scheme $Y_0$ over the residue field $k=W/(p)$,  ${\rm{Cris}}(Y_0/W)$
coincides with the site ${\rm{Cris}}(Y_0/{\rm{Spf}}(W))$ as defined in \cite[7.17ff.]{bogus}.  
Hence, our ${\rm{H}}^i((Y_0/W)_{\rm{cris}}, \mathscr{O}_{Y_0/W})$
is the same as ${\rm{H}}^i_{\rm{cris}}(Y_0)$ in \cite[(7.26.1)]{bogus}.
This has reasonable finiteness properties over $W$ when $Y_0$ is smooth
and proper over $k$, as shown in \cite[\S7]{bogus}.  

We get a new perspective
via Theorem \ref{hartshornethm} when there is a smooth proper lift
$Y$ of $Y_0$ over $\Spec(W)$, as follows.
Clearly $\mathscr{D}_Y(Y_0) = \mathscr{O}_Y$ and the ideal $\mathscr{J} = p \mathscr{O}_Y$ 
of $Y_0$ in $\mathscr{O}_Y$ satisfies
$n! \mathscr{J}^{[n]} = n! p^{[n]} \mathscr{O}_Y = (p \mathscr{O}_Y)^n$.  Hence, pd-adic completion for
an $\mathscr{O}_Y$-module is the same thing as $p$-adic completion.   
Applying Theorem \ref{hartshornethm} to $\mathscr{F} = \mathscr{O}_{Y_0/W}$ yields 
$${\rm{R}}\Gamma((Y_0/W)_{\rm{cris}}, \mathscr{O}_{Y_0/W}) \simeq
{\rm{R}}\Gamma(\widehat{Y}, \widehat{\Omega}^*_{Y/W}).$$
Passing to $i$th homologies, we get 
\begin{equation}\label{crisdR}
{\rm{H}}^i((Y_0/W)_{\rm{cris}}, \mathscr{O}_{Y_0/W}) \simeq
\mathbf{H}^i(\widehat{Y}, \widehat{\Omega}^*_{\widehat{Y}/W}).
\end{equation}
The main difference with the approach in \cite[\S7]{bogus} is that
the proof of (\ref{crisdR}) in \cite[7.26.3]{bogus}
involves a comparison isomorphism over each $W/(p^n)$,
whereas the proof via Theorem \ref{hartshornethm} 
involves a direct comparison to cohomology on the formal scheme $\widehat{Y}$.
\end{remark}

\end{document}